%% file: HardyMorrey.tex
\documentclass[12pt,final]{amsart}

\usepackage{amssymb}
\usepackage{amsthm}
\usepackage{amsmath}
\usepackage{graphicx}
\usepackage{subcaption}
\usepackage{fullpage}
\usepackage{color}
\usepackage{enumitem}
\usepackage{tikz}
\usetikzlibrary{patterns}
 \numberwithin{equation}{section}

\usepackage{pgfplots}
\newtheorem{innercustomthm}{Theorem}
\newenvironment{customthm}[1]
  {\renewcommand\theinnercustomthm{#1}\innercustomthm}
  {\endinnercustomthm}
\usepackage[pdftex, hidelinks]{hyperref}
 \usepackage{mathtools}
 \usepackage{placeins}
 \usepackage[nocompress]{cite}

\theoremstyle{plain}
\newtheorem{thm}{Theorem}[section]
\newtheorem{cor}[thm]{Corollary}

\newtheorem{lem}[thm]{Lemma}
\newtheorem{prop}[thm]{Proposition}

\theoremstyle{definition}
\newtheorem{defn}[thm]{Definition}
\newtheorem{ex}[thm]{Example}

\theoremstyle{remark}

\newtheorem{rem}[thm]{Remark}

\newcommand{\N}{\mathbb{N}}
\newcommand{\R}{\mathbb{R}}

\newcommand{\DD}{\mathcal{D}}

\newcommand{\RR}{\mathcal{R}}

\newcommand{\bp}{\begin{proof}[\ensuremath{\mathbf{Proof}}]}
\newcommand{\bs}{\begin{proof}[\ensuremath{\mathbf{Solution}}]}
\newcommand{\ep}{\end{proof}}
\newcommand{\Cee}{{\mathcal{C}}}
\newcommand{\Id}{\mathrm{Id}}
\usepackage{changepage}

\newcommand{\limplus}{{\mathchoice{\vcenter{\hbox{$\scriptstyle +$}}}
  {\vcenter{\hbox{$\scriptstyle +$}}}
  {\vcenter{\hbox{$\scriptscriptstyle +$}}}
  {\vcenter{\hbox{$\scriptscriptstyle +$}}}
}}
\newcommand{\limminus}{{\mathchoice{\vcenter{\hbox{$\scriptstyle -$}}}
  {\vcenter{\hbox{$\scriptstyle -$}}}
  {\vcenter{\hbox{$\scriptscriptstyle -$}}}
  {\vcenter{\hbox{$\scriptscriptstyle -$}}}
}}

\hypersetup{final}

\begin{document}

\title{On a Hardy--Morrey inequality}

\author{Ryan Hynd}
\address{Department of Mathematics, University of Pennsylvania, Philadelphia, PA 19104-6395, USA.}
\email{\href{mailto:rhynd@math.upenn.edu}{rhynd@math.upenn.edu}}

\author{Simon Larson}
\address{Department of Mathematical Sciences, Chalmers University of Technology and the University of Gothenburg, SE-412 96 G\"{o}teborg, Sweden. }
\email{\href{mailto:larsons@chalmers.se}{larsons@chalmers.se}}

\author{Erik Lindgren}
\address{Department of Mathematics, KTH - Royal Institute of Technology, SE-100 44 Stockholm, Sweden. }
\email{\href{mailto:eriklin@kth.se}{eriklin@kth.se}}

\begin{abstract} Morrey's classical inequality implies the H\"older continuity of a function whose gradient is sufficiently integrable. Another consequence is the Hardy-type inequality
$$
\lambda\biggl\|\frac{u}{d_\Omega^{1-n/p}}\biggr\|_{\infty}^p\le \int_\Omega |Du|^p \,dx 
$$
for any open set $\Omega\subsetneq \R^n$. This inequality is valid for functions supported in $\Omega$ and with $\lambda$ a positive constant independent of $u$. The crucial hypothesis is that the exponent $p$ exceeds the dimension $n$. This paper aims to develop a basic theory for this inequality and the associated variational problem. In particular, we study the relationship between the geometry of $\Omega$, sharp constants, and the existence of a nontrivial $u$ which saturates the inequality.
\end{abstract}

\maketitle 

\setcounter{tocdepth}{1}

 \tableofcontents

\section{Introduction and main results}
The topic of this paper concerns a geometric Hardy inequality in the setting of a Sobolev space for which
the associated exponent $p$ is larger than the dimension of the ambient space. Specifically, the inequality states that if $\Omega$ is a proper open subset of $\R^n$ and $p>n$, there exists a constant $\lambda>0$ such that 
\begin{equation}\label{eq: inequality intro}
\lambda\biggl\|\frac{u}{d_\Omega^{1-n/p}}\biggr\|_{\infty}^p\le \int_\Omega |Du|^p \,dx 
\end{equation}
for all $u \in C_c^\infty(\Omega)$. Here and in what follows $d_\Omega(x)$ denotes the distance from $x$ to the complement of $\Omega$. The inequality extends to $u$ being an element of the Sobolev space $\DD^{1,p}_0(\Omega)$ as discussed in Section~\ref{sec:Preliminaries} below. 

\par The one-dimensional case of this inequality previously appeared in~\cite{OK90} and the general case in \cite{psaradakis2013hardy}. This inequality has also recently been considered in~\cite{BrascoPrinariZagati_23} where it occurs as an endpoint case of a family of inequalities interpolating between Sobolev, Morrey, and Hardy inequalities. Nevertheless, we will explain below that the existence of a constant $\lambda$ such that~\eqref{eq: inequality intro} holds is a direct consequence of Morrey's classical inequality (see also \cite{psaradakis2013hardy}). As a result, it is natural to refer to~\eqref{eq: inequality intro} as a {\it Hardy--Morrey} inequality. We also acknowledge that this terminology has been used to describe related inequalities in~\cite{MR2984139,psaradakis2013hardy,MR3527624}.

\par In this note, we turn our attention to the variational problem associated to~\eqref{eq: inequality intro}. Namely, we define
\begin{equation*}
  \RR_p(\Omega, u) = \biggl\|\frac{u}{d_\Omega^{1-n/p}}\biggr\|^{-p}_{\infty}\|Du\|^p_{p} \quad \mbox{for }u \in \DD_0^{1,p}(\Omega)
\end{equation*}
and observe that the sharp constant $\lambda$ in~\eqref{eq: inequality intro} can be characterized as
\begin{equation}\label{eq: variational prob intro}
    \lambda_p(\Omega) = \inf_{u \in \DD_0^{1,p}(\Omega)\setminus\{0\}} \RR_p(\Omega, u)\,.
\end{equation}
Here $\|\cdot\|_p$ denotes the $L^p(\Omega)$ norm. Whenever the infimum~\eqref{eq: variational prob intro} is attained by a nontrivial $u \in \DD_0^{1,p}(\Omega)$, we say that $u$ is an {\it extremal} and that $\Omega$ has an extremal.

 The questions addressed in this paper are:
\begin{enumerate}
  \item How does $\lambda_p(\Omega)$ depend on the geometry of $\Omega$?
  \item When does $\Omega$ have an extremal? 
\end{enumerate}
As it turns out, both of these questions are subtle. To some extent, we shall see that this subtlety can be traced back to the fact that $\lambda_p(\Omega)$ is invariant under orthogonal transformations, translations, and dilations of $\Omega$.

\subsection{Main results}
Our first result provides sharp upper and lower bounds for $\lambda_p(\Omega)$. The upper bound involves the halfspace $\R^n_\limplus= \{x\in \R^n: x_n>0\}$. 
\begin{customthm}{A}\label{thm: universal sharp bounds}
Assume $p>n\geq 1$. If $\Omega \subsetneq \R^n$ is open, then
  \begin{equation*}
     C_{n,p}^{-p}=\lambda_p(\R^n \setminus\{0\})\leq \lambda_p(\Omega) \leq \lambda_p(\R^n_\limplus)=2^{n-1}C_{n,p}^{-p}\,.
  \end{equation*}
Here $C_{n,p}$ is the sharp constant in Morrey's inequality (see~\eqref{eq:morreyconstant}). 
\end{customthm}

A natural question to ask is whether equality is attained in the bounds of the theorem only when $\Omega = \R^n_\limplus$ or $\Omega = \R^n \setminus \{0\}$ up to the natural symmetries of the problem (see Section~\ref{sec:Preliminaries}). It turns out that this is not the case.
\begin{customthm}{B}\label{thm: convex, punctured, exterior domains}
  Suppose $p>n\geq 2$. 
  \begin{enumerate}
    \item If $\Omega \subsetneq \R^n$ is convex and open, then
  \begin{equation*}
    \lambda_p(\Omega) = \lambda_p(\R^n_\limplus)
  \end{equation*}
  and $\Omega$ has an extremal if and only if $\Omega$ is a halfspace.

  \item If $\Omega \subset \R^n$ is open and $x_0 \in \Omega$, then
  \begin{equation*}
    \lambda_p(\Omega \setminus\{x_0\}) = \lambda_p(\R^n\setminus\{0\})
  \end{equation*}
  and $\Omega \setminus\{x_0\}$ has an extremal if and only if $\Omega=\R^n$.

  \item If $K \subset \R^n$ is compact, then
  \begin{equation*}
    \lambda_p(\R^n \setminus K) = \lambda_p(\R^n\setminus\{0\})
  \end{equation*}
  and $\R^n \setminus K$ has an extremal if and only if $K$ is a singleton.
\end{enumerate}
\end{customthm}
\noindent According to this theorem, the infimum~\eqref{eq: variational prob intro} is attained for a halfspace and a punctured whole space. More generally, we will show that this holds whenever $\Omega^c$ is a closed convex cone. See Section~\ref{sec: Dilation invariant domains}.

In view of our remarks above, one might suspect that extremals exist only in rather special geometries. However, the following two theorems assert that this is far from the case. As will be elaborated on later, the first result is a consequence of a more general compactness threshold-result in the spirit of the work of Brezis and Nirenberg~\cite{MR0709644} (see Proposition~\ref{prop: Compactness threshold}). In what follows, we will simply say that $\Omega$ is $C^k$ if $\partial \Omega$ is $C^{k}$-regular.

\begin{customthm}{C}\label{thm: strict energy ineq implies existence}
  Fix $p>n\geq 2$. If $\Omega \subset \R^n$ is bounded, open, and $C^1$ with
  \begin{equation}\label{eq: compactness threshold inequality} 
    \lambda_p(\Omega) < \lambda_p(\R^n_\limplus)\,,
  \end{equation}
  then $\Omega$ has an extremal. 
\end{customthm}

\noindent We will verify the above claim by establishing that any minimizing sequence $\{u_k\}_{k\geq 1}\subset \DD_0^{1,p}(\Omega)$ for~\eqref{eq: variational prob intro} with $\|Du_k\|_p=1$ for all $k$ is precompact in $\DD_0^{1,p}(\Omega)$. In particular, we will show that $\{u_k\}_{k\geq 1}$ has a subsequence which converges to an extremal. 

\par We will say that a $C^2$ subset $\Omega \subset \R^n$ is {\it mean convex} provided that the mean curvature at each point of $\partial \Omega$ is nonnegative. We will always measure mean curvature with respect to the outward unit normal so that convex shapes have nonnegative mean curvature. The central assertion of our work is that a bounded $\Omega$ which is not mean convex admits an extremal.
\begin{customthm}{D}\label{thm: neg mean curvature}
 Let $p>n\geq 2$. If $\Omega \subset \R^n$ is bounded, open, $C^{2}$, and not mean convex, then~\eqref{eq: compactness threshold inequality} holds. Therefore,  $\Omega$ has an extremal. 
 \end{customthm}

In addition to the above theorems, we provide various examples where our results give detailed knowledge about $\lambda_p(\Omega)$; see Section~\ref{sec: Dilation invariant domains} and~\ref{sec: Examples and open problems}. These examples encompass for instance concave cones, polygons and piecewise $C^1$ sets, epigraphs and examples that indicate the instability $\lambda_p(\Omega)$ with respect to small changes in $\Omega$. It is also worth noting that we prove that any $\lambda\in[\lambda_p(\R^n\setminus\{0\}),\lambda_p(\R^n_\limplus)]$ is realized as $\lambda=\lambda_p(\Omega)$ for some open $\Omega\subsetneq\R^n$; refer to Theorem~\ref{thm: range of lambda} below. The last section of this article also includes a short list of open problems.

\subsection{Related results}
Inequality~\eqref{eq: inequality intro} can be seen as the limiting case as $q\to\infty$ of the family of Hardy-type inequalities
\begin{equation}
\label{eq: pqhardy}
\lambda_{p,q}(\Omega)\biggl\|\frac{u}{\;d_\Omega^{\,\gamma}\;}\biggr\|_{q}^p \le   \int_\Omega |Du|^p\,dx\, ,
\end{equation}
where $\gamma=n/q+1-n/p$ and $n<p\le q<\infty$. The special case when $p=q$
\begin{equation}
\label{eq: phardy}
 \lambda_{p,p}(\Omega)\biggl\|\frac{u}{\;d_\Omega\;}\biggr\|_{p}^p\le  \int_\Omega |Du|^p\,dx
\end{equation}
has been the topic of many studies. Making use of H\"older's inequality, it is straightforward to conclude that~\eqref{eq: pqhardy} follows from~\eqref{eq: inequality intro} and~\eqref{eq: phardy}, which can be seen as the endpoints of this family of inequalities. This has recently been observed in~\cite{BrascoPrinariZagati_23}. However, the theory for~\eqref{eq: inequality intro} and~\eqref{eq: pqhardy} in general is still in its infancy. 

As previously mentioned, the one-dimensional version of~\eqref{eq: inequality intro} appears in Chapter~1.5 of~\cite{OK90} and~\eqref{eq: pqhardy} is treated in Chapters 2 and 3. The $n$-dimensional version of~\eqref{eq: inequality intro} as well as~\eqref{eq: pqhardy} is mentioned in~\cite{Kaj}; see page 22 in Paper A. Other versions of inequality~\eqref{eq: inequality intro} also appear in Section 2.1.6 in~\cite{Mazya} and Section 2.1.0 in~\cite{Str84}.

Hardy's inequality~\eqref{eq: phardy} was first proved in the one dimensional setting by Hardy (cf.~\cite{hardy} and~\cite{hardy2}) even though a special case was perhaps known as early as 1907 (see~\cite{Boggio}). For an overview of Hardy's inequality~\eqref{eq: phardy} and its rich history, we refer the reader to~\cite{KMP,OK90,MR3408787,MR3052352}. The validity of such an inequality for $p\leq n$ is a rather delicate matter. However, in the case $p>n$ that we are concerned with,~\eqref{eq: phardy} holds for any open set $\Omega$ (see~\cite{Anc81},~\cite{Lewis},~\cite{MR2885951}, and~\cite{MR1010807}). Since~\eqref{eq: inequality intro} is also valid for general open sets,  H\"older's inequality implies that the same is true for~\eqref{eq: pqhardy}.

There is also a well established connection between the geometry of $\Omega$, the optimal constant, and the existence of extremals for inequality~\eqref{eq: phardy}. These results are very much in the spirit of what we accomplish in Theorems~\ref{thm: convex, punctured, exterior domains},~\ref{thm: strict energy ineq implies existence} and~\ref{thm: neg mean curvature}. For instance, it has been proved that $\lambda_{p,p}(\Omega)\leq c_p$ with 
$$
c_p=\biggl(1-\frac{1}{p}\biggr)^\frac1p
$$
for bounded and sufficiently regular $\Omega$ and $\lambda_{p,p}(\Omega)=c_p$ for convex $\Omega$. Moreover, for sufficiently smooth bounded sets, $\lambda_{p,p}(\Omega)= c_p$ if and only if there is no extremal (see~\cite{MR2885951},~\cite{MR1458330},~\cite{MR1817710}, and~\cite{MR4012806}). We also remark that a parallel theory for Hardy's inequality~\eqref{eq: phardy} on exterior domains has been established (as described in~\cite{MR4012806},~\cite{MR2196033}, and~\cite{MR1868901}).

For $1<p\leq n$ and $q \in [2, \frac{2n}{n-2}]$ the inequality in~\eqref{eq: pqhardy} was recently considered in~\cite{WangZhu24}, and its validity established under the assumption that $\Omega$ has Lipschitz-regular boundary. For the case $p=2, n\geq 2$ and $q\in (2, \frac{2n}{n-2})$, results similar to those proved in this paper concerning the attainability of the sharp constant in \eqref{eq: pqhardy} were obtained in~\cite{WangZhu24,SunWang24}.

Another result in the spirit of Theorem~\ref{thm: neg mean curvature} was discovered by Ghoussoub and Robert. They considered the Hardy--Sobolev inequality 
\begin{equation}\label{HardySobolev}
\mu_s(\Omega)\biggl(\int_\Omega\frac{|u|^{2^*}}{|x|^{s}}\,dx\biggr)^{2/2^*}\le \int_{\Omega}|Du|^2\,dx
\end{equation}
for a smooth and bounded domain $\Omega$ with $0\in \partial \Omega$. Here $n\ge 3$, $s\in (0,2)$, 
$$
2^*=\frac{2(n-s)}{n-2}\,,
$$
and the admissible functions $u$ belong to the Sobolev space $H^1_0(\Omega)$. 

Ghoussoub and Robert showed that if $\partial \Omega$ has negative mean curvature at $0$, then inequality~\eqref{HardySobolev} has an extremal~\cite{MR2276538}. That is, equality holds in~\eqref{HardySobolev} for a nontrivial $u\in H^1_0(\Omega)$. This extended earlier work by Egnell~\cite{MR1034662} who showed that extremals of~\eqref{HardySobolev} exist for certain conical domains and by Ghoussoub and Kang~\cite{MR2097030} who verified existence when $\partial \Omega$ is negatively curved at $0$. For other results along these lines, see~\cite{MR3472850,MR4599209,MR4169674,MR3902383,MR4150374}.

\subsection{Outline of the strategy}

Our approach rests heavily upon the fundamental property that $\lambda_p$ is invariant under translations, rotations and dilations (cf.\ Section~\ref{sec: siminv}). We follow two main strategies which we now briefly explain. 

The first strategy consist of obtaining information about the value of $\lambda_p$ and possible extremals by transplanting a competitor defined in $\Omega$ into the corresponding minimization problem for a different set $\Omega'$. In certain situations, this will allow us to deduce interesting bounds for $\lambda_p$. In particular, we obtain bounds by finding an appropriate exhaustion of a set $\Omega$ and using Lemma~\ref{lem: interior exhaustion}, or by touching $\Omega$ from outside with a cleverly chosen larger set $\Omega'$ and using Proposition~\ref{prop: Supporting sets}. These ideas lead up to the proof of Theorems~\ref{thm: universal sharp bounds} and~\ref{thm: convex, punctured, exterior domains} in Section~\ref{sec: convex, punctured, exterior domains}.

The second line of our analysis, which is the more technical, consists of studying the nature of sequences of trial functions which in a certain sense either concentrate at a boundary point or move of towards infinity. In order to explain this idea, it is convenient to first state a few properties of minimizers and introduce some notation. A first important observation is that any extremal $u$ satisfies 
$$
-\Delta_p u = 0\quad  \mbox{in } \Omega \setminus \{x_0\}
$$
for some $x_0\in \Omega$; see Proposition~\ref{prop: equation for extremals}. This leads to the useful idea (cf.\ Proposition~\ref{prop: potentials are enough}) that it is enough to study solutions of this PDE for a given $x_0\in \Omega$ which satisfy $u(x_0)=1$ and $u|_{\partial\Omega}=0$ when searching for an extremal. We call such functions potentials.

The above observations lead to Proposition~\ref{prop: Compactness threshold}, where we deduce that the following statements are equivalent:
\begin{enumerate}
 \item $\Omega$ has an extremal.
  \item There exists a sequence $\{x_k\}_{k\geq 1}\subset \Omega$ such that $\lim_{k\to \infty} x_k \in \Omega$ and the corresponding sequence of potentials is a minimizing sequence. 
\end{enumerate}

It therefore becomes crucial to understand when a sequence of points $x_k$ related to a minimizing sequence of potentials stays inside $\Omega$, approaches the boundary or escapes to infinity. The background for this is developed in Section~\ref{sec: extremals} and these questions are then pursued in detail in Sections~\ref{sec: Blow-up analysis} and~\ref{sec: Refined blow-up neg mean curvature}. 

In Section~\ref{sec: Blow-up analysis}, we are able to obtain estimates in terms of $\lambda_p$ for certain global prototype sets that locally or globally approximate $\Omega$. The main outcome of this analysis is Theorem~\ref{thm: strict energy ineq implies existence}. The local analysis amounts to performing a blow-up. Despite the substantial difference in how this approach is carried out, the underlying idea is similar to ideas used also in the context of Hardy's inequality~\eqref{eq: phardy}. See for instance Brezis--Marcus~\cite{MR1655516} and Marcus--Mizel--Pinchover~\cite{MR1458330} for more. 

In Section~\ref{sec: Refined blow-up neg mean curvature} we perform an analysis in the spirit of a domain variation. This is done to prove that if the boundary has a point of negative mean curvature, the sequence of points $x_k$ corresponding to a minimizing sequence of potentials cannot approach the boundary. This is one of the main ideas used to prove Theorem~\ref{thm: neg mean curvature}.

\subsection*{Acknowledgements}
 R.~H.\ was supported by an American Mathematical Society Claytor--Gilmer Fellowship and NSF grant DMS-2350454. E.~L.\ has been supported by the Swedish Research Council, grants no.~2023-03471, 2017-03736 and 2016-03639. S.~L.\ was supported by the Swedish Research Council grant no.~2023-03985 as well as the Knut and Alice Wallenberg Foundation grants KAW~2021.0193 and KAW~2017.0295.

 Part of this material is based upon work supported by the Swedish Research Council under grant no.~2016-06596 while all three authors were participating in the research program ``Geometric Aspects of Nonlinear Partial Differential Equations'', at Institut Mittag-Leffler in Djursholm, Sweden, during the fall of 2022.


\section{Preliminaries}
\label{sec:Preliminaries}

Throughout the paper, we will assume $n\in \N$ and $p>n$. We will denote by $B_r(x)$ the ball of radius $r$ centered at $x$. In the case when $x=0$, we will simply write $B_r$. Unless otherwise stated, $\Omega$ will always be a proper, nonempty, open subset of $\R^n$, and 
$$
d_\Omega(x)=\inf_{y\in \Omega^c}|x-y|\,, \quad x\in \R^n\,.
$$
Note that for $x \in \Omega$, $d_\Omega$ is the distance to the boundary of $\partial\Omega$.

\par Our work below will concern functions in the homogeneous Sobolev space  
$$
\DD^{1,p}(\R^n)=\{u\in L^1_{\text{loc}}(\R^n): u_{x_1},\dots, u_{x_n}\in L^p(\R^n)\}
$$
equipped with the seminorm $u \mapsto \|Du\|_p$.
As usual,  $u_{x_i}$ denotes the weak partial derivative with respect to $x_i$. When $p>n$ any $u\in \DD^{1,p}(\R^n)$ has a H\"older continuous representative $u^*$, and we will identify $u$ with $u^*$ going forward. 

In what follows, we shall mostly consider functions $u\in \DD^{1,p}(\R^n)$ which vanish on the complement of an open set $\Omega \subsetneq\R^n$. We will denote this space of functions by
$$
\DD_0^{1,p}(\Omega)=\{u\in \DD^{1,p}(\R^n): \text{$u(x)=0$ for each $x\not\in \Omega$}\}\,.
$$
By Morrey's inequality~\eqref{MorreyInequality}, $u \mapsto \|Du\|_p$ is a norm on the restricted space $\DD^{1,p}_0(\Omega)$. Moreover, the space $\DD_0^{1,p}(\Omega)$ is a Banach space which can be identified as the completion of $C_c^\infty(\Omega)$ (see Lemma~\ref{lem: approx with compact support}).

\par We will make use of {\it Morrey's inequality}, which asserts that there is $C>0$ depending on $n, p$ such that
\begin{equation}\label{MorreyInequality}
\sup_{x\neq y}\frac{|u(x)-u(y)|}{|x-y|^{1-n/p}}\le C\biggl(\int_{\R^n}|Du|^p\,dz\biggr)^{1/p}
\end{equation}
for each $u\in \DD^{1,p}(\R^n)$. Morrey's inequality is a consequence of {\it Morrey's estimate} that posits that there is another constant $c=c(n,p)$ such that if $B$ is a ball of radius $r$ and $x, y \in B$, then
\begin{equation}\label{MorreyEstimate}
|u(x)-u(y)|\le c r^{1-n/p}\biggl(\int_{B}|Du|^p\,dz\biggr)^{1/p}
\end{equation}
(see~\cite[Theorem 4.10]{MR3409135}). We note that~\eqref{MorreyInequality} holds with $C=c$. Let us denote by
\begin{equation}
\label{eq:morreyconstant}
\text{$C_{n,p}$, the smallest $C>0$ for which Morrey's inequality holds.}
\end{equation}

\par We will now show that for any $\Omega \subsetneq \R^n$ inequality~\eqref{eq: inequality intro} holds with some constant $\lambda$. In what follows, the quotient $u/d_\Omega^{1-n/p}$ should be interpreted as zero in the complement of $\Omega$ for $u\in  \DD_0^{1,p}(\Omega)$.
\begin{prop}\label{HMprop}
For any $u\in  \DD_0^{1,p}(\Omega)$,
\begin{equation*}
C_{n,p}^{-p}\biggl\| \frac{u}{d_\Omega^{1-n/p}}\biggr\|^p_{\infty}\le \int_{\Omega}|Du|^p\,dx\,.
\end{equation*}

\end{prop}

\begin{proof}
Let $u\in \DD_0^{1,p}(\Omega)\subset \DD^{1,p}(\R^n)$. By Morrey's inequality 
$$
\frac{|u(x)-u(y)|}{|x-y|^{1-n/p}}\le C_{n,p}\biggl(\int_{\Omega}|Du|^p\,dz\biggr)^{1/p}
$$
for distinct $x, y\in \R^n$. In particular, if we choose $x \in \Omega$ and $y\in \partial\Omega$ such that $d_\Omega(x)=|x-y|$, then 
$$
\frac{|u(x)|}{d_\Omega(x)^{1-n/p}}\le C_{n,p}\biggl(\int_{\Omega}|Du|^p\,dz\biggr)^{1/p}\,.
$$
We conclude by taking the supremum over $x\in \Omega$. 
\end{proof}

\par By the previous proposition, $\RR(\Omega,u)\ge C_{n,p}^{-p}$ for any  $u\in \DD_0^{1,p}(\Omega)$ with $u\not\equiv 0$. Since $\lambda_p(\Omega)$ is defined in~\eqref{eq: variational prob intro}  as the infimum of all such quotients,
\begin{equation}\label{universal lambda lower bound} 
\lambda_p(\Omega)\ge C_{n,p}^{-p}\,.
\end{equation}
In particular, $\lambda_p(\Omega)$ is positive and the above lower bound is independent of $\Omega$. 

Next we recall some technical results concerning functions in $\DD^{1,p}_0(\Omega)$.

\begin{lem}\label{TechLimLem}
Suppose $u \in \DD^{1,p}_0(\Omega)$.
\begin{enumerate}[label=(\textit{\roman*})]

\item\label{limituoverdist} If $y\in \partial \Omega$, then  
\begin{equation*}
\lim_{\substack {x\rightarrow y\\ x\in\Omega}}\frac{u(x)}{d_\Omega(x)^{1-n/p}}=0\,.
\end{equation*}

\item\label{limitatinfinity} If $\Omega$ is unbounded, 
\begin{equation*}
\lim_{\substack {|x|\rightarrow \infty\\ x\in\Omega}}\frac{u(x)}{d_\Omega(x)^{1-n/p}}=0\,.
\end{equation*}

\item\label{maxattained} The function $|u|/d_\Omega^{1-n/p}$ achieves its supremum within $\Omega$. 
\end{enumerate}
\end{lem}
\begin{proof}
In this proof, we will exploit the following limits, which were established in Section~6 of~\cite{MR4275749}. For any $v\in \DD^{1,p}(\R^n)$,
\begin{equation}\label{veeLimOne}
\lim_{|x-y|\rightarrow0}\frac{|v(x)-v(y)|}{|x-y|^{1-n/p}}=0
\end{equation}
and
\begin{equation}\label{veeLimTwo}
\lim_{|x|+|y|\rightarrow\infty}\frac{|v(x)-v(y)|}{|x-y|^{1-n/p}}=0\,.
\end{equation}
Claim~\ref{limituoverdist}  follows from~\eqref{veeLimOne} and~\ref{limitatinfinity} follows from~\eqref{veeLimTwo}. As for~\ref{maxattained}, we may select a maximizing sequence $\{x_k\}_{k\geq 1}\subset \Omega$ for $|u|/d_\Omega^{1-n/p}$. If $\{x_k\}_{k\geq 1}$ has a cluster point at the boundary of $\Omega$ or if $\{x_k\}_{k\geq 1}$ is unbounded, then $|u|/d_\Omega^{1-n/p}$ vanishes identically by~\ref{limituoverdist} and~\ref{limitatinfinity}. Otherwise, since $x \mapsto |u(x)|/d_\Omega(x)^{1-n/p}$ is continuous in $\Omega$ its maximum is attained at cluster points of $\{x_k\}_{k\geq 1}$ in $\Omega$.
\end{proof}

\subsection{Approximation results}

In this section, we recall some facts about approximation of functions in $\DD^{1,p}_0(\Omega)$ by smooth functions and list some consequences. The first result is the following lemma, which will be important in many of our constructions.
\begin{lem}\label{lem: approx with compact support}
  Suppose  $u \in \DD_0^{1,p}(\Omega)$ and $\epsilon>0$. There exists $v\in C_c^\infty(\Omega)$ such that
  \begin{equation*}
    \|Du-Dv\|_p \leq \epsilon \|Du\|_p\,.
  \end{equation*}
\end{lem}
Note that we did not assume any regularity or boundedness of the set $\Omega$ in the above lemma. This is a feature which is special to the supercritical setting $p>n$. As the proof is standard, yet somewhat lengthy, it is deferred to Appendix~\ref{sec: App A Approximiation}.

An important consequence of Lemma~\ref{lem: approx with compact support} is that in the infimum~\eqref{eq: variational prob intro} defining $\lambda_p(\Omega)$ the space of test functions $\DD^{1,p}_0(\Omega)$ can be exchanged with $C_c^\infty(\Omega)$. 

\begin{lem}\label{lem: variational def cpt support}
The infimum~\eqref{eq: variational prob intro} is also given by
  \begin{equation*}
    \lambda_p(\Omega) = \inf_{u \in C_c^\infty(\Omega)\setminus\{0\}}\RR_p(\Omega, u)\,.
  \end{equation*}
\end{lem}

\begin{proof}
Let $\epsilon>0$. There exists a $v\in \DD_0^{1,p}(\Omega)$ such that
  \begin{equation*}
    \RR_p(\Omega, v)\leq \lambda_p(\Omega)+ \epsilon\,.
  \end{equation*}
  According to Lemma~\ref{TechLimLem}, there is $x_0\in \Omega$ so that
  \begin{equation*}
    \frac{|v(x_0)|}{d_\Omega(x_0)^{1-n/p}} = \biggl\|\frac{v}{d_\Omega^{1-n/p}}\biggr\|_\infty\,.
  \end{equation*}
  In view of Lemma~\ref{lem: approx with compact support}, we may select $u \in C_c^\infty(\Omega)$ such that $\|Du-Dv\|_p\le \epsilon'\|Dv\|_p$ for every $\epsilon'>0$. Note $v(x_0)=u(x_0)+O(\epsilon')$
by the Hardy--Morrey inequality.

\par As a result, 
  \begin{align*}
    \RR_p(\Omega, u) &= \biggl\|\frac{u}{d_\Omega^{1-n/p}}\biggr\|_\infty^{-p}\|Du\|_p^p\\
    & \leq \frac{d_\Omega(x_0)^{p-n}\|Du\|_p^p}{|u(x_0)|^p} \\
    &\leq \frac{d_\Omega(x_0)^{n-p}\|Dv\|_p^p}{|v(x_0)|^p}(1+o_{\epsilon'\to 0}(1))\\
    &=\RR_p(\Omega, v)(1+o_{\epsilon'\to 0}(1)) \\
    &\leq \lambda_p(\Omega) + o_{\epsilon'\to 0}(1) + \epsilon \,.
  \end{align*}
Therefore, 
  \begin{equation*}
    \lambda_p(\Omega) \leq \inf_{u \in C_c^\infty(\Omega)}\RR_p(\Omega, u) \leq \lambda_p(\Omega) + o_{\epsilon'\to 0}(1) + \epsilon \,.
  \end{equation*}
We conclude after sending $\epsilon'$ and then $\epsilon$ to zero. 
\end{proof}
We will say that $\Omega$ is {\it exhausted} by $\{\Omega_j\}_{j\geq1}$ whenever every $\Omega_j\subset \R^n$ is open, $\Omega_j \subset \Omega_{j+1}$ for each $j\ge 1$, and $\Omega = \bigcup_{j\geq 1}\Omega_j.$  It turns out that $\lambda_p$ is upper semicontinuous with respect to exhaustions. 
\begin{lem}\label{lem: interior exhaustion}
 If $\Omega$ is exhausted by $\{\Omega_j\}_{j\geq1}$, then
  \begin{equation*}
    \limsup_{j \to \infty}\lambda_p(\Omega_j) \leq \lambda_p(\Omega)\,.
  \end{equation*}
\end{lem}
\begin{proof}
  Fix $\epsilon>0$. By Lemma~\ref{lem: variational def cpt support}, there exists $u \in C_c^\infty(\Omega)$ such that
  \begin{equation*}
     \RR_p(\Omega, u) \leq \lambda_p(\Omega)+ \epsilon\,.
   \end{equation*} 
   Since $u$ supported in a compact set of $\Omega$ and $\{\Omega_j\}_{j\geq 1}$ is an open cover of $\Omega$, $\mathrm{supp}(u)\subset \Omega_j$ for all sufficiently large $j$. Moreover, $\Omega_j \subset \Omega$ implies that $d_{\Omega_j}(x) \leq d_\Omega(x)$ for all $x \in \Omega_j$. Hence,
   \begin{equation*}
      \lambda_p(\Omega_j) \leq \RR_p(\Omega_j, u)
      \leq\RR_p(\Omega, u) \leq \lambda_p(\Omega) + \epsilon\,
    \end{equation*} 
    for all large enough $j$. Since $\epsilon$ was arbitrary, this proves the lemma.
  \end{proof}

\subsection{Similarity invariance}\label{sec: siminv}
For $Q\in O(n)$, $r>0$ and $y\in \R^n$, we define the similarity transform
\begin{equation*}
T_{r,Q,y} \colon \R^n \to \R^n\,;\; x \mapsto rQx+y\,.
\end{equation*}
For a set $U \subset \R^n$, we write 
$$
T_{r,Q,y}U =rQU+y=\{rQx+y: x\in U\}
$$
to denote the image of $U$ under the similarity transform. Note that $$T_{r, Q, y}^{-1} = T_{1/r, Q^{-1}, -Q^{-1}y/r}\,.$$ If the parameters $r, Q, y$ are understood, we may write simply $T$.

Since $T_{r,Q,y}$ is a diffeomorphism of $\R^n$, it must be that
$$
\partial(T_{r,Q,y}\Omega)=T_{r,Q,y}\partial\Omega\,.
$$
It follows that
\begin{equation}\label{distScaling}
d_{T_{r,Q,y}\Omega}(T_{r,Q,y}x)=rd_\Omega(x)\,,
\end{equation}
which will be useful below. 

\par An all important property of the best constant $\lambda_p(\Omega)$ is that it is invariant under similarity transformations. 
We will sometimes refer to this as {\it similarity invariance}.

\begin{lem}\label{ScalingInvarianceLem}
If $T\colon \R^n \to \R^n$ is a similarity transform and $\Omega \subsetneq \R^n$ is open, then
$$
\lambda_p(T\Omega)=\lambda_p(\Omega)
$$
and
$$
\RR_p(\Omega, u) = \RR_p(T\Omega, u\circ T^{-1})\quad \mbox{for all }u \in \DD_0^{1,p}(\Omega)\,.
$$
\end{lem}
\begin{proof}
Let $T=T_{r,Q,y}$. Suppose $u \in \DD^{1,p}_0(\Omega)$ and set  $v= u\circ T^{-1}$. Then $v$ vanishes in $T\Omega^c$ and by a change of variables
$$
\int_{T\Omega}|Dv(z)|^p\,dz=r^n\int_{\Omega}|Dv(Tx)|^p\,dx=r^{n-p}\int_{\Omega}|Du(x)|^p\,dx\,.
$$
Consequently, $v \in \DD_0^{1,p}(T\Omega)$.
Moreover, by~\eqref{distScaling},
$$
\sup_{z\in T\Omega}\frac{|v(z)|^p}{d_{T\Omega}(z)^{p-n}}=\sup_{x\in \Omega}\frac{|v(Tx)|^p}{d_{T\Omega}(Tx)^{p-n}}=
\sup_{x\in \Omega}\frac{|u(x)|^p}{d_{\Omega}(x)^{p-n}}r^{n-p}\,.
$$

\par Therefore, 
\begin{equation*}
\RR_p(T\Omega, v)= 
\RR_p(\Omega, u)\,.
\end{equation*}
This proves the second claim.

For any $\epsilon>0$, there exists $u \in \DD^{1,p}_0(\Omega)$ such that
$$
\lambda_p(\Omega)\ge \RR_p(\Omega, u)-\epsilon\,. 
$$
By the above, the function $u\circ T^{-1}\in \DD^{1,p}_0(T\Omega)$ satisfies $\RR_p(\Omega, u) = \RR_p(T\Omega, u\circ T^{-1})$ and thus
\begin{equation*}
\lambda_p(\Omega)\ge \RR_p(T\Omega, u\circ T^{-1})-\epsilon\ge \lambda_p(T\Omega)-\epsilon\,. 
\end{equation*}
Since $\epsilon$ was arbitrary it follows that $\lambda_p(\Omega)\ge \lambda_p(T\Omega)$. Switching the roles of $\Omega$ and $T\Omega$ in this argument gives the reverse inequality and completes the proof. 
\end{proof}


\section{Extremals and potentials}\label{sec: Extremals and potentials}\label{sec: extremals}

In this section, we focus on properties satisfied by extremals and more generally to properties of potentials. We recall that a function $u\in \DD^{1,p}_0(\Omega)$ with $u\not\equiv 0$ is an {\it extremal} provided that $ \lambda_p(\Omega)=\RR(\Omega,u)$. That is, 
$$
 \lambda_p(\Omega)\biggl\| \frac{u}{d_\Omega^{1-n/p}}\biggr\|^p_\infty= \int_{\Omega}|Du|^p\,dx\,.
$$
Below, $\delta_{x_0}$ is the Dirac delta distribution at $x_0$.

\begin{prop}\label{prop: equation for extremals}
Let $\Omega \subsetneq \R^n$ be an open set. A function $u \in \DD^{1,p}_0(\Omega)\setminus\{0\}$ is an extremal if and only if there is $x_0\in \Omega$ for which $u$ is a weak solution of 
\begin{equation}\label{ExtremalPDE}
\begin{cases}
-\Delta_pu=\displaystyle\lambda_p(\Omega) \frac{|u(x_0)|^{p-2}u(x_0)}{d_\Omega(x_0)^{p-n}}\delta_{x_0}\quad & \mbox{in } \Omega\,,\\
\hspace{24pt}u=0 \quad & \mbox{on } \partial\Omega\,.
\end{cases}
\end{equation}
\end{prop}

\begin{proof} Suppose $u$ is a weak solution of~\eqref{ExtremalPDE}. Then 
\begin{equation*}
\int_{\Omega}|Du|^{p-2}Du\cdot Dv\,dx=\lambda_p(\Omega) \frac{|u(x_0)|^{p-2}u(x_0)}{d_\Omega(x_0)^{p-n}}v(x_0)
\end{equation*}
for each $v\in \DD^{1,p}_0(\Omega)$. Choosing $v=u$ gives 
$$
\int_{\Omega}|Du|^{p}\,dx=\lambda_p(\Omega) \frac{|u(x_0)|^{p}}{d_\Omega(x_0)^{p-n}}\le \lambda_p(\Omega)\biggl\| \frac{u}{d_\Omega^{1-n/p}}\biggr\|^p_\infty.
$$
We conclude that $u$ is an extremal.

\medskip

Assume that $u\in \DD_0^{1,p}(\Omega)$ is an extremal. By part~\ref{maxattained} of Lemma~\ref{TechLimLem}, there is $x_0\in \Omega$ with 
$$
\biggl\| \frac{u}{d_\Omega^{1-n/p}}\biggr\|_\infty=\frac{|u(x_0)|}{d_\Omega(x_0)^{1-n/p}}\,.
$$
For $t>0$ and $v\in \DD^{1,p}_0(\Omega)$, 
\begin{align*}
\int_\Omega\biggl(\frac{|Du+tDv|^p-|Du|^p}{pt}\biggr)\,dx&=\frac{1}{pt}\int_{\Omega}|Du+tDv|^p\,dx -\frac{1}{pt}\int_{\Omega}|Du|^p\,dx\\
&\ge \frac{\lambda_p(\Omega)}{pt}\biggl\| \frac{u+tv}{d_\Omega^{1-n/p}}\biggl\|^p_\infty-\frac{\lambda_p(\Omega)}{pt}\biggl\| \frac{u}{d_\Omega^{1-n/p}}\biggr\|^p_\infty\\
&\ge \frac{\lambda_p(\Omega)}{pt}\frac{|u(x_0)+tv(x_0)|^p}{d_\Omega(x_0)^{p-n}}-\frac{\lambda_p(\Omega)}{pt}\frac{|u(x_0)|^p}{d_\Omega(x_0)^{p-n}}\\
&=\frac{\lambda_p(\Omega)}{d_\Omega(x_0)^{p-n}}\biggl(\frac{|u(x_0)+tv(x_0)|^p-|u(x_0)|^p}{pt}\biggr).
\end{align*}
By routine estimates, we find
$$
\int_{\Omega}|Du|^{p-2}Du\cdot Dv\,dx\ge \lambda_p(\Omega) \frac{|u(x_0)|^{p-2}u(x_0)}{d_\Omega(x_0)^{p-n}}v(x_0)
$$
in the limit as $t\rightarrow 0^\limplus$.
Replacing $v$ by $-v$ gives equality. Thus, $u$ is a weak solution of the boundary value problem~\eqref{ExtremalPDE}. 
\end{proof}

Based on our characterization of extremals, it is natural to consider the following family of functions. Given $y \in \Omega$ we call the unique weak solution $w^\Omega_{y}\in \DD_0^{1,p}(\Omega)$ of the equation
\begin{equation}\label{eq: potential equation}
  \begin{cases}
   - \Delta_p w^\Omega_y =0 & \mbox{in } \Omega \setminus \{y\}\,,\\
      \hspace{.33in} w^\Omega_y =0 & \mbox{on } \partial\Omega\,,\\
     \hspace{.12in}  w^\Omega_y(y)=1\,,
  \end{cases}
\end{equation}
a \emph{potential} in $\Omega$. 
That $w^\Omega_y$ is a weak solution of~\eqref{eq: potential equation} is equivalent to it being a weak solution of the equation
\begin{equation*}
  \begin{cases}
    -\Delta_p w^\Omega_y= \|Dw^\Omega_y\|_p^p \delta_{y} & \mbox{in }\Omega\,,\\
    \hspace{.33in}w^\Omega_y =0 & \mbox{on } \partial\Omega\,.
  \end{cases}
\end{equation*}

In addition, 
$$
\|Dw^\Omega_{y}\|_p\le \|Dv\|_p
$$
among all $v \in \DD^{1,p}_0(\Omega)$ which satisfy $v(y)=1$. Furthermore, $w^\Omega_{y}$ is the unique function in $\DD_0^{1,p}(\Omega)$ with this property. We will refer to this variational characterization of $w^\Omega_{y}$ several times below and call any $v \in \DD^{1,p}_0(\Omega)$ which satisfies $v(y)=1$ a {\it competitor} for $w^\Omega_{y}$.

\begin{cor}\label{ExtSignCor}{}
If $\Omega \subsetneq \R^n$ and $y \in \Omega$, then $w_y^\Omega$ is identically zero in each connected component of $\Omega$ except the one which contains $y$ where $w^\Omega_y$ is everywhere positive. In particular, if $u\in \DD_0^{1,p}(\Omega)$ is an extremal, then $u$ vanishes identically in all but one component of $\Omega$ where it is either everywhere positive or everywhere negative.
\end{cor}
\begin{proof}
If $\Omega'$ is a connected component of $\Omega$ and $y \notin \Omega'$, then $-\Delta_pw^\Omega_y=0$ in $\Omega'$ with $w^\Omega_y|_{\partial\Omega'}=0$, so $w^\Omega_y \equiv 0$ in $\Omega'$.
Let $\Omega_0$ be the connected component of $\Omega$ containing $y$. Then $w_y^\Omega(y)=1$, $-\Delta_pw^\Omega_y\ge 0$ in $\Omega_0$ and $w^\Omega_y|_{\partial\Omega_0}=0$. Therefore, $w_y^\Omega>0$ in $\Omega_0$ by the strong minimum principle. By Proposition~\ref{prop: equation for extremals}, there exists $y \in \Omega, c \neq 0$ such that $u = c w_y^\Omega$. The assertion for $u$ follows.
\end{proof}

As remarked in our proof above,  any extremal is a non-zero multiple of a potential. The next result tells us that not only must any extremal be a potential, but it is in fact sufficient to consider the infimum defining $\lambda_p(\Omega)$ restricted to potentials (even in the case extremals do not exist). Recall that Lemma~\ref{TechLimLem} implies that $|u|/d_\Omega^{1-n/p}$ attains a maximum in $\Omega$ provided $u \in \DD^{1,p}_0(\Omega)$.

\begin{prop}\label{prop: potentials are enough}
If $u \in \DD_0^{1,p}(\Omega)\setminus \{0\}$ and $x_0 \in \Omega$ satisfies
  \begin{equation}\label{eq: u special point prop potentials suffice}
    \frac{|u(x_0)|}{d_\Omega(x_0)^{1-n/p}} = \biggl\|\frac{u}{d_\Omega^{1-n/p}}\biggr\|_{\infty}\,,
  \end{equation}
  then 
  \begin{equation}\label{eq: R esitmate potential}
    \RR_p(\Omega, u) \geq 
    d_\Omega(x_0)^{p-n}\|Dw^\Omega_{x_0}\|_p^p\geq \RR_p(\Omega, w^\Omega_{x_0})\,.
  \end{equation}
  Equality holds in the first inequality if and only if $u=u(x_0)w^\Omega_{x_0}$ in which case equality holds also in the second.
\end{prop}

\begin{proof}
As $u(x_0)^{-1}u$ is a competitor for $w^\Omega_{x_0}$,
  \begin{equation}\label{eq: w energy minimizer}
    \|Dw^\Omega_{x_0}\|_p\leq \|D(u(x_0)^{-1}u)\|_p= |u(x_0)|^{-1}\|Du\|_p \,.
  \end{equation}
Moreover, equality holds if and only if $u = u(x_0)w^\Omega_{x_0}$.

\par  Using $$\biggl\|\frac{w^\Omega_{x_0}}{d_\Omega^{1-n/p}}\biggr\|_\infty \geq \frac{w^\Omega_{x_0}(x_0)}{d_\Omega(x_0)^{1-n/p}}= \frac{1}{d_\Omega(x_0)^{1-n/p}}\,,$$
  together with~\eqref{eq: u special point prop potentials suffice} and~\eqref{eq: w energy minimizer}, we find
    \begin{equation*}
    \RR_p(\Omega, w^\Omega_{x_0}) 
    \leq d_\Omega(x_0)^{p-n}\|Dw^\Omega_{x_0}\|_p^p \leq  \frac{\|Du\|_p^p}{\biggl(\displaystyle\frac{|u(x_0)|}{d_\Omega(x_0)^{1-n/p}}\biggr)^p} 
    = \RR_p(\Omega, u)\,.
  \end{equation*}
  In addition, equality holds in the second inequality if and only if $u= u(x_0)w^\Omega_{x_0}$. In this case, equality also holds in the first inequality by~\eqref{eq: u special point prop potentials suffice}.
\end{proof}

A direct consequence of Proposition~\ref{prop: potentials are enough} is that
\begin{equation}\label{eq: variational prob potential}
    \lambda_p(\Omega) =\inf_{x\in\Omega}\RR_p(\Omega, w^\Omega_{x})\,.
\end{equation}
Since $\RR_p(\Omega, w_x^\Omega)$ is a continuous function of $x\in \Omega$, the infimum above is a minimum if the value at some interior point is smaller than any limit either as $x$ approaches the boundary or $|x|$ tends to infinity. Along such sequences, it is useful to study a quantity that is slightly larger than $\RR_p(\Omega, w_x^\Omega)$ but has the key property that its limit is the same along minimizing sequences.

In order to formalize this idea we introduce the following notation. 
For a given $\Omega$, define $\mathcal{Y}_\Omega$ as the collection of sequences $\{x_k\}_{k\geq 1}\subset \Omega$ satisfying
\begin{equation*}
   \liminf_{k\to \infty}|x_k|=\infty \quad \mbox{or}\quad \limsup_{k \to \infty} d_\Omega(x_k) =0\,.
\end{equation*}   
That is, $\mathcal{Y}_\Omega$ is the set of sequences which eventually leave every compact subset of $\Omega$. Define
\begin{equation}\label{eq: biglambda}
  \Lambda_p(\Omega) := \inf\Bigl\{ \liminf_{k\to \infty} d_\Omega(x_k)^{p-n}\|Dw_{x_k}^\Omega\|_p^p: \{x_k\}_{k\geq 1}\in \mathcal{Y}_\Omega\Bigr\}\,.
\end{equation}
By a standard diagonalization argument, it follows that the infimum defining $\Lambda_p$ is actually a minimum (see Appendix~\ref{app: Lambda is attained}). The quantity $\Lambda_p(\Omega)$ is analogous to the ``Hardy constant at infinity'', which is central to the study of the existence of extremals for inequality~\eqref{eq: phardy} \cite{MR4012806,MR1655516,MR1817710,MR2196033,MR1868901}.

In view of~\eqref{eq: R esitmate potential} and~\eqref{eq: variational prob potential},
$$
\lambda_p(\Omega)\le \Lambda_p(\Omega)\,.
$$
Next, we show that the only way that $\Omega$ can lack an extremal is if it is favorable for minimizing sequences to concentrate at the boundary or move away to infinity. 
\begin{prop}\label{prop: Compactness threshold}
If
\begin{equation}\label{eq: concentration strict inequality intro}
  \lambda_p(\Omega) <\Lambda_p(\Omega)\,,
\end{equation}
then $\Omega$ has an extremal. Furthermore, all minimizing sequences $\{u_k\}_{k\geq 1}$ for~\eqref{eq: variational prob intro} with $\|Du_k\|_p=1$ for all $k$ are precompact in $\DD_0^{1,p}(\Omega)$.
 \end{prop} 

\begin{proof}
    Let $\{u_k\}_{k\geq 1}\subset \DD_0^{1,p}(\Omega)$ be a minimizing sequence for $\lambda_p(\Omega)$ with $\|Du_k\|_p=1$ for each $k \geq 1$. Since $\{u_k\}_{k\geq 1}$ is bounded in $\DD^{1,p}_0(\Omega)$ there exists a subsequence
    that converges weakly in $\DD_0^{1,p}(\Omega)$. Upon renaming this subsequence, we assume that the full sequence $\{u_k\}_{k\geq 1}$ converges weakly to $u$ in $\DD^{1,p}_0(\Omega)$. In particular, this implies that $u_k \to u$ in $C_{\rm loc}^{0, \alpha}(\R^n)$ for any $0<\alpha<1-n/p$. It then suffices to show that $u_k \to u$ in $\DD_0^{1,p}(\Omega)$ and $u$ is an extremal. We will use~\eqref{eq: concentration strict inequality intro} to exclude the possibility that $\{u_k\}$ either concentrates at the boundary or moves off to infinity.

    By Lemma~\ref{TechLimLem}, there exists a sequence $\{x_k\}_{k\geq 1}\subset \Omega$ such that
    \begin{equation*}
      \biggl\|\frac{u_k}{d_\Omega^{1-n/p}}\biggr\|_\infty = \frac{|u_k(x_k)|}{d_\Omega(x_k)^{1-n/p}}\,.
    \end{equation*}
    We claim that no subsequence of $\{x_k\}_{k\geq 1}$ can belong to $\mathcal{Y}_\Omega$. Indeed, if a subsequence of $\{x_k\}_{k\geq 1}$ belonged to $\mathcal{Y}_\Omega$ the corresponding subsequence of the sequence of potentials $\{w_{x_k}^\Omega\}_{k\geq 1}\in \DD_0^{1,p}(\Omega)$ is admissible in the definition of $\Lambda_p(\Omega)$. In view of Proposition~\ref{prop: potentials are enough}, the existence of such a subsequence would imply
    \begin{equation*}
       \Lambda_p(\Omega) \leq \limsup_{k\to \infty} d_\Omega(x_k)^{p-n}\|Dw_{x_k}^\Omega\|_p^p  \leq \lim_{k\to \infty} \RR_p(\Omega, u_k) = \lambda_p(\Omega)\,.
    \end{equation*} 
This contradicts our assumption and proves the claim. Consequently, $\limsup_{k\to \infty}|x_k| <\infty$ and $\liminf_{k\to \infty}d_\Omega(x_k)>0$. Therefore, $\{x_k\}_{k\geq 1}$ is precompact in $\Omega$.

    Passing to another subsequence if necessary, we may assume $\lim_{k\to \infty} x_k = x_0 \in \Omega$. Thus,    \begin{align*}
      \lim_{k\to \infty}\biggl\|\frac{u_k}{d_\Omega^{1-n/p}}\biggr\|_\infty &= \lim_{k\to \infty}\frac{|u_k(x_k)|}{d_\Omega(x_k)^{1-n/p}}
      = \frac{|u(x_0)|}{d_\Omega(x_0)^{1-n/p}}\le \biggl\|\frac{u}{d_\Omega^{1-n/p}}\biggr\|_\infty\,.
    \end{align*}
    Since $\|Du_k\|_p=1$ for all $k$ and $\{u_k\}_{k\geq 1}$ is a minimizing sequence, $u(x_0)\neq 0$. We also have 
    \begin{equation}\label{Lambda weak convergence inequality}
    \|Du\|_p\leq \lim_{k\rightarrow\infty}\|Du_k\|_p=1
   \end{equation}
    by weak convergence. Therefore, 
    \begin{equation*}
      \lambda_p(\Omega) = \lim_{k\to \infty} \RR_p(\Omega, u_k)  = \lim_{k\to \infty} \frac{d_\Omega(x_k)^{p-n}}{|u_k(x_k)|^{p}}  = \frac{d_\Omega(x_0)^{p-n}}{|u(x_0)|^p}\geq \biggl\|\frac{u}{d_\Omega^{1-n/p}}\biggr\|^{-p}_{\infty}\|Du\|^p_{p}\,.
    \end{equation*}
  As $\lambda_p(\Omega) \ge \RR_p(\Omega, u)$ and $u\in \DD^{1,p}_0(\Omega)$, $u$ is an extremal. In addition, equality must hold in~\eqref{Lambda weak convergence inequality}, from which we conclude that $u_k\rightarrow u$ in $\DD^{1,p}_0(\Omega)$, as weak convergence together with convergence of the $L^p$-norm implies strong convergence (see~\cite[Proposition 3.32]{MR2759829}).
\end{proof}


\section{A complete picture in one dimension}\label{sec: 1-d}
We can now fully describe what happens in the case $n=1$.

\begin{lem}\label{lem: 1D case}
Assume $n=1$. Then
  \begin{enumerate}
    \item $\lambda_p(\Omega) = 1$, and
    \item $\Omega$ has an extremal if and only if $\Omega$ contains an unbounded interval.
  \end{enumerate}
\end{lem}
\begin{proof}
  Fix $y \in \Omega$ and consider the potential $w_y^\Omega \in \DD^{1,p}_0(\Omega)$. By Corollary~\ref{ExtSignCor}, $w_y^\Omega$ vanishes in all connected components of $\Omega$ except the one containing $y$. We may assume that the connected component of $\Omega$ which contains $y$ is given by $(a, b)$ with $-\infty \leq a <b \leq \infty$ and either $-\infty<a$ or $b<\infty$. Routine computations lead us to the following observations. 
  \begin{enumerate}
    \item If $-\infty <a<b<\infty$, then
    \begin{equation*}
      w_y^\Omega(x) = \begin{cases}
        0 & \mbox{if } x\notin (a, b)\,,\\
      \displaystyle   \frac{x-a}{ y-a} \quad & \mbox{if } a< x \leq y\,,\\[10pt]
      \displaystyle  \frac{b-x}{b-y} \quad & \mbox{if } y< x <b\,,\\
      \end{cases}
      \quad \mbox{and} \quad \|Dw_y^\Omega\|_p^p 
      =\displaystyle \frac{1}{(y-a)^{p-1}}+ \frac{1}{(b-y)^{p-1}}\,.
    \end{equation*}
    \item If $a=-\infty$, then
    \begin{equation*}
      w_y^\Omega(x) = \begin{cases}
        1 & \mbox{if } x\leq y\,,\\
      \displaystyle  \frac{b-x}{b-y} \quad & \mbox{if } y< x <b\,,\\
        0 & \mbox{if } x\geq b\,,
      \end{cases}
      \quad \mbox{and} \quad \|Dw_y^\Omega\|_p^p 
      = \frac{1}{(b-y)^{p-1}}\,.
    \end{equation*} 
    \item if $b=\infty$, then
    \begin{equation*}
      w_y^\Omega(x) = \begin{cases}
        0 & \mbox{if } x\leq a\,,\\
     \displaystyle   \frac{x-a}{y-a} \quad & \mbox{if } a< x <y\,,\\
        1 & \mbox{if } x\geq b\,,
      \end{cases}
      \quad \mbox{and} \quad \|Dw_y^\Omega\|_p^p 
      = \frac{1}{(y-a)^{p-1}}\,.
    \end{equation*} 
  \end{enumerate}

 \par Note that in the cases of an unbounded interval, $\RR_p(\Omega, w_y^\Omega)=1$, while in the bounded case, $$\RR_p(\Omega, w_y^\Omega) = 1+ \biggl(\frac{\min\{y-a, b-y\}}{\max\{y-a, b-y\}}\biggr)^{p-1}>1\,.$$
Nevertheless, this expression for $\RR_p(\Omega, w_y^\Omega)$ can be made arbitrarily close to $1$ by letting $y$ approach either $a$ or $b$. In view of Proposition~\ref{prop: potentials are enough},  we conclude that in all cases 
$\lambda_p(\Omega)=1$. Furthermore, $\Omega \ni y \mapsto \RR_p(\Omega, w_y^\Omega)$ attains the value $1$ if and only if $\Omega$ contains an unbounded interval.
\end{proof}


\section{Universal bounds via supporting sets}
\label{sec: Universal bounds via supporting sets}

In this section, we turn to the question of universal upper and lower bounds for $\lambda_p(\Omega)$. We first notice that if $\Omega\subseteq \Omega'$ then any $u\in \DD^{1,p}_0(\Omega)$ also belongs to $\DD^{1,p}_0(\Omega')$. 
In this case, we also have $d_\Omega\leq d_{\Omega'}$ so that
$$
\biggl\|\frac{u}{d_{\Omega'}^{1-n/p}}\biggr\|_\infty \leq \biggl\|\frac{u}{d_{\Omega}^{1-n/p}}\biggr\|_\infty\,.
$$
Therefore, $\RR_p(\Omega, u) \leq \RR_p(\Omega', u)\,$ provided that $u \not\equiv 0$. In certain situations we can in fact conclude that $\RR_p(\Omega, u) = \RR_p(\Omega', u)\,$. To this end, we introduce the following notion.

\begin{defn}\label{def:SupportingSet}
  Suppose $\Omega, \Omega'\subsetneq \R^n$. We say that $\Omega'$ {\it supports} $\Omega$ at $x \in \partial\Omega$ if
  \begin{equation*}
    \Omega \subseteq \Omega' \qquad \mbox{and} \qquad x \in \partial\Omega'\,.
  \end{equation*}
  We say that $\Omega$ is {\it fully supported} by $\Omega'$, if for each $x \in \partial\Omega$ there exists a similarity transformation $T$ so that $T\Omega'$ supports $\Omega$ at $x$.
\end{defn}

\begin{rem}\label{rem:conv}
A set $\Omega$ is fully supported by $\R^n_\limplus$ if and only if $\Omega$ is convex. In fact, the notion of supporting sets is intended as a generalization of this property of convex sets.
\end{rem}

\begin{figure}[ht]
\centering
 \input{SupportedFig1}
 \caption{A non-convex polygon $\mathcal{P}$ which is fully supported by an infinite sector with opening angle $\varphi$ with one such supporting sector depicted in red. Equivalently, $\mathcal{P}$ satisfies a uniform (infinite) exterior cone condition.}
 \label{fig:Fully supported}
\end{figure}
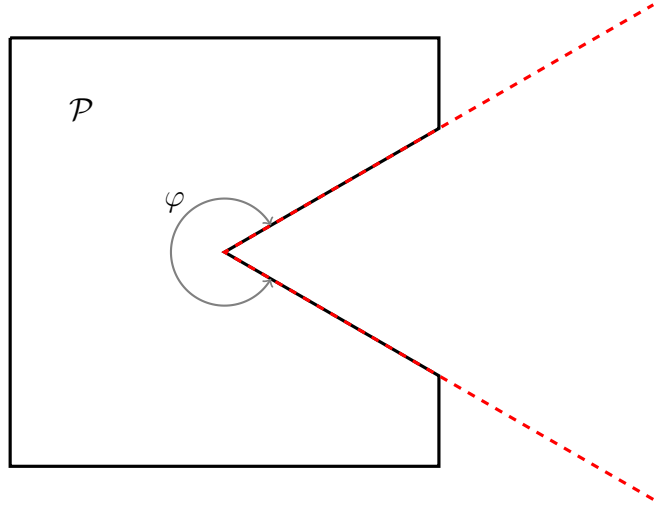
As we shall see, the following proposition can be useful both in proving upper and lower bounds for $\lambda_p$.

\begin{prop}\label{prop: Supporting sets}
  Assume $\Omega'$ supports $\Omega$ at $y_0 \in \partial \Omega$. If $u \in \DD^{1,p}_0(\Omega)$ and $x_0 \in \Omega$ satisfies
  \begin{equation*}
    \biggl\|\frac{u}{d_\Omega^{1-n/p}}\biggr\|_\infty = \frac{|u(x_0)|}{d_\Omega(x_0)^{1-n/p}} \quad \mbox{and} \quad |x_0-y_0|=d_\Omega(x_0)\,,
  \end{equation*}
  then
  \begin{equation*}
    \RR_p(\Omega, u) = 
    \RR_p(\Omega', u)\,.
  \end{equation*}
  Furthermore, if $\Omega$ is fully supported by $\Omega'$ then
  \begin{equation*}
    \lambda_p(\Omega) \geq \lambda_p(\Omega')\,.
  \end{equation*}
\end{prop}
\begin{proof}
As $u \in \DD_0^{1,p}(\Omega')$ and $d_\Omega \leq d_{\Omega'}$, 
  \begin{equation}\label{eq: quotient decrease domain inclusion}
    \frac{|u(x)|}{d_\Omega(x)^{1-n/p}} \geq \frac{|u(x)|}{d_{\Omega'}(x)^{1-n/p}}\quad \mbox{for all }x \in \Omega\,.
  \end{equation}
  Since $\Omega \subseteq \Omega'$, $y_0 \in \partial \Omega \cap \partial\Omega'$, and $d_\Omega(x_0)=|x_0-y_0|$ it follows that $d_{\Omega'}(x_0)= |x_0-y_0|$. Therefore, equality holds in~\eqref{eq: quotient decrease domain inclusion} for $x=x_0$. Consequently,
  \begin{equation*}
    \frac{|u(x_0)|}{d_{\Omega'}(x_0)^{1-n/p}}= \frac{|u(x_0)|}{d_{\Omega}(x_0)^{1-n/p}} = \biggl\|\frac{u}{d_{\Omega}^{1-n/p}}\biggr\|_\infty= \biggl\|\frac{u}{d_{\Omega'}^{1-n/p}}\biggr\|_\infty\,.
  \end{equation*}
We conclude that $    \RR_p(\Omega, u)= \RR_p(\Omega', u).$

  If $\Omega$ is fully supported by $\Omega'$, then for each $u \in \DD^{1,p}_0(\Omega)$ we can apply Lemma~\ref{TechLimLem} and the above reasoning to deduce that there exists a similarity transform $T_u$ such that $u \in\DD_0^{1,p}(T_u\Omega')$ and $\RR_p(\Omega, u) = \RR_p(T_u\Omega', u)$. By Lemma~\ref{ScalingInvarianceLem}, $\RR_p(T_u\Omega', u) = \RR_p(\Omega', u\circ T_u^{-1})$. Therefore,
  \begin{equation*}
     \lambda_p(\Omega') \leq \!\inf_{u \in \DD_0^{1,p}(\Omega)\setminus \{0\}}\!\! \RR_p(\Omega', u\circ T_u^{-1}) = \!\inf_{u \in \DD_0^{1,p}(\Omega)\setminus \{0\}}\!\! \RR_p(\Omega, u) = \lambda_p(\Omega)\,. \qedhere
  \end{equation*}
\end{proof}

We obtain the following corollaries which follow directly from Proposition~\ref{prop: Supporting sets}. 
\begin{cor}\label{cor: Universal bounds}
For any open $\Omega \subsetneq \R^n$,
 $$\lambda_p(\R^n\setminus \{0\})\leq \lambda_p(\Omega)\leq \lambda_p(B_1)\,.$$
\end{cor}
\begin{proof}
The lower bound follows from the translation invariance of $\lambda_p$ and noting that if $x \in \partial\Omega$ then $\R^n \setminus \{x\}$ supports $\Omega $ at $x$. This implies that every $\Omega \subsetneq\R^n$ is fully supported by $\R^n \setminus\{0\}$. The upper bound follows by the similarity invariance combined with the observation that $B_1$ is fully supported by any set $\Omega $. Indeed, for any $y_0 \in \Omega$, $\Omega$ supports the ball $B_{d_\Omega(y_0)}(y_0)$ at $x_0$ where $|x_0-y_0|=d_\Omega(y_0)$. Since $B_1$ is rotationally invariant, we conclude that $B_1$ is fully supported by $\Omega $.
\end{proof}

\begin{figure}[ht]
\centering
 \input{SupportedFig2}
 \caption{A schematic description of the fact that every open set $\Omega$ fully supports $B_1$. Given $\Omega, y_0 \in \Omega, x\in \partial B_1$ a similarity transform $T$ is constructed satisfying that $Ty_0=0$ and $T\Omega$ supports $B_1$ at $x$.}
 \label{fig:Every set fully supports B_1}
\end{figure}
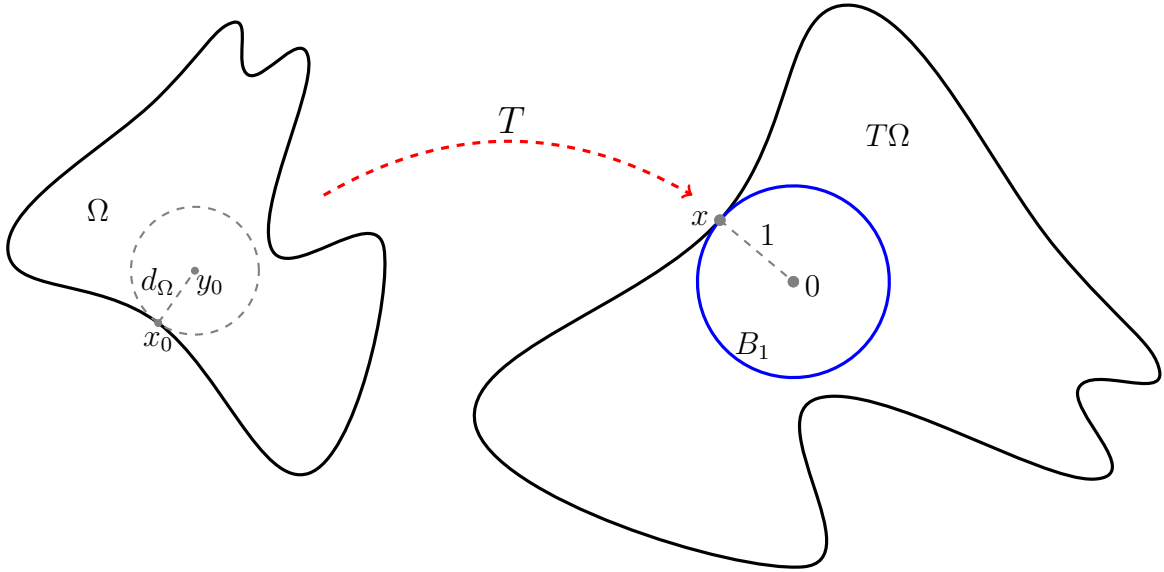
\begin{cor}\label{cor: Lower bound convex}
  If $\Omega \subsetneq \R^n$ is convex then
  \begin{equation*}
    \lambda_p(\R^n_\limplus)\leq \lambda_p(\Omega)\leq \lambda_p(B_1) \,.
  \end{equation*}
\end{cor}

\begin{proof}
The upper bound was proved above, so we focus on the lower bound. As noted in Remark~\ref{rem:conv}, $\Omega$ is fully supported by $\R^n_\limplus$. We can therefore conclude that $\lambda_p(\Omega)\geq \lambda_p(\R^n_\limplus)$.
\end{proof}

\section{Proof of Theorem~\ref{thm: universal sharp bounds} and Theorem~\ref{thm: convex, punctured, exterior domains}}\label{sec: proof}

\label{sec: convex, punctured, exterior domains}

Our next aim is to prove Theorem~\ref{thm: universal sharp bounds} and Theorem~\ref{thm: convex, punctured, exterior domains}. We first establish the lemma below, where we study the particular cases $\Omega=\R^n_\limplus$ and $\Omega=\R^n \setminus\{0\}$. It turns out that both the value of $\lambda_p(\Omega)$ and extremals for these domains can be expressed in terms of the sharp constant in Morrey's inequality and Morrey extremals, respectively. Here we call $v\in \DD^{1,p}(\R^n)$ a {\it Morrey extremal} if $v$ is not constant throughout $\R^n$ and equality holds in~\eqref{MorreyInequality} with $C=C_{n,p}$. 

\begin{lem}\label{lem: half- and punctured space}
  The following assertions hold. 
  \begin{enumerate}[label=(\textit{\roman*})]
      \item\label{constantHalfspace} $\lambda_p(\R^n_\limplus) = 2^{n-1}C_{n,p}^{-p}$.\\[-6.5pt]
      
      \item\label{extremalHalfspace} If $y \in \R^n_\limplus$, then $w_{y}^{\R^n_{\footnotesize\limplus}}$ is an extremal for $\lambda_p(\R^n_\limplus)$. Furthermore, the odd reflection of $w_{y}^{\R^n_{\footnotesize\limplus}}$ through the hyperplane $x_n=0$ is a Morrey extremal.\\[-6.5pt]
    
      \item\label{constantPunctured} $\lambda_p(\R^n\setminus \{0\}) = C_{n,p}^{-p}$.\\[-6.5pt]
      
      \item\label{extremalPunctured} If $y\in \R^n\setminus\{0\}$, then $ w_{y}^{\R^n\setminus\{0\}}$ is an extremal for $\lambda_p(\R^n\setminus\{0\})$ and a Morrey extremal.
  \end{enumerate}
\end{lem}

\begin{proof} {\it Part 1:~\ref{constantHalfspace} and~\ref{extremalHalfspace}}. We claim that $\RR_p(\R^n_\limplus, w_y^{\R^n_{\footnotesize\limplus}})$ is independent of $y \in \R^n_\limplus$. Note that  if $y= (y', y_n) \in \R^n_\limplus$, 
   then 
   $$
      \tilde w(x) = w_{y}^{\R^n_{\footnotesize\limplus}}((y',0)+y_nx) \in \DD^{1,p}_0(\R^n_\limplus)
   $$
   satisfies 
   \begin{equation}
  \begin{cases}
   - \Delta_p \tilde w =0 & \mbox{in } \R^n_\limplus \setminus \{e_n\}\,,\\
   \hspace{.33in}\tilde w =0 & \mbox{on } \partial\R^n_\limplus\,,\\
      \hspace{.05in} \tilde w(e_n)=1\,.
  \end{cases}
\end{equation}
That is, $\tilde w = w_{e_n}^{\R^n_{\footnotesize\limplus}}$. Lemma~\ref{ScalingInvarianceLem} then implies that $\RR_p(\R^n_\limplus, w_y^{\R^n_{\footnotesize\limplus}})= \RR_p(\R^n_\limplus, w_{e_n}^{\R^n_{\footnotesize\limplus}})$. By~\eqref{eq: variational prob potential}, $\lambda_p(\R^n_\limplus)=\RR_p(\R^n_\limplus, w_{y}^{\R^n_{\footnotesize\limplus}}).$ We conclude that 
$w_{y}^{\R^n_{\footnotesize\limplus}}$ is an extremal. 

Again fix $y=(y',y_n)\in \R^n_\limplus$. By~\cite[Theorem 2.4]{MR4275749}, there exists a Morrey extremal $u$ satisfying
\begin{equation}\label{eq: properties Morrey extremal}
  u(y)=1\,, \quad u((y', -y_n)) =-1\,, \quad \mbox{and}\quad [u]_{1-n/p} = \frac{|u(y)-u((y', -y_n))|}{|y-(y',-y_n)|^{1-n/p}} \,.
\end{equation}
Here $[u ]_{1-n/p}$ is the $1-n/p$ H\"older seminorm of $u$. Moreover, $-\Delta_pu=0$ in $\R^n\setminus\{y,(y', -y_n)\}$. By~\cite[Theorem 1.1]{MR4160015}, $u$ is antisymmetric across the hyperplane $x_ n=0$. In particular, $u$ vanishes on this hyperplane. It follows that $u|_{\R^n_{\footnotesize\limplus}}= w_{y}^{\R^n_{\footnotesize\limplus}}$. Therefore, the odd reflection of $w_{y}^{\R^n_{\footnotesize\limplus}}$ through this hyperplane is $u$.

\par Let us write $w=w_{e_n}^{\R^n_{\footnotesize\limplus}}$ and $u$ for the odd reflection of $w$ through $x_n=0$. As noted above, $u$ is a Morrey extremal satisfying~\eqref{eq: properties Morrey extremal} with $y=e_n$. As a result,  
\begin{align*}
  1&=\frac{|w(e_n)|}{d_{\R^n_\limplus}(e_n)^{1-n/p}}\\
  &\leq \biggl\|\frac{w}{d_{\R^n_\limplus}^{1-n/p}}\biggr\|_\infty\\
  &=\sup_{x\in \R^n_\limplus}\frac{|w(x)|}{|x_n|^{1-n/p}} \\
  &= 2^{-n/p}\sup_{x\in \R^n_\limplus} \frac{| u(x)- u((x',-x_n))|}{|x-(x', -x_n)|^{1-n/p}}\\
  & \leq 2^{-n/p}[u]_{1-n/p}\\
  & =1\,
\end{align*}
and
\begin{equation*}
  \lambda_p(\R^n_\limplus) =  \biggl\|\frac{w}{d_{\R^n_\limplus}^{1-n/p}}\biggr\|^{-p}_{\infty}\|Dw\|^p_{p}  = 2^{n-1}[u]_{1-n/p}^{-p}\|D u\|_p^p = 2^{n-1}C_{n,p}^{-p}\,.
\end{equation*}

\medskip
\par {\it Part 2:~\ref{constantPunctured} and~\ref{extremalPunctured}}. As we argued above, we can deduce that $\RR_p(\R^n\setminus\{0\}, w_y^{\R^n\setminus\{0\}})$ does not depend on $y \in \R^n\setminus\{0\}$ by showing 
$$
w_{e_n}^{\R^n\setminus\{0\}}(x)=w_{y}^{\R^n\setminus\{0\}}(|y|Qx)
$$
for any $Q \in O(n)$ with $Qe_n=y/|y|$. Again,~\eqref{eq: variational prob potential} gives 
$\lambda_p(\R^n\setminus\{0\})=\RR_p(\R^n\setminus\{0\}, w_{y}^{\R^n\setminus\{0\}})$. It follows that $w_{y}^{\R^n\setminus\{0\}}$ is an extremal. 

\par Let $u$ be a Morrey extremal with 
\begin{equation*}
u(0)=0\,, \quad u(y)=1\,,\quad \text{and}\quad  [u]_{1-n/p} = \frac{|u(y)|}{|y|^{1-n/p}}\,.
\end{equation*}
As $u\in \DD^{1,p}_0(\R^n\setminus\{0\})$ and $-\Delta_p u= 0$ in $\R^n\setminus\{0, y\}$, $u=w^{\R^n\setminus\{0\}}_y$. 
It follows that 
\begin{equation*}
\lambda_p(\R^n\setminus\{0\}) = \biggl\|\frac{u}{d_{\R^n\setminus\{0\}}^{1-n/p}}\biggr\|_\infty^{-p}\|Du\|_p^p = [u]_{1-n/p}^{-p}\|Du\|^p_p= C_{n,p}^{-p}\,.\qedhere
\end{equation*}
\end{proof}

We will now prove Theorems ~\ref{thm: universal sharp bounds} and~\ref{thm: convex, punctured, exterior domains} by
applying various of the assertions derived above. \begin{proof}[Proof of Theorem~\ref{thm: universal sharp bounds}] We first observe that $\R^n_\limplus$ can be exhausted by $\{B_j(je_n)\}_{j\ge 1}$. Therefore, by Lemma~\ref{lem: interior exhaustion} and the similarity invariance
$$
\lambda_p(B_1)=\lim_{j\to \infty}\lambda_p(B_j(je_n))\leq \lambda_p(\R^n_\limplus)\,.
$$
  The theorem now follows from this inequality together with Lemma~\ref{lem: half- and punctured space} and Corollary~\ref{cor: Universal bounds}.
\end{proof}
\begin{proof}[Proof of Theorem~\ref{thm: convex, punctured, exterior domains}]
$(1)$ Assume $\Omega$ is convex. By Theorem~\ref{thm: universal sharp bounds} and Corollary~\ref{cor: Lower bound convex},
$$
\lambda_p(\R^n_\limplus)\le \lambda_p(\Omega)\le \lambda_p(B_1)\le \lambda_p(\R^n_\limplus)\,.
$$
That is, $\lambda_p(\Omega)=\lambda_p(\R^n_\limplus)$. 

\par Now suppose $u\in \DD^{1,p}_0(\Omega)$ is an extremal. By Corollary~\ref{ExtSignCor}, we may assume that $u>0$ in~$\Omega$. Let us also choose $x_0\in \Omega$ with
$$
\biggl\|\frac{u}{d_\Omega^{1-n/p}} \biggr\|_\infty=\frac{u(x_0)}{d_\Omega(x_0)^{1-n/p}}
$$
and $y_0\in \partial \Omega$ with $d_\Omega(x_0)=|x_0-y_0|$. Since $\Omega$ is convex, there exists a halfspace $\Pi$ such that 
\begin{equation}
  \Omega \subseteq \Pi \quad \mbox{and} \quad y_0 \in \partial \Pi\,.
\end{equation}
Then $u\in \DD^{1,p}_0(\Pi)$ and $d_\Pi(x_0)=|x_0-y_0|$. By Proposition~\ref{prop: Supporting sets},
$$
\RR_p(\Pi, u) =\RR_p(\Omega, u)=\lambda_p(\Omega)=\lambda_p(\R^n_\limplus)=\lambda_p(\Pi)\,.
$$
Hence, $u$ is also an extremal for $\Pi$. If $\Omega \subsetneq \Pi$, then $u$ vanishes in the open set $\Pi\setminus\overline{\Omega}$. This contradicts Corollary~\ref{ExtSignCor} since $\Pi$ is connected. As a result, $\Omega$ does not admit an extremal unless it is a halfspace. 

\medskip 

(2) Suppose $\Omega\subset \R^n$ is open. By translation invariance, we may assume $x_0=0$. We have already established that $\lambda_p(\Omega\setminus\{0\})\ge \lambda_p(\R^n\setminus\{0\})$ in Theorem~\ref{thm: universal sharp bounds}. As for the upper bound, $\R^n\setminus\{0\}$ is exhausted by $\{j\Omega\setminus\{0\}\}_{j\ge 1}$, so Lemma~\ref{lem: interior exhaustion} implies $\lambda_p(\Omega\setminus \{0\})\leq \lambda_p(\R^n\setminus \{0\})$. Arguing as in $(1)$, we can 
show that if $u\in\DD^{1,p}_0(\Omega\setminus\{0\})$ is an extremal, it is also an extremal in $\R^n\setminus\{0\}$ which vanishes at any point in the complement of $\Omega\setminus\{0\}$. As $\R^n\setminus\{0\}$ is connected, $u$ would have to vanish identically unless $\Omega=\R^n$. We conclude that no such extremal exists for $\Omega\setminus\{0\}$ unless $\Omega=\R^n$. 

\medskip 

(3) Let $K\subset \R^n$ be compact and nonempty. By translation invariance, we may assume $0\in K$. According to Theorem~\ref{thm: universal sharp bounds}, $\lambda_p(\R^n\setminus K)\ge \lambda_p(\R^n\setminus\{0\})$. Since $\R^n\setminus\{0\}$ is exhausted by $\{j^{-1}(\R^n\setminus K )\}_{j\ge 1}$, Lemma~\ref{lem: interior exhaustion} gives $\lambda_p(\R^n\setminus K)\leq \lambda_p(\R^n\setminus\{0\})$. Similar to how we reasoned in (2) and (3), we may conclude that the only way for $\R^n\setminus K$ to have an extremal is if $K=\{0\}$. 
\end{proof}

\begin{rem}\label{rem: onlyadmit} The method used to prove the nonexistence of extremals in Theorem~\ref{thm: convex, punctured, exterior domains} can also be used to show: if $\Omega$ has an extremal and $\lambda_p(\Omega)=\lambda_p(\R^n\setminus \{0\})$, then $\Omega=\R^n\setminus\{x_0\}$ for some $x_0$. The key observation here is that $\Omega$ is fully supported by $\R^n\setminus\{0\}$. We leave the details to the reader.
\end{rem}

\begin{rem} Modulo symmetry, the only connected sets whose extremals are restrictions of Morrey extremals are those covered by Lemma~\ref{lem: half- and punctured space}. Indeed, if $\Omega \subsetneq \R^n$ is a connected set and $u \in \DD^{1,p}_0(\Omega)$ is an extremal for $\lambda_p(\Omega)$ which is the restriction of a Morrey extremal $v$, then $\partial\Omega$ is the zero level set of this Morrey extremal. By~\cite{MR4275749}, each level set of a Morrey extremal is either a bounded convex set, a halfspace, or the complement of a bounded convex set. If $\Omega$ is convex and not a halfspace then $\Omega$ does not admit an extremal, contradicting that $u = v|_\Omega$ is an extremal. Similarly if $\Omega$ is the complement of a compact set then $\Omega$ admits an extremal if and only if this compact set is a singleton. Therefore, halfspaces and $\R^n \setminus\{x_0\}$ are the only connected sets having extremals given by restrictions of Morrey extremals.
\end{rem}

\section{Dilation invariant domains}
\label{sec: Dilation invariant domains}
We will say $\mathcal{C} \subset \R^n$ is a {\it cone} if 
$$
t\mathcal{C}=\mathcal{C} \mbox{ for each } t>0\,. 
$$
That is, $\mathcal{C}$ is a cone if it is dilation invariant. It is evident that the complement of a cone is also a cone. It is also easy to check that if a cone is convex, then it is closed under vector addition. In this section, we will give sufficient conditions under which a cone admits an extremal. We will assume throughout that $\Omega\subsetneq\R^n$ is open and nonempty and $n\ge 2$.

\begin{prop}\label{translationProp}
Suppose $\Omega \subsetneq \R^n$ is a cone. Further assume there is a compact 
$$
K\subset \{x\in \Omega: d_\Omega(x)=1\}
$$
such that for each $x\in K$ there is $y\in \R^n$ with 
$$
\begin{cases}
x+y\in K\,,\\
y+\Omega\subset \Omega\,.
\end{cases}
$$
Then $\Omega$ admits an extremal $u$ for which $|u|/d_\Omega^{1-n/p}$ attains its maximum in $K$.
\end{prop}

\begin{proof}
{\it Step 1.} Let $\{u_k\}_{k\geq 1}\subset \DD^{1,p}_0(\Omega)$ with 
$$
\lambda_p(\Omega)=\lim_{k\rightarrow\infty}\RR_p(\Omega,u_k)
$$
and choose $\{z_k\}_{k\geq 1}\subset \Omega$ which satisfies
\begin{equation*}
   \biggl\|\frac{u_k}{d_\Omega^{1-n/p}}\biggr\|_\infty = \frac{|u_k(z_k)|}{d_\Omega(z_k)^{1-n/p}}
\end{equation*} 
for each $k\ge 1$. Since $\Omega$ is a cone, we may define $v_k \in \DD_0^{1,p}(\Omega)$ by setting
$$
v_k(x)=\frac{u_k(d_\Omega(z_k)x)}{d_\Omega(z_k)^{1-n/p}}\,.
$$
By Lemma~\ref{ScalingInvarianceLem}, $ \RR_p(\Omega, v_k) = \RR_p(\Omega, u_k)$.
Therefore, $\{v_k\}_{k\geq 1}$ is also a minimizing sequence. 

\par Next we set
\begin{equation*}
  x_k = \frac{z_k}{d_\Omega(z_k)}
\end{equation*}
and note that
\begin{equation*}
  \biggl\|\frac{v_k}{d_\Omega^{1-n/p}}\biggr\|_\infty = \frac{|v_k(x_k)|}{d_\Omega(x_k)^{1-n/p}}\,.
\end{equation*}
By~\eqref{distScaling}, $d_\Omega$ is positively homogeneous as $\Omega$ is dilation invariant. It follows that $d_\Omega(x_k) = 1$, so $|v_k|/d_\Omega^{1-n/p}$ attains its maximum in the set $\{x\in \Omega: d_\Omega(x) =1\}$.

\medskip

{\noindent\it Step 2.} By assumption, there exists $y_k\in \R^n$ for each $k$ with
$$
\begin{cases}
x_k+y_k\in K\,,\\
y_k+\Omega\subset \Omega\,.
\end{cases}
$$
Setting
$$
w_k(x)=v_k(x-y_k)\,,
$$ 
we observe that $w_k\in \DD^{1,p}_0(\Omega)$ and
\begin{equation}\label{eq: translation to K gradnorm}
\|Dw_k\|_p=\|Dv_k\|_p\,.
\end{equation}
Moreover, since $y_k + \Omega \subseteq \Omega$
$$
d_{\Omega}(x-y_k)=d_{y_k+\Omega}(x)\le d_\Omega(x)
$$
for all $x\in y_k+\Omega$. Thus, 
\begin{equation}\label{eq: translation to K supnorm}
\frac{|w_k(x)|}{d_\Omega(x)^{1-n/p}}\le \frac{|w_k(x)|}{d_{\Omega}(x-y_k)^{1-n/p}}= \frac{|v_k(x-y_k)|}{d_{\Omega}(x-y_k)^{1-n/p}}\le \biggl\|\frac{v_k}{d_\Omega^{1-n/p}}\biggr\|_\infty
\end{equation}
for all $x\in \mathrm{supp}(w_k)\subseteq y_k+\Omega$ with equality when $x = x_k + y_k \in K$. Combining~\eqref{eq: translation to K gradnorm} and~\eqref{eq: translation to K supnorm} leads to
$$
\RR_p(\Omega, w_k) = \RR_p(\Omega, v_k)\,.
$$
It follows that $\{w_k\}_{k\geq 1}$ is a minimizing sequence satisfying that $|w_k|/d_\Omega^{1-n/p}$ is attained in $K$ for each $k\geq 1$.

\medskip

{\noindent \it Step 3.} Since $K$ is compact, $\{x_k+y_k\}_{k\geq 1}$ has a subsequence which converges to a limit $x\in K$. 
Arguing as at the end of the proof of Proposition~\ref{prop: Compactness threshold}, we deduce that along this subsequence $w_{k}$ 
converges locally uniformly and in $\DD^{1,p}_0(\Omega)$ to a function $u$ which is an extremal for $\lambda_p(\Omega)$. As $|w_k|/d_\Omega^{1-n/p}$ was maximized in $K$ for each $k\geq 1$, it follows that this property remains true for $|u|/d_\Omega^{1-n/p}$. We leave the details to the reader. 
\end{proof}

\par We will verify that $\Omega$ satisfies the hypotheses of Proposition~\ref{translationProp} whenever its complement is a convex cone. First we will need to justify the following lemma. 
\begin{lem}\label{lem: moving special point complement of convex cone}
Assume $\Omega^c$ is a convex cone, $x_0\in \Omega$, and $y_0\in \partial\Omega$ with $d_\Omega(x_0)=|x_0-y_0|$. Then 
$$
-y_0+\Omega\subset \Omega \quad \mbox{and} \quad
d_\Omega(x_0-y_0)=|x_0-y_0|\,.
$$
\end{lem}
\begin{proof} Let $x\in \Omega$. We aim to prove that $-y_0+x\in \Omega$. If this were not the case, then $x=(-y_0+x)+y_0\in \Omega^c$ since $\Omega^c$ is a convex cone. However, this would be a contradiction. As $x\in \Omega$ was arbitrary, we conclude that $-y_0+\Omega\subset \Omega$. 

\par Since $0\in \partial \Omega$,
$$
d_\Omega(x_0-y_0)\le |x_0-y_0|\,.
$$
We are left to verify the opposite inequality. To this end, we recall that 
$$
(x_0-y_0)\cdot (x-y_0)\le 0
$$
for all $x\in \Omega^c$ as $\Omega^c$ is closed and convex~\cite[Theorem 5.2]{MR2759829}. Choosing $x=2y_0$ and $x=0$ gives $(x_0-y_0)\cdot y_0=0$ and  hence
$$
\Omega^c\subset \Pi:=\{x\in \R^n: (x_0-y_0)\cdot x\le 0\}\,.
$$

\par It follows that 
\begin{equation*}
d_\Omega(x_0-y_0)=\inf_{w\in \Omega^c}|x_0-y_0-w|\ge \inf_{w\in \Pi}|x_0-y_0-w|\,.
\end{equation*}
And as $(x_0-y_0)\cdot w\le 0$ for $w\in \Pi$,
$$
|x_0-y_0-w|^2=|x_0-y_0|^2+|w|^2-2(x_0-y_0)\cdot w\ge |x_0-y_0|^2+|w|^2\ge |x_0-y_0|^2\,.
$$
We conclude that
\begin{equation*}
d_\Omega(x_0-y_0)\ge \inf_{w\in \Pi}|x_0-y_0-w|=|x_0-y_0|\,. \qedhere
\end{equation*}
\end{proof}

\begin{thm}\label{prop: complement of convex cone}
Assume that $\Omega^c$ is a convex cone. Then $\Omega$ has an extremal.
\end{thm}

\begin{proof} Note that
$$
K=\{x\in \Omega: d_\Omega(x)=1\}\cap \overline{B_1}
$$
is a compact subset of $\{x\in \Omega: d_\Omega(x)=1\}$. If $x\in \Omega$ and $y\in \partial\Omega$ with $d_\Omega(x)=|x-y|=1$, the previous lemma shows that
$$
\begin{cases}
x-y\in K\,,\\
-y+\Omega\subset \Omega\,.
\end{cases}
$$
We conclude the proof by appealing to Proposition~\ref{translationProp}.
\end{proof}

Theorem~\ref{prop: complement of convex cone} implies the existence of extremals proved in Lemma~\ref{lem: half- and punctured space} since the complement of $\R^n_\limplus$ and $\R^n\setminus\{0\}$ are convex cones. We will consider another family of examples below. 

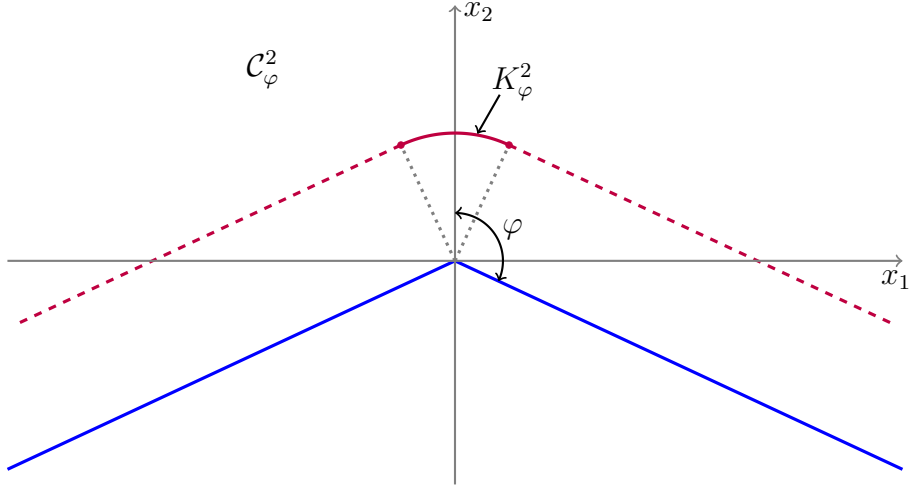
\begin{figure}[ht]
\centering
 \input{CeePhiFig}
 \caption{A domain $\Cee^2_\varphi\subset \R^2$ for $\varphi\in (\pi/2,\pi]$. The domain is the region of the plane above the blue curve. The purple curve indicates the set $d_{\Cee^2_\varphi}=1$, which is the union of three curves; two rays (dashed) and a circular arc $K_\varphi^2$ (solid). In higher dimensions the set $\Cee^n_\varphi$ can be obtained by rotation of $\Cee^2_\varphi$ around the axis of symmetry.}
 \label{fig:CeePhi}
\end{figure}

\begin{ex}\label{ComplementCircCone}For $\varphi \in (0, \pi)$, set
$$
\Cee_\varphi^n=\bigl\{x = (x', x_n) \in \R^n: x_n > \cot(\varphi)|x'|\bigr\}\,. 
$$
We will also consider the limiting case as the angle $\varphi$ tends to $\pi$
$$
  \Cee^n_\pi = \R^n \setminus\{x = (x', x_n) \in \R^n: |x'|=0, x_n\leq 0 \}\,.
$$
The set $\Cee^n_\varphi$ is a circular cone around the positive $x_n$-axis with opening angle $\varphi$ measured relative to the $x_n$-axis (see Figure~\ref{fig:CeePhi}). If $\varphi< \pi/2$, then $\Cee^n_\varphi$ is a convex cone which is not a halfspace. By Theorem~\ref{thm: convex, punctured, exterior domains}, $\lambda_p(\Cee^n_\varphi)=\lambda_p(\R^n_\limplus)$, and no extremals exist. Alternatively, if $\varphi\geq \pi/2$, the set $\Cee^n_\varphi$ is the complement of a convex cone. Theorem~\ref{prop: complement of convex cone} therefore implies that $\Cee^n_\varphi$ has an extremal for every $\varphi \in [\pi/2, \pi]$. 
\end{ex}
\par The family $\Cee^n_\varphi$ is of particular interest as it forms a natural class of limiting profiles that can occur in the analysis developed in Sections~\ref{sec: Blow-up analysis} and~\ref{sec: Examples and open problems} (especially, for $n=2$). With this in mind, we note for future reference that $[\pi/2, \pi]\ni \varphi \mapsto \lambda_p(\Cee^n_\varphi)$ is strictly decreasing. 
\begin{lem}\label{lem: CeePhi monotonicity}
For $\pi/2\le \varphi_1<\varphi_2\le \pi$,
$$
\lambda_p(\Cee^n_{\varphi_2})< \lambda_p(\Cee^n_{\varphi_1})\,.
$$
\end{lem}
\begin{proof}
Given $x_0 \in \partial \Cee^n_{\varphi_1}$ we observe that $\Cee^n_{\varphi_1} \subset x_0+\Cee^n_{\varphi_2}$ and $x_0 \in \partial(x_0+\Cee^n_{\varphi_2})$. Therefore, $\Cee^n_{\varphi_1}$ is fully supported by $\Cee^n_{\varphi_2}$; Proposition~\ref{prop: Supporting sets} implies that $\lambda_p(\Cee^n_{\varphi_2})\leq \lambda_p(\Cee^n_{\varphi_1})$. By Proposition~\ref{translationProp}, $\Cee^n_{\varphi_1}$ has an extremal $u\in \DD^{1,p}_0(\Cee^n_{\varphi_1})$. If $\lambda_p(\Cee^n_{\varphi_2})=  \lambda_p(\Cee^n_{\varphi_1})$, then $u$ would be an extremal also in $ \Cee^n_{\varphi_2}$. As $u$ vanishes in $ \Cee^n_{\varphi_2}\setminus\overline{\Cee^n_{\varphi_1}}\neq \emptyset$, this contradicts Corollary~\ref{ExtSignCor}. Thus,  $\lambda_p(\Cee_{\varphi_2})<\lambda_p(\Cee_{\varphi_1})$.
\end{proof}
\begin{rem}\label{lem: CeePhi continuity}
It is also possible to show that $[\pi/2,\pi]\ni \varphi \mapsto \lambda_p(\Cee^n_\varphi)$ is continuous. 
\end{rem}

\section{Local and global analysis of \texorpdfstring{$\Lambda_p(\Omega)$}{Lambda}}
\label{sec: Blow-up analysis}

In this section, we will focus on understanding how $\Lambda_p$ compares to the value of $\lambda_p$ in simpler model sets. The typical model sets are dilation invariant sets such as halfspaces or cones, which were discussed in Section~\ref{sec: Dilation invariant domains}.
Our aim is to prove upper bounds on $\Lambda_p(\Omega)$, and in turn on $\lambda_p(\Omega)$, in terms of blow-up/blow-down limits of $\Omega$. We will also establish lower bounds on $\Lambda_p(\Omega)$, by studying the asymptotic behavior of $\RR_p(\Omega, w_{x_k}^\Omega)$ when $\{x_k\}_{k\geq 1} \subset \Omega$ does not have a limit in $\Omega$. The combination of these lower bounds and Proposition~\ref{prop: Compactness threshold} will be one our main tools in proving the existence of extremals for $\lambda_p(\Omega)$.

For a given $\Omega \subsetneq \R^n$, we will deduce upper bounds on $\Lambda_p(\Omega)$ by making the following observation: if we can locally approximate a dilation invariant set $\mathcal{C}$ to arbitrary precision by a sequence of dilations and translations acting on $\Omega$, then a trial sequence for $\Lambda_p(\Omega)$ can be constructed from an almost minimizer of $\lambda_p(\mathcal{C})$. 
There are two important cases to have in mind. The first is zooming in at a boundary point of $\Omega$, in which case $\mathcal{C}$ is the blow-up of $\Omega$ at this point. The second case occurs 
when zooming out so far that only asymptotic features of $\Omega$ remain visible and are described by a cone $\mathcal{C}$.
\begin{lem}\label{lem: blow-up upper bound}
Let $\mathcal{C}\subsetneq \R^n$ be an open cone. Assume for all small enough $\delta \in (0, 1)$ there are $t=t(\delta)>0$ and $y=y(\delta)\in \R^n$ such that
\begin{equation*}
  (t\Omega-y)\cap B_1 \supset \{x \in \mathcal{C}: d_\mathcal{C}(x)>\delta\}\cap B_1
\end{equation*}
and for each $x \in \mathcal{C}\cap B_{1/2}$
\begin{equation*}
\lim_{\delta \to 0} d_{t\Omega-y}(x) = d_{\mathcal{C}}(x)\,.
\end{equation*}
Then $$\Lambda_p(\Omega)\leq \lambda_p(\mathcal{C})\,.$$ 
\end{lem}

\begin{figure}[ht]
\centering
\input{BlowupFig}
\caption{A sequence of four blow-ups around a point on the boundary of a domain $\Omega$ with a limiting profile $\mathcal{C}$ depicted in blue.}
 \label{fig:Blowup}
\end{figure}
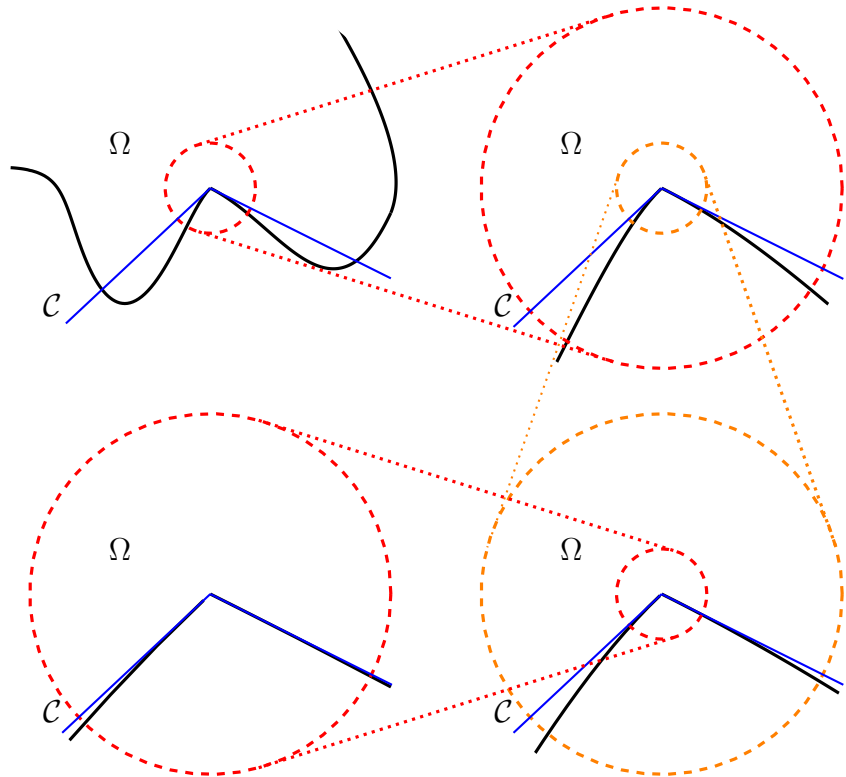

\begin{proof} 
Fix $\epsilon >0$. By Lemma~\ref{lem: variational def cpt support}, there exists an $r>0$ and a function $u \in C_c^\infty(\mathcal{C}\cap B_r)$ so that
  \begin{equation*}
  \RR_p(\mathcal{C},u)\leq \lambda_p(\mathcal{C})+\epsilon\,.
  \end{equation*}
  Since both $\mathcal{C}$ and the Rayleigh quotient are invariant under dilations, we can ensure that $r=1/2$. By Lemma~\ref{TechLimLem}, there is $x^* \in \mathcal{C}\cap B_{1/2}$ such that
  \begin{equation*}
    \frac{|u(x^*)|}{d_\mathcal{C}(x^*)^{1-n/p}} = \biggl\|\frac{u}{d_{\mathcal{C}}^{1-n/p}}\biggr\|_\infty\,.
  \end{equation*}

   For $\delta>0$, define
   \begin{equation*}
     \mathcal{C}_\delta = \{x\in \mathcal{C}: d_\mathcal{C}(x)>\delta\}\,.
   \end{equation*}
   By assumption, there are $t>0, y\in \R^n$ such that $(t\Omega-y)\cap B_1 \supset \mathcal{C}_\delta \cap B_1$. 
   As the support of $u$ is a compact subset of $\mathcal{C}\cap B_{1/2}$, $\mathrm{supp}(u)\subset \mathcal{C}_\delta \cap B_{1/2}$ for sufficiently small $\delta$. Consequently, $v_\delta$ defined by $v_\delta(x)=u(tx-y)$ belongs to $\DD_0^{1,p}(\Omega)$ for all small enough $\delta$.

  Set $x_\delta= t^{-1}(x^*+y)$. By similarity invariance, 
  \begin{equation*}
    d_\Omega(x_\delta)^{p-n}\frac{\|Dv_\delta\|_p^p}{|v_\delta(x_\delta)|^p}=d_{t\Omega-y}(x^*)^{p-n}\frac{\|Du\|_p^p}{|u(x^*)|^p}\,.
  \end{equation*}
  It follows from the definition of $u$ and $x^*$ that
  \begin{equation*}
    d_{t\Omega-y}(x^*)^{p-n}\frac{\|Du\|_p^p}{|u(x^*)|^p}= \frac{d_{t\Omega-y}(x^*)^{p-n}}{d_{\mathcal{C}}(x^*)^{p-n}} \RR_p(\mathcal{C}, u) \leq  \frac{d_{t\Omega-y}(x^*)^{p-n}}{d_{\mathcal{C}}(x^*)^{p-n}} (\lambda_p(\mathcal{C})+\epsilon)\,.
  \end{equation*}
By hypothesis  
  \begin{equation}\label{eq: convergence of dist at special point}
    \lim_{\delta \to 0} d_{t\Omega-y}(x^*) = d_{\mathcal{C}}(x^*)\,,
  \end{equation}
which gives
  \begin{equation}\label{eq: trial funcs 1 blowup}
    \limsup_{\delta \to 0}d_\Omega(x_\delta)^{p-n}\frac{\|Dv_\delta\|_p^p}{|v_\delta(x_\delta)|^p} \leq  \lambda_p(\mathcal{C}) + \epsilon\,.
  \end{equation}
Moreover, the variational characterization of $w_{x_\delta}^\Omega$ implies that 
  \begin{equation}\label{eq: trial funcs 2 blowup}
    \limsup_{\delta \to 0}d_\Omega(x_\delta)^{p-n}\|Dw_{x_\delta}^\Omega\|_p^p \leq  \lambda_p(\mathcal{C})+\epsilon\,.
  \end{equation}

  We split the remainder of this proof into two cases depending on the asymptotic behavior of $t$ and $y$ as $\delta\rightarrow 0$.

  \medskip

  {\noindent\it Case 1:}
  If along a sequence $\{\delta_k\}_{k\geq 1}$ with $\lim_{k\to \infty}\delta_k =0$ it holds that
  \begin{equation}\label{eq: divergence condition}
    \limsup_{k\to \infty} t(\delta_k)=\infty\,,\quad  
    \liminf_{k\to \infty} t(\delta_k)=0\,,\quad  \mbox{or}\quad
    \limsup_{k\to \infty} |y(\delta_k)|= \infty\,,
  \end{equation}
  then we claim that a subsequence of $\{x_{\delta_k}\}_{k\geq 1}$ belongs to $\mathcal{Y}_\Omega$. Thus along this subsequence $w^\Omega_{x_{\delta_{k}}}$ is admissible in the definition of $\Lambda_p(\Omega)$, which when combined with~\eqref{eq: trial funcs 2 blowup} implies $$\Lambda_p(\Omega) \leq \lambda_p(\mathcal{C})+\epsilon\,.$$  Since $\epsilon$ was arbitrary, we would then conclude our proof in this case. 

  To prove the claim, we argue as follows. By~\eqref{eq: convergence of dist at special point},
  $$
  d_\Omega(x_{\delta_k}) = t(\delta_k)^{-1}d_{t(\delta_k)\Omega-y(\delta_k)}(x^*) = t(\delta_k)^{-1}(d_{\mathcal{C}}(x^*)+o(1))\,
  $$
  as $k\rightarrow\infty$. Therefore, if $\liminf_{k\to \infty}t(\delta_k)=0$, then $\limsup_{k\to \infty}|x_{\delta_k}|= \infty$, and a subsequence of $\{x_{\delta_k}\}_{k\geq 1}$ belongs to $\mathcal{Y}_\Omega$. If instead $\limsup_{k\to \infty}t(\delta_k)=\infty$, then $\liminf_{k\to \infty}d_\Omega(x_{\delta_k}) = 0$ and again a subsequence of $\{x_{\delta_k}\}_{k\geq 1}$ belongs to $\mathcal{Y}_\Omega$. Finally if there exist $c,C>0$ so that $c<t(\delta_k)<C$ for all $k$ but $\limsup_{k\to \infty}|y(\delta_k)|=\infty$, then
  $$\limsup_{k\to \infty}|x_{\delta_k}| = \limsup_{k\to \infty} |t(\delta_k)^{-1}(x^*+y(\delta_k))| =\infty\,.$$
  Again we deduce that $\{x_{\delta_k}\}_{k\geq 1}$ belongs to $\mathcal{Y}_\Omega$.

  \medskip

  {\noindent\it Case 2:}
  If~\eqref{eq: divergence condition} fails, then there exists $t_0>0, y_0\in \R^n$, and a sequence $\{\delta_k\}_{k\geq 1}$ such that
  \begin{equation*}
    \lim_{k \to \infty}\delta_k=0\,, \quad \lim_{k\to \infty} t(\delta_k) = t_0\,, \quad \mbox{and} \quad \lim_{k\to \infty}y(\delta_k) = y_0\,.
  \end{equation*}
 As $\lim_{\delta\to0}d_{t\Omega-y}(x) = d_{\mathcal{C}}(x)$ for $x \in  \mathcal{C}\cap B_{1/2}$,
  \begin{equation*}
   \lim_{k\to \infty}d_{t(\delta_k)\Omega-y(\delta_k)}(x)= d_{t_0\Omega-y_0}(x) = d_{\mathcal{C}}(x) \quad \mbox{ for }x\in  \mathcal{C}\cap B_{1/2}\,.
  \end{equation*}
Here we used that $d_\Omega$ is continuous and $d_{t\Omega-y}(x) = td_\Omega (t^{-1}(x+y))$. As a result, $\mathcal{C}\cap B_{1/2}\subset t_0\Omega-y_0$. 

\par Since $\mathcal{C}$ is dilation invariant, the function defined by $v_s(x) = u((t_0x-y_0)/s)$ belongs to $\DD_0^{1,p}(\Omega)$ for any $s\in (0, 1)$. Moreover, we claim that
  \begin{equation*}
    \biggl\|\frac{v_s}{d_\Omega^{1-n/p}}\biggr\|_\infty = \frac{|v_s(x_s)|}{d_\Omega(x_s)^{1-n/p}}
    \quad \text{with}\quad  x_s= \frac{sx^*+y_0}{t_0}
  \end{equation*}
  for $s\in (0,1)$. To see this, recall that $d_\mathcal{C}= d_{t_0\Omega-y_0}$ in $B_{1/2} \cap\mathcal{C}$, $\mathrm{supp}(u(\cdot /s)) \subset \mathcal{C}\cap B_{s/2}$, and $d_{\mathcal{C}}(x/s) = s^{-1}d_\mathcal{C}(x)$ for all $x\in \mathrm{supp}(u(\cdot/s))$. It follows that
  \begin{align*}
    \frac{|v_s(x)|}{d_\Omega(x)^{1-n/p}} 
    &= t_0^{1-n/p}\frac{|u((t_0x-y_0)/s)|}{d_{t_0\Omega-y_0}((t_0x-y_0))^{1-n/p}}\\
    &= (t_0/s)^{1-n/p}\frac{|u((t_0x-y_0)/s)|}{d_{\mathcal{C}}((t_0x-y_0)/s)^{1-n/p}}\\
    &\leq (t_0/s)^{1-n/p} \frac{|u(x^*)|}{d_{\mathcal{C}}(x^*)^{1-n/p}}\,,
  \end{align*}
  with equality for $x= x_s$. Thus,
  \begin{equation}\label{eq: trial function second blow up}
    \RR_p(\Omega, v_s) = \RR_p(\mathcal{C}, u) \leq \lambda_p(\mathcal{C})+\epsilon\,.
  \end{equation}

Notice that $y_0/t_0=\lim_{s\rightarrow 0}x_s\in \partial\Omega$ since
  $$
  \lim_{s\to 0}d_\Omega(x_s) = t_0^{-1}\lim_{s\to 0}d_{t_0\Omega-y_0}(sx^*) = t_0^{-1}\lim_{s\to 0} d_{\mathcal{C}}(sx^*) = t_0^{-1}\lim_{s\to 0}s d_{\mathcal{C}}(x^*)=0\,.
  $$
  Therefore, along any sequence $\{s_k\}_{k\geq 1}$ with $s_k\to 0$ the sequence $\{x_{s_k}\}_{k\geq 1} \in \mathcal{Y}_\Omega$. 
  Consequently, the potentials $\{w^\Omega_{x_{s_{k}}}\}_{k\geq1}$ are admissible in the definition of $\Lambda_p(\Omega)$. 
Combining this observation with Proposition~\ref{prop: potentials are enough} and~\eqref{eq: trial function second blow up}, we have
  \begin{equation*}
    \Lambda_p(\Omega) \leq \liminf_{k\to \infty}d_\Omega(x_{s_k})^{p-n}\|Dw^\Omega_{x_{s_{k}}}\|_p^p\leq \liminf_{k\to \infty}\RR_p(\Omega, v_{s_k}) \leq  \lambda_p(\mathcal{C})+\epsilon\,.
  \end{equation*}
  Since $\epsilon$ was arbitrary, this completes our proof.
\end{proof}
The most evident case when Lemma~\ref{lem: blow-up upper bound} can be applied is if the boundary has at least one point at which it is differentiable. More 
generally it holds when $\partial\Omega$ is asymptotically a cone at some boundary point.
\begin{cor}\label{cor: boundaryblowup}  
Assume
$$
\Omega\cap B_r=\{x=(x',x_n)\in B_r: x_n>f(x')\}
$$
for some $r>0$, where $f\colon \R^{n-1}\rightarrow \R$ is continuous with $f(0)=0$. Further suppose 
the limit 
$$
F(x')=\lim_{t\to\infty}tf(x'/t)
$$
exists locally uniformly. Then $\Lambda_p(\Omega)\le\lambda_p(\mathcal{C})$,
where 
$$
\mathcal{C}=\{x=(x',x_n)\in \R^n: x_n>F(x')\}\,.
$$ 
\end{cor}

\begin{proof} As $F$ is positively homogeneous and continuous, $\mathcal{C}$ is an open cone. Therefore, it suffices to check the two hypotheses of the previous lemma for each $\delta\in (0,1)$. 

\par By assumption, there is $t>r^{-1}$ with 
$
F(x')-\delta<t f(x'/t)<F(x')+\delta
$
uniformly in $|x'|\le 1$. Since 
$$
t\Omega\cap B_1=\{(x',x_n)\in B_1: x_n>t f(x'/t)\}\,,
$$
it follows that 
\begin{equation}\label{ex: blowup limit inclusion}
(\mathcal{C}+\delta e_n)\cap B_1\subset t\Omega\cap B_1\subset (\mathcal{C}-\delta e_n)\cap B_1.
\end{equation}
Observe that if $d_{\mathcal{C}}(x)>\delta$, then $\overline{B_\delta}(x)\subset {\mathcal{C}}$. In this case, $x-\delta e_n\in \overline{B_\delta}(x)\subset {\mathcal{C}}$ and thus $x\in  {\mathcal{C}}+\delta e_n.$ In view of the first inclusion in~\eqref{ex: blowup limit inclusion}, 
$$
\{x\in \R^n: d_{\mathcal{C}}(x)>\delta\}\cap B_1\subset t\Omega\cap B_1\,.
$$
This verifies the first hypothesis of the lemma. 

\par For the remainder of this proof, fix $x\in {\mathcal{C}}$ with $|x|<1/2$. We claim that
$$
d_{t\Omega\cap B_1}(x)=d_{t\Omega}(x) 
$$
To see this, we note 
$d_{t\Omega\cap B_1}(x)\le |x|<1/2$ since $0\notin t\Omega $. Moreover, if the distance from $x$ to the complement is realized for some $y\not\in B_1$, then $d_{t\Omega\cap B_1}(x)=|x-y|\ge 1-|x|>1/2.$  The claim follows. As $0\notin \mathcal{C}+\delta e_n$, we also conclude 
$$
d_{(\mathcal{C}+\delta e_n)\cap B_1}(x)=d_{(\mathcal{C}+\delta e_n)}(x)\,.
$$

\par Based on the first inclusion in~\eqref{ex: blowup limit inclusion} and the observations just made, 
$$
d_{t\Omega}(x)\ge d_{\mathcal{C}+\delta e_n}(x)\ge d_{\mathcal{C}}(x)-\delta\,.
$$
Moreover, by the second inclusion in~\eqref{ex: blowup limit inclusion}
$$
d_{t\Omega}(x)\le d_{(\mathcal{C}-\delta e_n)\cap B_1}(x)\le d_{\mathcal{C}-\delta e_n}(x)\le  d_{\mathcal{C}}(x)+\delta\,. 
$$ 
We deduce that the second hypothesis of the lemma holds as $|d_{t\Omega}(x)- d_{\mathcal{C}}(x)|\le \delta$.
\end{proof}
\begin{rem}
The corollary holds under the weaker assumption that  
$$
F(x')=\lim_{k\to\infty}t_k f(x'/t_k)
$$
locally uniformly for some sequence $t_k\to\infty$ provided that $F$ is positively homogeneous. 
\end{rem}
\begin{rem}\label{rem: blow-up at a smooth point remark}
In the above corollary, if $f$ is differentiable at $0$, then $F(x')=Df(0)\cdot x'$. In this case, $\mathcal{C}$ is a halfspace. We conclude that $\Lambda_p(\Omega)\le \lambda_p(\R^n_\limplus)$ whenever $\partial\Omega$ has at least one point where it is differentiable.
\end{rem}

Now we turn our attention to a lower bound on $\Lambda_p$ and the behavior of sequences of potentials associated with sequences belonging to $\mathcal{Y}_\Omega$. In the statement below, we will use the notation $T_{r,Q,y}(x)=rQx+y$ for a similarity transformation as discussed in subsection~\ref{sec: siminv}.

\begin{prop}\label{prop: blow-up lower bound on Lambda}
  Assume $\{x_k\}_{k\geq1}\subset \Omega$ and $\{y_k\}_{k\geq1}\subset \partial\Omega$ satisfy
  \begin{equation*}
    |y_k -x_k|= d_\Omega(x_k)\,
      \end{equation*}
  and choose $Q_k \in O(n)$ with $Q_k(e_n) = (x_k-y_k)/d_\Omega(x_k)$ for each $k\ge 1$. Then
  \begin{equation*}
    v_k = 
    w_{x_k}^\Omega \circ T_{d_\Omega(x_k), Q_k, y_k}\,,
  \end{equation*}
 defines a bounded sequence in  $\DD^{1,p}_0(\R^n\setminus\{0\})$.
  If $v_k \rightharpoonup v$ in $\DD^{1,p}_0(\R^n\setminus\{0\})$ then $\{v >0\}\subset \R^n$ is open, 
  $$
  B_1(e_n)\subset \{v>0\} \subset \R^n\setminus\{0\}\,,
  $$
  and 
  \begin{equation*}
    \liminf_{k \to \infty} d_\Omega(x_k)^{p-n}\|Dw^\Omega_{x_k}\|_p^p \geq \|Dv\|_p^p\geq \lambda_p(\{v>0\})\,.
  \end{equation*}
\end{prop}

\begin{proof}
As $w(x) = (1-|x-x_k|/d_\Omega(x_k))_\limplus$  is a competitor for $w_{x_k}^\Omega$, 
  \begin{equation*}
    \|Dw_{x_k}^\Omega\|_p^p \leq |B_1|d_\Omega(x_k)^{n-p}\,.
  \end{equation*}
  We also have 
  \begin{equation*}
  \|Dv_{k}^\Omega\|_p^p =   d_\Omega(x_k)^{p-n} \|Dw_{x_k}^\Omega\|_p^p \geq \RR_p(\Omega, w_{x_k}^\Omega) \geq  \lambda_p(\Omega)
  \end{equation*}
by Proposition~\ref{prop: potentials are enough}.
Therefore, the sequence $\{v_k\}_{k\geq 1}$ is bounded in $\DD^{1,p}_0(\R^n\setminus\{0\})$ with
  \begin{equation*}
    \lambda_p(\Omega)^{1/p} \leq\|Dv_k\|_p \leq |B_1|^{1/p}\,.
  \end{equation*}

  Assume that $v_k \rightharpoonup v$ in $\DD^{1,p}_0(\R^n\setminus\{0\})$. By the weak lower semi-continuity of the Dirichlet energy,
  \begin{equation*}
     \liminf_{k\to \infty}d_\Omega(x_k)^{p-n}\|Dw_{x_k}^\Omega\|_p^p =\liminf_{k\to \infty}\|Dv_k\|_p^p \geq \|Dv\|_p^p\,.
  \end{equation*} 
This proves the first inequality in the proposition.

  Since we always identify $v\in \DD^{1,p}(\R^n)$ with its H\"older continuous representative the set $\{v>0\}$ is open.
  To see that $B_1(e_n)\subset \{v>0\}\subset \R^n \setminus \{0\}$ we observe that for each $k\geq 1$, $v_k$ is $p$-superharmonic in $B_1(e_n)$, nonnegative and satisfies $v_k(e_n)=1$. By passing to a subsequence if necessary, we may assume $v_k \to v$ uniformly on compact subsets, by the Arzel\`a--Ascoli theorem and Morrey's inequality. Therefore, the above properties hold true also for $v$. The strong minimum principle implies that $v>0$ in $B_1(e_n)$.
  
  As $0 \in \partial\{v>0\}$ and $B_1(e_n)\subset \{v>0\}$, it follows that $d_{\{v>0\}}(e_n)=1$. By combining these two observations,
  \begin{equation*}
    \|Dv\|_p^p= \biggl(\frac{v(e_n)}{d_{\{v>0\}}(e_n)^{1-n/p}}\biggr)^{-p}\|Dv\|^p_p \ge \biggl\|\frac{v}{d_{\{v>0\}}^{1-n/p}}\biggr\|^{-p}_{\infty}\|Dv\|^p_{p} \geq  \lambda_p(\{v>0\}) \,.
  \end{equation*}
  This completes the proof of the proposition.
\end{proof}

The application of the previous result gives the following lemma.
\begin{lem}\label{lem: potentials concentrating at regular boundary point}
  Assume $n\geq 2$. Suppose further that $\{x_k\}_{k\geq 1}\in \mathcal{Y}_\Omega, x_0 \in \partial \Omega$ and $r>0$ are such that 
  \begin{equation*}
    \lim_{k\to \infty} x_k = x_0
  \end{equation*}
  and that $\partial \Omega \cap B_r(x_0)$ is $C^1$ regular or $\Omega \cap B_r(x_0)$ is convex. Then
  \begin{equation*}
    \liminf_{k\to \infty} d_\Omega(x_k)^{p-n}\|Dw_{x_k}^\Omega\|_p^p \geq  \lambda_p(\R^n_\limplus)\,.
  \end{equation*}
\end{lem}

\begin{proof}
  Let $\{x_k\}_{k\geq 1}$ be as in the statement of the lemma. By passing to a subsequence we may without loss of generality assume that
  \begin{equation*}
    \liminf_{k\to \infty} d_\Omega(x_k)^{p-n}\|Dw_{x_k}^\Omega\|_p^p = \lim_{k\to \infty} d_\Omega(x_k)^{p-n}\|Dw_{x_k}^\Omega\|_p^p\,.
  \end{equation*}
   Let $\{y_k\}_{k\geq 1} \subset \partial \Omega$ be a sequence such that
  \begin{equation*}
    |x_k-y_k| = d_\Omega(x_k) \quad \mbox{for each }k\geq 1\,,
  \end{equation*}
  and $\{Q_k\}_{k \geq 1} \in O(n)$ be a sequence of rotations such that $Q_k(e_n) = \frac{x_k-y_k}{d_\Omega(x_k)}$. Define 
  $$v_k = w_{x_k}^\Omega \circ T_{d_\Omega(x_k),Q_k,y_k}
  $$
as in Proposition~\ref{prop: blow-up lower bound on Lambda}. We will analyze $\{v>0\}$, where $v$ is a subsequential limit of $v_k$.

\par Performing a translation and rotation, we can assume that $x_0 = 0$ without loss of generality. By assumption, there exists a function $f \colon \R^{n-1} \to \R$ with
  \begin{equation*}
    \Omega \cap B_r = \{x=(x', x_n)\in \R^n: x_n > f(x')\}\cap B_r\,.
  \end{equation*}
  If $\Omega \cap B_r$ is convex, $f$ is convex. While if $\partial\Omega \cap B_r$ is $C^1$, then $f$ is $C^1$. In either case, we may assume that $f$ is Lipschitz.

  Since $\{x_k\}_{k \geq 1}$ converges and $\lim_{k\to \infty}d_\Omega(x_k)=\lim_{k\to \infty}|x_k-y_k|=0$, it follows that $\lim_{k\to \infty} y_k = x_0=0$. 
  In particular, $|y_k|<r/2$ for $k$ large enough and
  \begin{equation*}
    (y_k)_n = f(y_k')\,.
  \end{equation*}
  Therefore, $x \in \Omega \cap B_{r}$ if and only if $|x|<r$ and
  \begin{equation*}
    (x-y_k)_n > f(x')-f(y_k')\,.
  \end{equation*}
  Consequently, if $z \in B_R$ with $R<r/(2d_\Omega(x_k))$ then $y_k + d_\Omega(x_k)Q_k(z)\in \Omega$ if and only if
  \begin{equation*}
    (Q_k(z))_n > \frac{f(d_\Omega(x_k)(Q_k(z))'+y_k')-f(y_k')}{d_\Omega(x_k)}\,.
  \end{equation*}

  Since $\{v_k\}_{k\geq 1}\subset \DD^{1,p}_0(\R^n\setminus\{0\})$ is bounded we may by passing to a subsequence assume that $v_k \rightharpoonup v$ in $\DD^{1,p}_0(\R^n\setminus\{0\})$ and $v_k \to v$ uniformly on compact sets. Since $O(n)$ is compact we may also pass to a further subsequence such that $Q_k \to Q_0$. In particular, this implies that $\tilde v_k :=w_{x_k}^\Omega(d_\Omega(x_k)\cdot+y_k) \rightharpoonup v(Q_0^{-1} \cdot) := \tilde v$. We claim that $\{\tilde v>0\} = Q_0\{v>0\}$ is given by the region above a Lipschitz graph. We first observe that since $w_{x_k}^\Omega$ is a potential
  \begin{equation}\label{eq: vanishing/positivity set}
  \begin{aligned}
    \tilde v_k(z) = 0 \quad &\mbox{if }z_n \leq \frac{f(d_\Omega(x_k)z'+y_k')-f(y_k')}{d_\Omega(x_k)}\,,\\
    \tilde v_k(z) > 0 \quad &\mbox{if }z_n > \frac{f(d_\Omega(x_k)z'+y_k')-f(y_k')}{d_\Omega(x_k)}\,,
  \end{aligned}
  \end{equation}
  for $z \in B_R$.
Moreover, $\tilde v(Q_0 e_n)=1$  and since the $p$-Laplacian is invariant under the action of $O(n)$, $\tilde v$ is $p$-superharmonic whenever $\tilde v>0$. 

\par Since $f$ is Lipschitz, the functions 
  \begin{equation*}
    f_k(z'):=\frac{f(d_\Omega(x_k)z'+y_k')-f(y_k')}{d_\Omega(x_k)}
  \end{equation*}
  are uniformly bounded and Lipschitz for $|z'|<R$ for any fixed $R>0$ provided $k$ is large enough (with a Lipschitz constant independent of $R$). Therefore, we can pass to a subsequence along which $f_k$ converges uniformly to a limit $f_{0, R}$ in $|z'|<R$. It follows that $\tilde v$ is nonnegative, $\tilde v=0$ in $\{z_n < f_{0,R}(z')\}\cap B_R$ and $\tilde v$ is $p$-superharmonic in $\{z_n > f_{0,R}(z')\}\cap B_R$. Since, $\tilde v(Q_0 e_n)=1$, the strong minimum principle implies 
  $$\{z \in \R^n:\tilde v(z)>0, |z|<R\} = \{z\in \R^n: z_n>f_{0,R}(z'), |z|<R\}.$$
  Since $\{\tilde v>0\}$ is independent of $R$ this implies that $f_{0,R}$ is the restriction of some globally Lipschitz function $f_0$ and $\{\tilde v>0\} = \{z_n >f_0(z')\}$.

  If $f$ was convex, then $f_0$ is convex; so $\{\tilde v>0\}$ and thus $\{v>0\}$ is convex. If $f$ was $C^1$ at $0$, then $f_0(z') = Df(0)\cdot z'$; so $\{\tilde v>0\}$ and hence $\{v>0\}$ is a halfspace. In either case $\lambda_p(\{v>0\})= \lambda_p(\R^n_\limplus)$ by Theorem~\ref{thm: convex, punctured, exterior domains}. Appealing to Proposition~\ref{prop: blow-up lower bound on Lambda}, we finally conclude that
  \begin{equation*}
    \lim_{k\to \infty} d_\Omega(x_k)^{p-n}\|Dw_{x_k}^\Omega\|_p^p \geq \lambda_p(\R^n_\limplus)\,. \qedhere
  \end{equation*}
\end{proof}
\begin{rem}
  In the above proof, we only used the regularity assumption on  $\partial \Omega \cap B_r(x_0)$ when concluding that $\lambda_p(\{v>0\})=\lambda_p(\R^n_\limplus)$. Whenever $\partial \Omega \cap B_r(x_0)$ is given by a Lipschitz graph it follows in the same manner that $\{v>0\}$ is the region above a Lipschitz graph. 
\end{rem}
Combining the results of this section we are able to determine $\Lambda_p(\Omega)$ if $\Omega$ is regular.
\begin{cor}\label{cor: potentials concentrating at bdry}
  Assume $n\geq 2$. If $\Omega$ is bounded and $C^1$, then $\Lambda_p(\Omega) = \lambda_p(\R^n_\limplus).$
\end{cor}

\begin{proof}
 Fix a sequence $\{x_k\}_{k\geq 1} \in \mathcal{Y}_\Omega$ such that 
  \begin{equation*}
    \lim_{k\to \infty} d_\Omega(x_k)^{p-n}\|Dw_{x_k}^\Omega\|_p^p = \Lambda_p(\Omega)\,.
  \end{equation*}
 Since $\overline{\Omega}$ is compact, we may assume that $\{x_k\}_{k\geq 1}$ converges to some limit $x_0$. As $\{x_k\}_{k\geq 1}\in \mathcal{Y}_\Omega$, it follows that $x_0\in \partial\Omega$. By Lemma~\ref{lem: potentials concentrating at regular boundary point}, $\Lambda_p(\Omega) \geq \lambda_p(\R^n_\limplus)$. 
In view of Corollary~\ref{cor: boundaryblowup} (and Remark~\ref{rem: blow-up at a smooth point remark}), we also have 
 $\Lambda_p(\Omega)\leq \lambda_p(\R^n_\limplus)$.
\end{proof}
\begin{rem}\label{cor: universal bounds Lambda} While the equality in Corollary~\ref{cor: potentials concentrating at bdry} is not true in general, it is possible to prove an analogue of Theorem~\ref{thm: universal sharp bounds} for $\Lambda_p$. Namely, for any $\Omega \subsetneq \R^n$,
 $$\lambda_p(\R^n \setminus \{0\})\leq \Lambda_p(\Omega)\leq \lambda_p(\R^n_\limplus)\,.$$
 Indeed, the lower bound follows directly from the general fact that $\lambda_p(\Omega)\leq \Lambda_p(\Omega)$ and the lower bound in Theorem~\ref{thm: universal sharp bounds}. The upper bound can be obtained by transplanting a sequence of potentials realizing $\Lambda_p(B_1) = \lambda_p(\R^n_\limplus)$ into $\Omega$ in the spirit of our proof of Corollary~\ref{cor: Universal bounds}.
\end{rem}

We are now ready to prove Theorem~\ref{thm: strict energy ineq implies existence} by combining our previous results.

\begin{proof}[Proof of Theorem~\ref{thm: strict energy ineq implies existence}]
  By Corollary~\ref{cor: potentials concentrating at bdry}, $\Lambda_p(\Omega)= \lambda_p(\R^n_\limplus)$. Therefore, the assumption of this theorem implies that $\lambda_p(\Omega)<\Lambda_p(\Omega)$. Proposition~\ref{prop: Compactness threshold} yields the desired conclusion.
\end{proof}


\section{Refined analysis at points of negative mean curvature}
\label{sec: Refined blow-up neg mean curvature}

The aim of this section is to prove that if $\partial\Omega$ has a point $x_0$ where the mean curvature is negative then $\lambda_p(\Omega)<\lambda_p(\R^n_\limplus)$. The idea is to take a sequence of potentials whose singular point approaches $x_0$ from within $\Omega$ and show that along this sequence $\RR_p$ is eventually smaller than when it is evaluated on a potential in $\R^n_\limplus$. We begin by considering a model situation in which $\partial\Omega$ is a parabola in a neighborhood of $x_0$. We will assume throughout this section that $p>n\ge 2$ and $\Omega \subsetneq \R^n$ is open.

Let $K \in \R^{(n-1)\times (n-1)}$ be a real symmetric matrix and denote its operator norm by $\|K\|$. Define
\begin{equation*}
  \Omega_K = \{x= (x', x_n)\in \R^n : x_n > x'\cdot Kx', |x|<1\}\,,
\end{equation*}
and set $u_{K,\epsilon}= w^{\Omega_K}_{\epsilon e_n}$ for $\epsilon\in (0, 1/2)$. Recall that $u_{K,\epsilon}\in \DD^{1,p}_0(\Omega_K)$ is the solution of 
\begin{equation}\label{eq: def uKeps}
   \begin{cases}
    - \Delta_p u_{K,\epsilon} =0 & \mbox{in } \Omega_K \setminus\{\epsilon e_n\}\,,\\
     \hspace{.33in}   u_{K,\epsilon} =0 & \mbox{on }\partial \Omega_K\,,\\
       u_{K,\epsilon}(\epsilon e_n) =1\,.
   \end{cases}
\end{equation} 
We will also write $u_0 = w^{{\R^n_{\footnotesize\limplus}}}_{e_n}$ and use that $u_0 \in  \DD^{1,p}_0(\R^n_\limplus)$ is characterized as the solution of  
\begin{equation}\label{eq: def u0}
   \begin{cases}
    - \Delta_p u_0 =0 & \mbox{in }\R^n_\limplus \setminus\{e_n\}\,,\\
     \hspace{.33in}  u_0 =0 & \mbox{on }\partial \R_\limplus^n\,,\\
     \hspace{.06in}  u_0(e_n) =1\,.
   \end{cases}
\end{equation} 

The crucial ingredient in our proof of Theorem~\ref{thm: neg mean curvature} is the following proposition.
\begin{prop}\label{prop: mean curvature expansion}
 Assume $K \in \R^{(n-1)\times (n-1)}$ is symmetric with $\|K\|<1/2$. There are constants $C>0$ and $\gamma \in (0, 1]$ so that
\begin{equation*}
  \epsilon^{p-n}\|Du_{K,\epsilon}\|_p^p \leq \|Du_0\|_p^p + \epsilon\,  \mathrm{tr}(K)\frac{p-1}{n-1} \int_{\R^{n-1}}|Du_0(x', 0)|^{p}|x'|^2\,dx' + C\epsilon^{1+\gamma}
\end{equation*}
for $\epsilon \in (0, 1/8)$.
\end{prop}

\begin{figure}[ht]
\centering
 \input{NegMeanCurvature}
 \caption{A depiction of the geometric assumptions in Theorem~\ref{thm: neg mean curvature implies energy estimate}. Here $\Omega$ is the set above the black curve, $x_0$ is the origin, $Q$ the identity, and $\partial\Omega$ can be touched at $x_0$ from the inside with a negatively curved parabola (blue).}
 \label{fig:NegCurvature}
\end{figure}
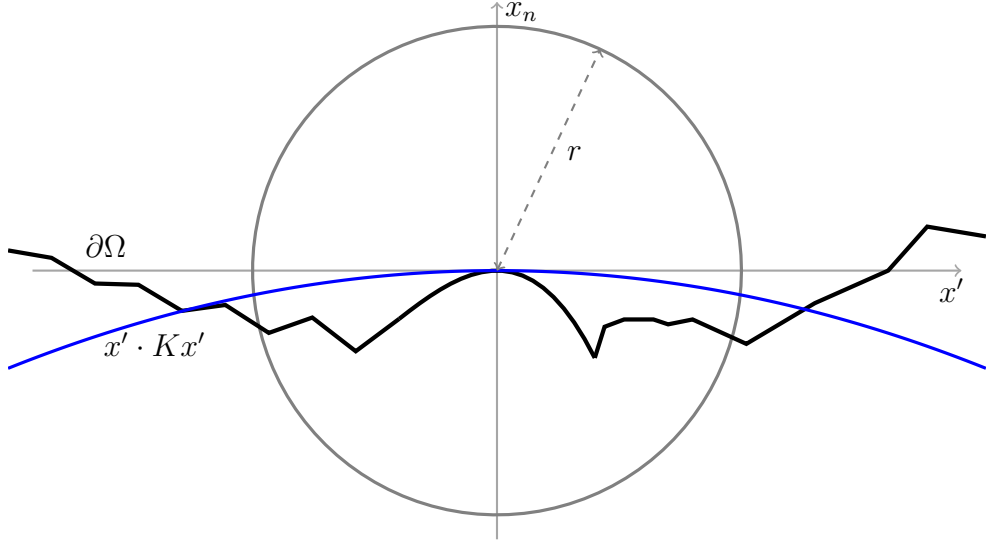

With this proposition in hand we can prove the following theorem.
\begin{thm}\label{thm: neg mean curvature implies energy estimate}
 Assume that there exist $x_0\in \partial \Omega, Q \in O(n), r>0$, and a symmetric matrix $K \in \R^{(n-1)\times (n-1)}$ with $\mathrm{tr}(K)<0$ such that
  \begin{equation*}
    \{ x = (x', x_n)\in \R^n: x_n> x'\cdot Kx'\} \cap B_r \subset Q(\Omega-x_0)\,.
  \end{equation*}
  Then $\lambda_p(\Omega)<\lambda_p(\R^n_\limplus)$.
\end{thm}

\begin{proof}
  Since $\lambda_p(\Omega) = \lambda_p(Q(\Omega-x_0))$, we may assume without loss of generality that $x_0=0$ and $Q=\Id$. Also observe that if the assumption of the theorem holds for some $r>0$, then it is also valid for any smaller value of $r$. As such, we may assume that
  \begin{equation}\label{r smallness assumption}
r\|K\|< \frac{1}{2}\,. 
  \end{equation}
 \par It suffices to construct $w\in \DD^{1,p}_0(\Omega)$ which satisfies $\RR_p(\Omega, w)<\lambda_p(\R^n_\limplus)$. 
  With this in mind, we define $w_\epsilon \in \DD_0^{1,p}(r\Omega_{rK})$ via $w_\epsilon(x) = u_{rK,\epsilon/r}(x/r)$. As
  \begin{equation*}
    r\Omega_{rK} = \{x=(x', x_n) \in \R^n: x_n > x'\cdot Kx'\}\cap B_r \subset \Omega\,,
  \end{equation*}
  we also have $w_\epsilon \in \DD^{1,p}_0(\Omega)$. Moreover,  $\|Dw_\epsilon\|_p = r^{n/p-1}\|Du_{rK,\epsilon/r}\|_p$. And since $w_\epsilon(\epsilon e_n) =1$ for $\epsilon <r$ and $0\in \partial\Omega$,
  \begin{equation*}
    \Bigl\|\frac{w_\epsilon}{d_\Omega^{1-n/p}}\Bigr\|_\infty \geq \frac{|w_\epsilon(\epsilon e_n)|}{d_\Omega(\epsilon e_n)^{1-n/p}} \geq \epsilon^{n/p-1}\,.
  \end{equation*}
 Consequently,
  \begin{equation*}
    \RR_p(\Omega, w_\epsilon) \leq (\epsilon/r)^{p-n}\|Du_{rK,\epsilon/r}\|_p^p\,.
  \end{equation*}
 
 \par In view of~\eqref{r smallness assumption}, $\|rK\|<1/2$, so Proposition~\ref{prop: mean curvature expansion} can be applied. As $\|Du_0\|_p^p = \lambda_p(\R^n_\limplus)$,
  \begin{equation*}
    \RR_p(\Omega, w_\epsilon) \leq \lambda_p(\R^n_\limplus) + \epsilon\,\mathrm{tr}(K) \frac{p-1}{n-1}\int_{\R^{n-1}}|Du_0(x', 0)|^p |x'|^2\,dx' + C (\epsilon/r)^{1+\gamma}
  \end{equation*}
for all $\epsilon$ sufficiently small. By Lemma~\ref{lem: half- and punctured space}, $u_0$ is the restriction of a Morrey extremal to $\R^n_\limplus$. It then follows that $|Du_0(x',0)|>0$ for $x'\in \R^{n-1}$~\cite[Proposition 3.6]{MR4275749}. As a result, 
  \begin{equation*}
    \int_{\R^{n-1}}|Du_0(x', 0)|^p |x'|^2\,dx'>0\,.
  \end{equation*}
 In addition, this integral is finite by the decay estimate proved in~\cite{HyndLarsonLindgren-Decay}. Finally, as $\mathrm{tr}(K)<0$ and $\gamma >0$, we can choose $\epsilon$ sufficiently small so that $\RR_p(\Omega, w_\epsilon)< \lambda_p(\R^n_\limplus)$. 
\end{proof}

We are now ready to prove Theorem~\ref{thm: neg mean curvature}. Let us briefly recall the formula for the mean curvature of a surface in $\R^n$ given by the graph of $C^2$ function $f\colon\R^{n-1}\rightarrow \R$. The mean curvature $H$ at $x=(x',f(x'))$ is given by
$$
(n-1)H=\text{div}\biggl(\frac{Df}{\sqrt{1+|Df|^2}}\biggr)\,.
$$ 
This is the mean curvature with respect to the unit normal  $(Df, -1)/\sqrt{1+|Df|^2}$ at $x$, and the derivatives of $f$ are evaluated at $x'$. Also note that $(n-1)H=\Delta f$ at $x=(x',f(x'))$ whenever $Df(x')=0$. 
\begin{proof}[Proof of Theorem~\ref{thm: neg mean curvature}]
  By assumption, $ \Omega$ is $C^2$ and there is $x_0\in \partial\Omega$ which has negative mean curvature. After translating and rotating $\Omega$, we may assume $x_0=0$ and  
 \begin{equation}\label{graph equation proof theorem neg mean curv}
\{x=(x',x_n)\in \R^n: x_n>f(x')\}\cap B_r=\Omega\cap B_r
\end{equation}
 for some $r>0$. Here $f\colon \R^{n-1}\to \R$ is $C^2$ when $|x'|<r$ and satisfies that $f(0)=|Df(0)|=0$, and $\Delta f(0)<0$. 
 
 \par Fix $\epsilon$ so small that $\mathrm{tr}(K)<0$, where
 $$
 K=\frac{1}{2}D^2f(0)+\epsilon\,\Id'\,.
 $$
 Here $\Id'$ is the $(n-1)\times (n-1)$ identity matrix. 
 Reducing $r$ if necessary, $f(x')\le Kx'\cdot x'$ for $|x'|<r$ by Taylor's theorem. In view of~\eqref{graph equation proof theorem neg mean curv},
    $$
\{ x = (x', x_n)\in \R^n: x_n> x'\cdot Kx'\} \cap B_r\subset \Omega\,.
 $$
We conclude that $\lambda_p(\Omega)<\lambda_p(\R^n_\limplus)$ by Theorem~\ref{thm: neg mean curvature implies energy estimate}. Theorem~\ref{thm: strict energy ineq implies existence} in turn implies that $\Omega$ has an extremal. 
\end{proof}

\begin{rem}
 The assumptions in Theorem~\ref{thm: neg mean curvature} can be weakened significantly. It suffices to assume that
  \begin{enumerate}
    \item $\Omega$ is bounded,
    \item that after an appropriate rotation and translation $0 \in \partial \Omega$ and there exists $r>0$ so that 
    \begin{equation*}
        \Omega \cap B_r \supset \{x =(x', x_n)\in \R^n : x_n > x'\cdot Kx'\}\cap B_r 
      \end{equation*}
      where $K \in \R^{(n-1)\times(n-1)}$ is a symmetric matrix with $\mathrm{tr}(K)<0$, and
    \item that for every $x \in \partial \Omega$ there exists $r>0$ so that either
    \begin{enumerate}
        \item $\Omega \cap B_r(x)$ is convex, or 
        \item $\partial\Omega \cap B_r(x)$ is $C^1$-regular.
  \end{enumerate}
  \end{enumerate}
  Indeed, under these assumptions Lemma~\ref{lem: potentials concentrating at regular boundary point} implies that $\Lambda_p(\Omega)\geq \lambda_p(\R^n_\limplus)$. By Theorem~\ref{thm: neg mean curvature implies energy estimate}, $\lambda_p(\Omega)<\lambda_p(\R^n_\limplus)$. Then Proposition~\ref{prop: Compactness threshold} yields the existence of an extremal.
\end{rem}

\subsection{Proof of Proposition~\ref{prop: mean curvature expansion}}
Proposition~\ref{prop: mean curvature expansion} is a direct consequence of the following two lemmas. 

\begin{lem}\label{lem: mean curvature expansion}
Assume $K \in \R^{(n-1)\times (n-1)}$ is symmetric with $\|K\|<1/2$. There are constants $C>0$ and $\gamma \in (0, 1]$ so that
\begin{equation*}
  \epsilon^{p-n}\|Du_{K,\epsilon}\|_p^p \leq \|Du_0\|_p^p - \epsilon\frac{2p}{n-1}\mathrm{tr}(K) \int_{\R^n_\limplus}|Du_0(x)|^{p-2}\partial_{x_n} u_0(x)Du_0(x)\cdot (x', 0)\,dx + C\epsilon^{1+\gamma}\,
\end{equation*}
for $\epsilon \in (0, 1/8)$.
 \end{lem} 

\begin{lem}\label{lem: integration by parts identity}
 The equality 
  \begin{equation*}
    \int_{\R^n_\limplus}|Du_0(x)|^{p-2}\partial_{x_n} u_0(x)Du_0(x)\cdot (x', 0)\,dx = -\frac{p-1}{2p}\int_{\R^{n-1}} |Du_0(x', 0)|^{p} |x'|^2\,dx'\,
  \end{equation*}
holds and both integrals are convergent.
\end{lem}

\begin{proof}[Proof of Lemma~\ref{lem: mean curvature expansion}]
We prove this lemma by exploiting the variational characterization of $u_{K,\epsilon}$. In particular, we will derive the desired inequality by constructing an appropriate competitor for $u_{K,\epsilon}$. 
\medskip

{\it Step 1.}  For $y=(y',y_n)\in \R^n$, define
\begin{equation}
  \Phi(y) = (y', y_n-y'\cdot K y')\,.
\end{equation}
This mapping is a diffeomorphism from $\R^n$ to itself, $\Phi(t e_n)=t e_n$ for all $t\in \R$, 
\begin{equation}
  \Phi(\{y\in \R^n: y_n>y'\cdot K y'\}) = \R^n_\limplus\,, \quad \mbox{and} \quad \Phi^{-1}(x) = (x', x_n+x'\cdot Kx')\,.
\end{equation}
As
\begin{equation}\label{JacobianPhi}
  D\Phi(y) = \biggl(\begin{matrix}
    \Id & 0\\
    -2(Ky')^\top & 1
  \end{matrix}\biggr),
\end{equation}
it also follows that  ${\det D\Phi(y)}=1$ for all $y \in \R^n$. 

In addition, we claim  
\begin{equation}\label{PhiClaim}
|\Phi(y)| \geq 1/2 \mbox{ for all } y \in B_{1}^c\,.
\end{equation}
We first note that if $|y|\geq 1$ and $|y'|\geq  1/2$, then
 \begin{equation*}
   |\Phi(y)|^2 =  |y'|^2 +(y_n -y'\cdot Ky')^2\geq |y'|^2\geq  1/4\,.
 \end{equation*}
 Next, we observe that if $|y| \geq 1$ and $|y'|< 1/2$, then
 $$
 |y_n|^2 =|y|^2-|y'|^2 \geq 3/4\,.
 $$ 
As $\|K\|< 1/2$ and $t^2-t/4$ is increasing for $t\ge \sqrt{3}/2$, 
   \begin{align*}
   |\Phi(y)|^2 &= |y'|^2 +(y_n -y'\cdot Ky')^2 \\
                &\geq |y_n|^2- 2y_n y'\cdot Ky' \\                   
                 &\geq |y_n|^2-\frac{1}{4} |y_n|\\
                &\geq \Bigl(\frac{3}{4}\Bigr)-\frac{1}{4} \Bigl(\frac{\sqrt{3}}{2}\Bigr)\\
              &>\frac{1}{4}\,.
 \end{align*} 
We conclude that~\eqref{PhiClaim} holds. 

\par Set
$$
\phi(s) = \begin{cases}
  0 & \mbox{if }s\leq 0\,,\\
  s & \mbox{if }0<s<1\,,\\
  1 & \mbox{if }s\geq 1\,,
\end{cases}
$$
and define 
\begin{equation*}
  w(y) = \phi\bigl(2-4|\Phi(y)|\bigr)u_0(\Phi(y)/\epsilon)
\end{equation*}
for $\epsilon\in (0,1/8)$ and $y\in \R^n$. Since $|\Phi(y)|\geq 1/2$ whenever $|y|\geq 1$, the first factor of $w$ vanishes for $|y|\geq 1$. The second factor of $w$ vanishes for all $y$ with $y_n < y'\cdot K y'$. Indeed $\Phi$ maps this set of points to $(\R^n_\limplus)^c$, where $u_0$ vanishes. Therefore, $w$ is supported in $\Omega_K$. 

\par Since the first factor of $w$ is Lipschitz with compact support and the second factor belongs to $ \DD^{1,p}(\R^n)$,  $w \in \DD^{1,p}_0(\Omega_K)$. Further, $w(\epsilon e_n)=\phi(2-4\epsilon)u_0(e_n)=1$ as $0<\epsilon<1/4$. It follows that
\begin{equation*}
  \|Du_{K,\epsilon}\|_p^p \leq \|Dw\|_p^p
\end{equation*}
by the variational characterization of $u_{K, \epsilon}$. We now wish to estimate $\|Dw\|_p^p$ in terms of $u_0$ and $\epsilon$. 

\medskip 

{\it Step 2.} In view of~\eqref{JacobianPhi}, 
\begin{equation}\label{DPhi Inequalities}
D\Phi(\Phi^{-1}(x))^\top z=z-z_n(2Kx',0)\quad \mbox{and} \quad |D\Phi(\Phi^{-1}(x))^\top z|\le |z|(1+|x'|)
\end{equation}
for $x,z\in \R^n$. Using the change of variables $y=\Phi^{-1}(x)$ in the integral below leads to
\begin{equation}\label{eq: first bound Dw integral}
\begin{aligned}
 &\hspace{-15pt} \int_{\R^n} |Dw(y)|^p\,dy\\
  &=
  \int_{B_{1/2}} \bigl|D\Phi(\Phi^{-1}(x))^\top D\bigr(\phi(2-4|x|)u_0(x/\epsilon)\bigr)\bigr|^p\,dx\\
  &=
  \epsilon^{-p}\int_{B_{1/4}} \bigl|D\Phi(\Phi^{-1}(x))^\top Du_0(x/\epsilon)\bigr|^p\,dx\\
  &\quad  + \int_{B_{1/2} \setminus B_{1/4}}\Bigl|D\Phi(\Phi^{-1}(x))^\top \Bigl( -4u_0(x/\epsilon)\frac{x}{|x|}+\frac{1}{\epsilon}(2-4|x|) Du_0(x/\epsilon)\Bigr)\Bigr|^p\,dx\\
  &\leq
  \epsilon^{-p} \int_{B_{1/4}}|Du_0(x/\epsilon)-\partial_{x_n} u_0(x/\epsilon)(2Kx',0)|^p\,dx\\
    &\quad  + (3/2)^p\int_{B_{1/2} \setminus B_{1/4}}\Bigl| -4u_0(x/\epsilon)\frac{x}{|x|}+\frac{1}{\epsilon}(2-4|x|) Du_0(x/\epsilon)\Bigr|^p\,dx\\
      &\leq
  \epsilon^{-p} \int_{B_{1/4}}|Du_0(x/\epsilon)-\partial_{x_n} u_0(x/\epsilon)(2Kx',0)|^p\,dx\\
    &\quad  + (3/2)^p4^{p}2^{p-1}\int_{B_{1/2} \setminus B_{1/4}}|u_0(x/\epsilon)|^p\,dx+(3/2)^p\epsilon^{-p}2^{p-1}\int_{B_{1/2} \setminus B_{1/4}}|Du_0(x/\epsilon)|^p\,dx\,. \hspace{-5pt}
\end{aligned}
\end{equation}
Note that the first inequality is due to~\eqref{DPhi Inequalities} and the second follows from the elementary inequality:  
$(a+b)^p \leq 2^{p-1}a^p+2^{p-1}b^p$ for $a,b\ge 0$.

\par We shall also utilize the estimate 
\begin{equation*}
|z+h|^p\le |z|^p+p|z|^{p-2}z\cdot h+\frac{p(p-1)}{2}|z|^{p-2}|h|^2
\end{equation*}
for $z,h\in \R^n$ with $|h|\le |z|$; see Lemma 10.2.1 in~\cite{MR2401600}, for example. With $z=Du_0(x/\epsilon)$ and $h=-\partial_{x_n} u_0(x/\epsilon)(2Kx',0)$, the above inequality and that $|Kx'|\leq \|K\||x'| <|x'|/2$ gives 
\begin{equation*}
\begin{aligned}
  &\int_{B_{1/4}}|Du_0(x/\epsilon)-\partial_{x_n} u_0(x/\epsilon)(2Kx',0)|^p\,dx\\
  &\quad \le   
    \int_{B_{1/4}}|Du_0(x/\epsilon)|^p\,dx-2p\int_{B_{1/4}}|Du_0(x/\epsilon)|^{p-2}\partial_{x_n} u_0(x/\epsilon)Du_0(x/\epsilon)\cdot(Kx', 0)\,dx\\
  &\quad\quad 
    +\frac{p(p-1)}{2}\int_{B_{1/4}}|Du_0(x/\epsilon)|^p|x'|^2\,dx\\
  &\quad \le  
    \int_{B_{1/4}}|Du_0(x/\epsilon)|^p\,dx-2p\int_{\R^n_\limplus}|Du_0(x/\epsilon)|^{p-2}\partial_{x_n} u_0(x/\epsilon)Du_0(x/\epsilon)\cdot(Kx', 0)\,dx\\
  &\quad\quad 
    +\frac{p(p-1)}{2}\int_{B_{1/4}}\!|Du_0(x/\epsilon)|^p|x'|^2\,dx\\
  &\quad\quad
    +2p\int_{B_{1/4}^c}|Du_0(x/\epsilon)|^{p-2}\partial_{x_n} u_0(x/\epsilon)Du_0(x/\epsilon)\cdot(Kx', 0)\,dx\\
  &\quad \le 
    \int_{\R^n_\limplus}|Du_0(x/\epsilon)|^p\,dx-2p\int_{\R^n_\limplus}|Du_0(x/\epsilon)|^{p-2}\partial_{x_n} u_0(x/\epsilon)Du_0(x/\epsilon)\cdot(Kx', 0)\,dx\\
  &\quad\quad 
    + \frac{p(p-1)}{2}\int_{B_{1/4}}|Du_0(x/\epsilon)|^p|x'|^2\,dx+p\int_{B_{1/4}^c}|Du_0(x/\epsilon)|^p|x'|\,dx\,.
\end{aligned}
\end{equation*}
Changing variables $x\mapsto \epsilon x$ in the integrals above and combining this inequality with the upper bound on $\|Dw\|_p^p$ in~\eqref{eq: first bound Dw integral} gives 
\begin{equation}\label{eq: preliminary upper bound Dw integral}
\begin{aligned}
  &\int_{\R^n} |Dw(y)|^p\,dy \\
  &\quad  \le 
    \epsilon^{n-p}\int_{\R^n_\limplus}|Du_0(x)|^p\,dx-2p\epsilon^{n+1-p}\int_{\R^n_\limplus}|Du_0(x)|^{p-2}\partial_{x_n} u_0(x)Du_0(x)\cdot(Kx', 0)\,dx\\
  &\quad\quad 
    + C_0 \epsilon^{n+2-p}\int_{B_{1/(4\epsilon)}}|Du_0(x)|^p|x'|^2\,dx+ C_0\epsilon^{n+1-p}\int_{B_{1/(4\epsilon)}^c}|Du_0(x)|^p|x'|\,dx\\
  &\quad\quad  
    +C_0 \epsilon^n\int_{B_{1/(2\epsilon)}\setminus B_{1/(4\epsilon)}}|u_0(x)|^p\,dx+C_0\epsilon^{n-p}\int_{B_{1/(2\epsilon)} \setminus B_{1/(4\epsilon)}}|Du_0(x)|^p\,dx\,.
\end{aligned}
\end{equation}
Here $C_0$ is a constant which only depends on $p$. 

Since $u_0$ is the restriction of a Morrey extremal, $u_0$ is axially-symmetric about the $x_n$-axis~\cite[Theorem 1.1]{MR4160015}. Therefore, 
$$
\partial_{x_n} u_0(r \theta, x_n) = \partial_{x_n} u_0(r\theta_0, x_n)\quad\text{ and }\quad (Du_0(r\theta, x_n))'= \theta |(Du_0(r\theta_0, x_n))'|
$$
for any $\theta, \theta_0 \in \mathbb{S}^{n-2}, r>0, x_n>0$.
Furthermore, if $\theta_0 \in \mathbb{S}^{n-2}$ is fixed, then
\begin{equation}\label{eq: neg curvature extracting TrK}
\begin{aligned}
  \hspace{-20pt}\int_{\R^n_\limplus}&|Du_0(x)|^{p-2}\partial_{x_n} u_0(x)Du_0(x)\cdot(Kx', 0)\,dx\\
  &\hspace{-3pt}=
  \int_0^\infty\hspace{-8.5pt} \int_0^{\infty}\hspace{-8.5pt}\int_{\mathbb{S}^{n-2}} \hspace{-3pt}|Du_0(r \theta_0, x_n)|^{p-2}\partial_{x_n} u_0(r\theta_0, x_n)|Du_0(r\theta_0, x_n)'|(\theta\cdot K\theta) r^{n-1}\,d\sigma(\theta)drdx_n \\
  &\hspace{-3pt}=
  \frac{\mathrm{tr}(K)}{n-1}\int_0^\infty\hspace{-8.5pt} \int_0^{\infty}\hspace{-8.5pt} \int_{\mathbb{S}^{n-2}}\hspace{-3pt}|Du_0(r \theta_0, x_n)|^{p-2}\partial_{x_n} u_0(r\theta_0, x_n)|Du_0(r\theta_0, x_n)'||\theta|^2r^{n-1}\,d\sigma(\theta)drdx_n \hspace{-13pt}\\
  &\hspace{-3pt}=
  \frac{\mathrm{tr}(K)}{n-1}\int_{\R^n_\limplus}|Du_0(x)|^{p-2}\partial_{x_n} u_0(x)Du_0(x)\cdot (x', 0)\,dx.
\end{aligned}
\end{equation}
Here $\sigma$ denotes the surface measure and we used the identity 
\begin{equation*}
  \frac{1}{\sigma(\mathbb{S}^{n-2})}\int_{\mathbb{S}^{n-2}}\theta\cdot K\theta\,d\sigma(\theta) = \frac{\mathrm{tr}(K)}{n-1}\,.
\end{equation*}

Combining~\eqref{eq: preliminary upper bound Dw integral},~\eqref{eq: neg curvature extracting TrK}, we arrive at the estimate
\begin{equation}\label{eq: expansion with integral errors}
\begin{aligned}
  \epsilon^{p-n}\|Du_{K,\epsilon}\|_{p}^p &\leq \|Du_0\|_p^p - \epsilon\frac{2p\,\mathrm{tr}(K)}{n-1} \int_{\R^n_\limplus}|Du_0(x)|^{p-2}\partial_{x_n} u_0(x)Du_0(x)\cdot (x', 0)\,dx \\
  &\quad +C_0 \epsilon^{2}\int_{B_{1/(4\epsilon)}} |Du_0(x)|^p|x'|^2\,dx+C_0 \epsilon\int_{B_{1/(4\epsilon)}^c} |Du_0(x)|^p|x'|\,dx\\
  &\quad  + C_0\epsilon^{p}\int_{B_{1/(2\epsilon)}\setminus B_{1/(4\epsilon)}}|u_0(x)|^p\,dx + C_0 \int_{B_{1/(2\epsilon)}\setminus B_{1/(4\epsilon)}}|Du_0(x)|^p\,dx\,.
\end{aligned}
\end{equation}
It remains is to bound the last four terms.

\medskip 
{\it Step 3.} It follows from~\cite{HyndLarsonLindgren-Decay} that for any
\begin{equation}\label{eq: decay beta}
   0<\beta < \beta_0(p):=-\frac{1}{3}+\frac{2}{3(p-1)}+\sqrt{\Bigl(-\frac{1}{3}+\frac{2}{3(p-1)}\Bigr)^2+ \frac{1}{3}}
 \end{equation} 
 there is a constant $C_1$ depending on $\beta$ and $p$ such that
\begin{equation}\label{eq: decay bounds Morrey}
  |u_0(x)|\leq C_1|x|^{-\beta} \quad \mbox{and} \quad |Du_0(x)| \leq C_1 |x|^{-\beta -1}\,
\end{equation}
 for all $|x|\geq 2$. Note that $[2, \infty) \ni p \mapsto \beta_0(p)$ is decreasing. As $\lim_{p\to \infty}\beta_0(p)=\frac{1}{3}$, it must be that $\beta_0(p)>1/p$ for all $p>3$. Moreover, the monotonicity of $\beta_0$ implies that 
 $$\beta_0(p) \geq  \beta_0(3) = \frac{1}{\sqrt{3}}>\frac{1}2 \geq \frac{1}p \;\; \mbox{for all } p \in [2, 3]\,.$$
 Therefore, $\beta_0(p)>1/p$ for all $p \in [2, \infty)$.

We now fix 
$$
\frac{1}{p}<\beta< \beta_0(p).
$$
and use the decay estimates~\eqref{eq: decay bounds Morrey} to bound the four error terms in~\eqref{eq: expansion with integral errors}. To estimate the integral of $|Du_0(x)|^p|x'|^2$ over $B_{1/(4\epsilon)}$ we split it into two parts; one away from the singularity where we can employ the decay estimates, and one near to the singularity where we simply use that $|Du_0|\in L^p$. 

\par Observe that
 \begin{align*}
   \epsilon\int_{B_{1/(4\epsilon)}^c}|Du_0(x)|^p|x'|\,dx &\leq \sigma(\mathbb{S}^{n-1})C_1 \epsilon \int_{1/(4\epsilon)}^\infty r^{-(\beta+1)p}r^{n}\,dr \leq C_2 \epsilon^{p-n+\beta p}\,,\\
   \epsilon^{2}\int_{B_{2}}|Du_0(x)|^p|x'|^2\,dx & \leq 2^2 \epsilon^{2}\int_{B_{2}}|Du_0(x)|^p\,dx \leq C_2 \epsilon^{2}\,,\\
   \epsilon^{2}\int_{B_{1/(4\epsilon)}\setminus B_{2}}|Du_0(x)|^p|x'|^2\,dx &\leq \sigma(\mathbb{S}^{n-1})C_1 \epsilon^{2}\int_{2}^{1/(4\epsilon)}r^{-(\beta+1)p+n+1}\,dr \leq C_2 \epsilon^{p-n+\beta p}\,,\\
  \epsilon^p\int_{B_{1/(2\epsilon)}\setminus B_{1/(4\epsilon)}}|u_0(x)|^p\,dx &\leq \sigma(\mathbb{S}^{n-1})C_1 \epsilon^p\int_{1/(4\epsilon)}^{1/(2\epsilon)}r^{-\beta p +n-1}\,dr \leq C_2 \epsilon^{p-n+\beta p}\,,\\
   \int_{B_{1/(2\epsilon)}\setminus B_{1/(4\epsilon)}}|Du_0(x)|^p\,dx & \leq \sigma(\mathbb{S}^{n-1})C_1 \int_{1/(4\epsilon)}^{1/(2\epsilon)}r^{-(\beta+1)p+n-1}\,dr \leq C_2 \epsilon^{p-n+\beta p}\,.
 \end{align*}
Here $C_2$ is a constant which only depends on $p,n$, and $\beta$. Combining these estimates with~\eqref{eq: expansion with integral errors}, we deduce
\begin{align*}
  \epsilon^{p-n}\|Du_{K,\epsilon}\|_{p}^p &\leq \|Du_0\|_p^p - \epsilon \frac{2p\, \mathrm{tr}(K)}{n-1} \int_{\R^n_\limplus}|Du_0(x)|^{p-2}\partial_{x_n} u_0(x)Du_0(x)\cdot (x', 0)\,dx\\
  &\quad + C_3(\epsilon^{p-n+\beta p}+\epsilon^{2})
\end{align*}
for some constant $C_3$. Choosing 
$$
\gamma = \min\{p-n +\beta p-1, 1\}
$$
completes the proof since $\gamma>0$ and $ C_3(\epsilon^{p-n+\beta p}+\epsilon^{2})\le 2C_3 \epsilon^{1+\gamma}$. 
\end{proof}

\begin{proof}[Proof of Lemma~\ref{lem: integration by parts identity}]
 We first note that $|Du_0|^p|x'| \in L^1(\R^n_\limplus)$ by the decay estimate proved in~\cite{HyndLarsonLindgren-Decay}. 
 It follows that the integral on the left-hand side of the lemma is well-defined. Therefore, dominated convergence implies that
  \begin{equation*}
    \int_{\R^n_\limplus}\hspace{-4pt}|Du_0(x)|^{p-2}\partial_{x_n} \hspace{-2pt}u_0(x)Du_0(x)\cdot (x'\hspace{-3pt}, 0)\hspace{0.6pt}dx =\hspace{-1pt} \lim_{\delta \to 0^\limplus}\hspace{-2pt}\int_{\R^n_\limplus\!\setminus B_\delta(e_n)}\hspace{-4pt}|Du_0(x)|^{p-2}\partial_{x_n} \hspace{-2pt}u_0(x)Du_0(x)\cdot (x'\hspace{-3pt}, 0)\hspace{0.6pt}dx\,.
  \end{equation*}
  Next observe that since $u_0$ is smooth and $p$-harmonic in $\R^n_\limplus\setminus \{e_n\}$,   \begin{align*}
    |Du_0(x)|^{p-2}Du_0(x)\cdot DV(x)^\top Du_0(x) 
    &= \mathrm{div}\bigl(|Du_0(x)|^{p-2}Du_0(x) Du_0(x)\cdot V(x)\bigr)\\
    &\quad - \frac{1}{p}D\bigl(|Du_0(x)|^{p}\bigr)\cdot V(x)\\
    &\quad - \underbrace{\mathrm{div}\bigl(|Du_0(x)|^{p-2}Du_0(x)\bigr)}_{\Delta_pu_0=0}Du_0(x)\cdot V(x)\\
    &=\mathrm{div}\bigl(|Du_0(x)|^{p-2}Du_0(x) Du_0(x)\cdot V(x)\bigr)\\
    &\quad - \frac{1}{p}D\bigl(|Du_0(x)|^{p}\bigr)\cdot V(x)\,
  \end{align*}
for any $x\in \R^n_\limplus\setminus \{e_n\}$ and $V \in C^\infty(\R^n; \R^n)$.

  \par    Employing this identity with $V(x) = (0, \ldots, 0, |x'|^2/2)$ and integrating by parts gives
    \begin{align*}
    \int_{\R^n_\limplus\setminus B_\delta(e_n)}&|Du_0(x)|^{p-2}\partial_{x_n} u_0(x)Du_0(x)\cdot (x', 0)\,dx \\
    &= \int_{\R^n_\limplus\setminus B_\delta(e_n)}\mathrm{div}\bigl(|Du_0(x)|^{p-2}Du_0(x) Du_0(x)\cdot V(x)\bigr)\,dx\\
    &\quad - \frac{1}{p}\int_{\R^n_\limplus\setminus B_\delta(e_n)} D\bigl(|Du_0(x)|^{p}\bigr)\cdot V(x)\,dx\\
    &=
    \int_{\partial(\R^n_\limplus\setminus B_\delta(e_n))} |Du_0(x)|^{p-2}Du_0(x)\cdot V(x) Du_0(x)\cdot \nu(x)\,d\sigma(x)\\
    &+\frac{1}{p}\int_{\R^n_\limplus\setminus B_\delta(e_n)}|Du_0(x)|^p\mathrm{div}(V(x))\,dx - \frac{1}{p}\int_{\partial (\R^n_\limplus\setminus B_\delta(e_n))}|Du_0(x)|^p V(x)\cdot \nu(x)\,d\sigma(x)\,,
  \end{align*}
where $\nu$ is the outward unit normal to the surface and $\sigma$ is the surface measure. The expression above simplifies as $\mathrm{div}(V)\equiv 0$ for our choice of $V$. Moreover,  $Du_0(x) = -\nu \partial_{x_n} u_0(x)$ for $x \in \partial \R^n_\limplus$ since $u_0$ vanishes on $\partial \R^n_\limplus$. As a result, the equality above reduces to
\begin{equation}\label{need to send delta to 0 to get second derivative identity}
\begin{aligned}
    \int_{\R^n_\limplus\setminus B_\delta(e_n)}&|Du_0(x)|^{p-2}\partial_{x_n} u_0(x)Du_0(x)\cdot (x', 0)\,dx \\
    &=
    -\frac{p-1}{2p}\int_{\partial \R^n_\limplus} |Du_0(x)|^{p} |x'|^2\,dx'\\
    & \quad+ \frac{1}{2}\int_{\partial B_\delta(e_n)} |Du_0(x)|^{p-2}\partial_{x_n}u_0(x)|x'|^2 Du_0(x)\cdot \frac{e_n-x}{|e_n-x|}\,d\sigma(x)\\
    & \quad- \frac{1}{2p}\int_{\partial B_\delta(e_n)} |Du_0(x)|^p |x'|^2\frac{1-x_n}{|e_n-x|}\,d\sigma(x)\,.
  \end{aligned}
  \end{equation}

  Since $|x'|\leq \delta$ on $\partial B_\delta(e_n)$ and $|Du_0(x)| \leq C |x-e_n|^{-\frac{n-1}{p-1}}$ (by, e.g.,~\cite[Proposition 2.8]{zbMATH07126544}), the integrals above over $\partial B_\delta(e_n)$ can be estimated as
  \begin{align*}
    \Biggl|\int_{\partial B_\delta(e_n)} |Du_0(x)|^{p-2}\partial_{x_n}u_0(x)|x'|^2 Du_0(x)\cdot \frac{e_n-x}{|e_n-x|}\,d\sigma(x)\Biggr| \leq \sigma(\mathbb{S}^{n-1})C \delta^{n+1-p\frac{n-1}{p-1}}
  \end{align*}
  and 
    \begin{align*}
    \Biggl|\int_{\partial B_\delta(e_n)} |Du_0(x)|^p |x'|^2\frac{1-x_n}{|e_n-x_n|}\,d\sigma(x)\Biggr|\leq \sigma(\mathbb{S}^{n-1})C \delta^{n+1-p\frac{n-1}{p-1}}\,.
  \end{align*}
Our assumption $p>n$ gives
  \begin{equation*}
    n+1-p \frac{n-1}{p-1} = 1+\frac{p-n}{p-1} > 1\,.
  \end{equation*}
We conclude upon sending $\delta \to 0$ in~\eqref{need to send delta to 0 to get second derivative identity}.
\end{proof}


\section{On the range of \texorpdfstring{$\lambda_p(\Omega)$}{the best constants}}\label{sec: Range of lambdap}
Theorem~\ref{thm: universal sharp bounds} asserts that 
$$
\lambda_p(\R^n \setminus\{0\})\leq \lambda_p(\Omega) \leq \lambda_p(\R^n_\limplus)
$$
for any open $\Omega\subsetneq\R^n$. In this section, we shall use Theorem~\ref{thm: neg mean curvature} to deduce that every point in the interval
$$
\bigl[\lambda_p(\R^n\setminus\{0\}), \lambda_p(\R^n_\limplus)\bigr]
$$
equals $\lambda_p(\Omega)$ for some $\Omega$. In fact, we shall prove that it suffices to consider the annular regions
$$
A_{r_1,r_2}=\{x\in \R^n: r_1<|x|<r_2\} \quad \mbox{for }0<r_1<r_2\,.
$$  
\begin{thm}\label{thm: range of lambda}
  Assume $p>n \geq 2$. For any 
  \begin{equation*}
    \lambda_p(\R^n\setminus\{0\}) < \lambda <\lambda_p(\R^n_\limplus)
  \end{equation*}
  there exists a unique $\delta \in (0, 1)$ such that $\lambda = \lambda_p(A_{\delta,1})$.
\end{thm}

\begin{proof} If suffices to show the function $(0, 1)\ni \delta \mapsto \lambda_p(A_{\delta,1})$
is increasing, continuous, and satisfies
     \begin{equation*}
        \lim_{\delta \to 0^\limplus}\lambda_p(A_{\delta,1}) = \lambda_p(\R^n \setminus \{0\}) \quad \mbox{and}\quad \lim_{\delta \to 1^\limminus}\lambda_p(A_{\delta,1}) = \lambda_p(\R^n_\limplus)\,.
      \end{equation*} 
First we will make a few preliminary observations. 

\par By Theorem~\ref{thm: neg mean curvature},  $A_{\delta,1}$ admits an extremal $u_\delta$ and $\lambda_p(A_{\delta,1})<\lambda_p(\R^n_\limplus)$ for each $\delta\in (0,1)$. As $A_{\delta,1}$ is rotationally symmetric, we may assume 
  $$
  u_\delta=w_{r_\delta e_n}^{A_{\delta,1}}
  $$
for some $r_\delta \in (\delta, 1)$. If $1-r_\delta \leq r_\delta-\delta$, then $d_{A_{\delta,1}}(r_\delta e_n) = d_{B_1}(r_\delta e_n)$. It would then follow from Proposition~\ref{prop: Supporting sets} that 
  $$
    \lambda_p(A_{\delta,1}) = \RR_p(A_{\delta,1}, u_\delta) = \RR_p(B_1, u_\delta) \geq \lambda_p(B_1) = \lambda_p(\R^n_\limplus)\,,$$ 
  which is a contradiction. Therefore, 
  $$
  r_\delta-\delta < 1-r_\delta\,.
  $$
  That is, $r_\delta e_n$ is closer to the inner boundary sphere of $A_{\delta, 1}$ than to its outer boundary sphere. 

  \medskip
  
   {\it Part 1} (Monotonicity):   Suppose $0<\delta_1<\delta_2<1$. Observe that $A_{\delta_1,1}$ fully supports $A_{\delta_2,1}$. Indeed, if $y_0 \in \partial A_{\delta_2,1}$ and $|y_0|=1$, then $A_{\delta_1,1}$ supports $A_{\delta_2,1}$ at $y_0$. If $|y_0|=\delta_2$ instead, then $TA_{\delta_1,1}$ supports $A_{\delta_2,1}$ at $y_0$, where $T(x)=(\delta_2/\delta_1)x$. By Proposition~\ref{prop: Supporting sets}, 
   $$
   \lambda_p(A_{\delta_1,1})\le \lambda_p(A_{\delta_2,1})\,.
   $$
   We claim that this inequality is strict.

  Note that  $\tilde w(x) = w_{r_{\delta_{2}} e_n}^{A_{\delta_{2},1}}\in \DD^{1,p}_0(A_{\delta_2,\delta_2/\delta_1})$. And as $r_{\delta_2}-\delta_2< 1-r_{\delta_2} < \delta_2/\delta_1-r_{\delta_2}$, 
  \begin{align*}
    \lambda_p(A_{\delta_1,1}) &= \lambda_p(A_{\delta_2,\delta_2/\delta_1})\leq \RR_p(A_{\delta_2,\delta_2/\delta_1}, \tilde w)=\RR_p(A_{\delta_2,1}, \tilde w)= \lambda_p(A_{\delta_2,1})\,. 
  \end{align*}
  Consequently, if $\lambda_p(A_{\delta_1,1})=\lambda_p(A_{\delta_2,1})$, $\tilde w$ is an extremal for $\lambda_p(A_{\delta_2,\delta_2/\delta_1})$. However, this would contradict Corollary~\ref{ExtSignCor} since $\tilde w$ vanishes in $A_{1,\delta_2/\delta_1}\neq \emptyset$. Therefore, $ \lambda_p(A_{\delta_1,1})< \lambda_p(A_{\delta_2,1})$.

  \medskip
  { \it Part 2} (Continuity): We will argue that $\delta \mapsto \lambda_p(A_{\delta,1})$ is both right and left continuous at a fixed but arbitrary $\delta_0\in (0,1)$. By the monotonicity from Part 1, the function $\delta \mapsto \lambda_p(A_{\delta,1})$ has both left and right limits at $\delta_0$. By Lemma~\ref{lem: interior exhaustion}, 
  \begin{equation*}
    \lim_{\delta \to \delta_0^\limplus} \lambda_p(A_{\delta,1})  \leq \lambda_p(A_{\delta_0,1})\,.
  \end{equation*}
 Also note that the monotonicity proved above implies $\lambda_p(A_{\delta_0,1})\le  \lambda_p(A_{\delta,1})$ for 
 $\delta>\delta_0$. Therefore, 
  \begin{equation*}
\lambda_p(A_{\delta_0,1})\le \lim_{\delta \to \delta_0^\limplus} \lambda_p(A_{\delta,1}) \,.
  \end{equation*}
  This verifies right continuity at $\delta_0$. 
  
  \par Monotonicity also gives 
  \begin{equation*}
    \lim_{\delta \to \delta_0^\limminus} \lambda_p(A_{\delta,1})  \leq \lambda_p(A_{\delta_0,1})\,.
  \end{equation*}
  In order to conclude the left continuity at $\delta_0$, we need to verify that
  \begin{equation}\label{left continuity limit lambda_p}
    \lim_{\delta \to \delta_0^\limminus} \lambda_p(A_{\delta,1})  \geq \lambda_p(A_{\delta_0,1})\,.
  \end{equation}
  To this end, we choose an increasing sequence $\{\delta_k\}_{k\geq 1}$ of positive numbers tending to $\delta_0$. Passing to a subsequence of $\{\delta_k\}_{k\geq 1}$ if necessary, we may also assume that $r_{\delta_k}$ converges. As $r_{\delta_k}\ge \delta_k$ for all $k$,  
  \begin{equation}\label{first r_delta_k limit}
    \lim_{k\to \infty} r_{\delta_k} = \delta_0
  \end{equation}
  or
  \begin{equation}\label{second r_delta_k limit}
    \lim_{k\to \infty} r_{\delta_k} > \delta_0\,.
  \end{equation}
  We will consider both cases separately below. 

  \smallskip {\noindent \it Case 1}. We claim that~\eqref{first r_delta_k limit} cannot occur. If~\eqref{first r_delta_k limit} holds, we will show that
    \begin{equation}\label{eq: limit at jump}
     \lim_{k\to \infty}\lambda_p(A_{\delta_k,1}) = \lambda_p(\R^n_\limplus)\,.
    \end{equation}
  As $\delta \mapsto \lambda_p(A_{\delta,1})$ is increasing,~\eqref{first r_delta_k limit} would imply that $\lambda_p(A_{\delta,1}) >\lambda_p(\R^n_\limplus)$ for $\delta >\delta_0$ but this contradicts Theorem~\ref{thm: universal sharp bounds}. As a result, we can focus on proving that~\eqref{first r_delta_k limit} implies~\eqref{eq: limit at jump}.

  Assume that~\eqref{first r_delta_k limit} holds. Consider the rescaled sequence of extremals defined by
  \begin{equation*}
    w_k(y) = w_{r_{\delta_{k}}e_n}^{A_{\delta_{k},1}}\bigl((r_{\delta_k}-\delta_k)y+\delta_ke_n\bigr)\,.
  \end{equation*}
  Notice that $w_k = w_{e_n}^{\Omega_k}$, where $\Omega_k$ is the annulus
  \begin{equation*}
    \Omega_k = \biggl\{y\in \R^n: \frac{\delta_k}{r_{\delta_k}-\delta_k}<\biggl| y+\frac{\delta_ke_n}{r_{\delta_k}-\delta_k}\biggr|<\frac{1}{r_{\delta_k}-\delta_k}\biggr\}\,.
  \end{equation*}
  As
    $$
      \lambda_p(A_{\delta_k,1})=(\delta_k-r_{\delta_k})^{p-n}\|Dw_{r_{\delta_{k}}e_n}^{A_{\delta_{k},1}}\|_p^p=\|D w_k\|_p^p\,,
    $$
  $\{w_k\}_{k\geq 1}\subset \DD^{1,p}_0(\R^n \setminus \{0\})$ is bounded. By passing to a subsequence if needed, we may assume $w_k \rightharpoonup w_0$ in $\DD^{1,p}_0(\R^n \setminus \{0\})$ and $w_k \to w_0$ uniformly in $B_2$. It follows that 
  $$
    \lim_{k \to \infty} \lambda_p(A_{\delta_k,1})=\liminf_{k \to \infty} \|Dw_k\|_p^p \ge  \|Dw_0\|_p^p. 
  $$

 \par Recall that  $B_{s_k}(-s_ke_n)$ is an exhaustion of $\{y\in \R^n: y_n<0\}$ for any increasing sequence $s_k\rightarrow\infty$. 
 Since
 $$
 B_{\frac{\delta_k}{r_{\delta_k}-\delta_k}}\Bigl(-\frac{\delta_ke_n}{r_{\delta_k}-\delta_k}\Bigr)
 \subset \Omega_k^c
 $$
and $\delta_k/(r_{\delta_k}-\delta_k) \to \infty$,  $B_s(-se_n)\subset \Omega_k^c$ for any $s>0$ provided $k$ is large enough. It follows that $w_0$ vanishes in $B_s(-se_n)$ for every $s>0$. This implies $w_0(y)=0$ whenever $y_n\le 0$. Therefore, $w_0 \in \DD_0^{1,p}(\R^n_\limplus)$. By uniform convergence in $B_2$, we also have $w_0(e_n)=1$. Thus,
   \begin{equation*}
     \lim_{k\to \infty}\lambda_p(A_{\delta_k,1})  \geq \|Dw_0\|_p^p \geq \RR_p(\R^n_\limplus,w_0)\ge \lambda_p(\R^n_\limplus)\,.
   \end{equation*}
By Theorem~\ref{thm: universal sharp bounds}, $\lambda_p(A_{\delta_k,1})\leq \lambda_p(\R^n_\limplus)$ for each $k$ and thus we conclude that~\eqref{first r_delta_k limit} implies~\eqref{eq: limit at jump}.

   \smallskip {\it Case 2.} Now we assume that~\eqref{second r_delta_k limit} holds. By passing to a subsequence, we may assume that the sequence $w_k =w_{r_{\delta_{k}}e_n}^{A_{\delta_{k},1}}$ converges weakly to some $w_0 \in \DD^{1,p}_0(A_{\delta_0,1})$ and uniformly to $w_0$ on compact sets. In particular, $w_0(r_0e_n)=1$. We deduce that
  \begin{align*}
  \lim_{k\to \infty}\lambda_p(A_{\delta_k,1}) 
    &= \lim_{k\to \infty} (\delta_k-r_{\delta_k})^{p-n}\|Dw_k\|_p^p\\
    & \geq (\delta_0-r_0)^{p-n}\|Dw_0\|_p^p \\
    &\ge \RR_p(A_{\delta_0,1}, w_0)\\
    &\geq \lambda_p(A_{\delta_0,1})\,.
  \end{align*}
Therefore,~\eqref{left continuity limit lambda_p} holds. This concludes the proof that $\delta \mapsto \lambda_p(A_{\delta,1})$ is continuous.

   \medskip {\it Part 3} (Endpoint limits): By Theorem~\ref{thm: universal sharp bounds}, $\lambda_p(A_{\delta,1}) \ge \lambda_p(\R^n\setminus\{0\})$ for any $\delta\in (0,1)$. Lemma~\ref{lem: interior exhaustion} yields that  
   $$
   \lim_{\delta\rightarrow 0^\limplus}\lambda_p(A_{\delta,1}) \le  \lambda_p(B_1\setminus\{0\})\,.
   $$
   Therefore, in view of Theorem~\ref{thm: convex, punctured, exterior domains},
  \begin{equation*}
    \lim_{\delta \to 0^\limplus}\lambda_p(A_{\delta,1}) = \lambda_p(B_1\setminus \{0\}) = \lambda_p(\R^n\setminus\{0\})\,. 
  \end{equation*}
   
 \par As for the limit  
   $$
   \lim_{\delta\rightarrow 1^\limminus}\lambda_p(A_{\delta,1}) =  \lambda_p(\R^n_\limplus)\,,
   $$
   we can adapt our argument in Case 1 of Part 2 by choosing a sequence of positive numbers $\delta_k$ tending to 1 from below. In this case, we also have $r_{\delta_k}\in (\delta_k, 1)$ for each $k$. Therefore, it must be that $r_{\delta_k}\to 1$. As a result, the argument used in Case 1 of Part 2 translates directly to establish the above limit. 
\end{proof}


\section{Examples}
\label{sec: Examples and open problems}

In Corollary~\ref{cor: potentials concentrating at bdry}, we computed $\Lambda_p(\Omega)=\lambda_p(\R^n_\limplus)$ under the assumption that all boundary points of $\Omega$ are regular. The key to this result was that we were able to characterize all limiting geometries as halfspaces. The existence of well-defined limits may hold under weaker regularity assumptions on $\partial \Omega$. However, it is in general difficult both to characterize all possible limits and even more so to compute the value of $\lambda_p$ at these limits. In this section, we shall illustrate how the analysis can be carried for certain classes of domains.

\subsection{Polygonal domains}
We will now proceed to study the best constant $\lambda_p$ in polygonal domains. Specifically, we say that a domain $\Omega\subset \R^2$ is {\it polygonal} if $\partial\Omega$ is nonempty and consists of the union of finitely many line segments or rays $\{\Gamma_j\}_{j=1}^K$ only intersecting at their endpoints. We'll also require that each endpoint belongs to exactly two of the $\Gamma_j$'s. We will denote by $\mathcal{S}_\Omega$  the collection of corners of a polygonal domain $\Omega$, which are defined as the set of endpoints of the $\Gamma_j$. To each $y \in \mathcal{S}_\Omega$ we can associate the corresponding interior angle $\varphi_y \in (0, 2\pi)$. We emphasize that we do not exclude $\varphi_y = \pi$ even though such an angle could be removed by merging the two line segments that meet at $y$. 

The following claim is a corollary of Lemma~\ref{lem: blow-up upper bound}. It basically asserts that the family of cones $\Cee^2_\varphi$ for $\varphi\in (0,\pi)$ constitutes the natural class of local model sets for polygonal domains.

\begin{lem}\label{lem: Polygonal domains}
If $\Omega \subsetneq \R^2$ is polygonal with $\mathcal{S}_\Omega \neq \emptyset$, then
  \begin{equation*}
    \Lambda_p(\Omega) \leq  \min_{y \in \mathcal{S}_\Omega} \lambda_p(\Cee^2_{\varphi_y/2})= \lambda_p(\Cee^2_{\max\limits_{\scalebox{0.5}{$y\! \in\! \mathcal{S}_\Omega$}}\varphi_y/2})\,.
  \end{equation*}
  Furthermore, if $\Omega$ is bounded equality holds.
\end{lem}

\begin{proof}
Each $y \in \mathcal{S}_\Omega$ is isolated, so there exists exists an $r>0$ small enough so that up to translation and rotation $B_r(y) \cap \Omega$ agrees with $B_r\cap \Cee^2_{\varphi_y/2}$. Therefore the inequality follows by applying Lemma~\ref{lem: blow-up upper bound} at each of the corners and recalling that $\varphi \mapsto \lambda_p(\Cee^2_\varphi)$ is non-increasing by Lemma~\ref{lem: CeePhi monotonicity}.

To deduce that we have equality in the case when $\Omega$ is bounded we argue as in Corollary~\ref{cor: potentials concentrating at bdry}. Take a sequence $\{x_k\}_{k \geq 1} \in \mathcal{Y}_\Omega$ such that
  \begin{equation*}
    \lim_{k\to \infty}d_\Omega(x_k)^{p-n}\|Dw_{x_k}^\Omega\|_p^p = \Lambda_p(\Omega)\,.
  \end{equation*}
Since $\{x_k\}_{k\geq 1}\in \mathcal{Y}_\Omega$ and $\overline{\Omega}$ is compact, $\{x_k\}_{k\geq 1}$ has a convergent subsequence whose limit $x_0$ belongs to $\partial\Omega$. If $x_0 \notin \mathcal{S}_\Omega$, then $\partial\Omega$ is $C^1$-regular in a neighborhood of $x_0$ so $\Lambda_p(\Omega)\geq \lambda_p(\R^n_\limplus)$ by Lemma~\ref{lem: potentials concentrating at regular boundary point}. Alternatively, if $x_0 \in \mathcal{S}_\Omega$, then following the argument in Lemma~\ref{lem: potentials concentrating at regular boundary point} one proves that $\Lambda_p(\Omega) \ge \lambda_p(\Cee^2_{\varphi_{x_0}/2})$ if the approach of $\{x_k\}_{k\geq 1}$ to $x_0$ is non-tangential while $\Lambda_p(\Omega) \ge \lambda_p(\R^n_\limplus)$ if the approach is tangential. Since we established above that $\Lambda_p(\Omega) \leq \min_{y \in \mathcal{S}_\Omega} \lambda_p(\Cee^2_{\varphi_y/2})$, it follows from Lemma~\ref{lem: CeePhi monotonicity} and Theorem~\ref{thm: universal sharp bounds} that either 
$$
  \Lambda_p(\Omega)=\Lambda_p(\Cee^2_{\varphi_{x_0}/2})=\min_{y \in \mathcal{S}_\Omega}\lambda_p(\Cee^2_{\varphi_y/2}) \quad
  \mbox{or} 
  \quad \Lambda_p(\Omega)= \lambda_p(\R^n_\limplus) = \min_{y \in \mathcal{S}_\Omega}\lambda_p(\Cee^2_{\varphi_y/2})\,.
$$
We note that the latter case can only happen if $\max_{y \in \mathcal{S}_\Omega}\varphi_y \leq \pi$. In either case, we have proved the desired equality. 
\end{proof}

Recall that Theorem~\ref{thm: convex, punctured, exterior domains} asserts that convex domains in $\R^n$ which are not halfspaces do not admit extremals. By applying Proposition~\ref{prop: Supporting sets} there is a similar criterion for a class of polygonal domains. We say that a polygonal domain $\Omega$ is {\it fully supported} if $\Omega$ is fully supported by $\Cee^2_{\varphi^*/2}$ with $\varphi^* = \max_{y\in \mathcal{S}_\Omega}\varphi_y$. In the case $\varphi^* \leq \pi$,  $\Omega$ is a convex polygon, while for $\varphi^*>\pi$, $\Omega$ satisfies a uniform (infinite) outer cone condition with opening angle $2\pi-\varphi^*$.

 \begin{figure}
\centering
\begin{subfigure}{.45\textwidth}
  \centering
  \input{Polygon1}
\end{subfigure}
\begin{subfigure}{.45\textwidth}
  \centering
  \input{Polygon2}
\end{subfigure}
\medskip
\begin{subfigure}{.45\textwidth}
  \centering
  \input{Polygon3}
\end{subfigure}
\begin{subfigure}{.45\textwidth}
  \centering
  \input{Polygon4}
\end{subfigure}
\caption{Four polygonal domains. For $\mathcal{P}_1, \mathcal{P}_3$ the largest interior angle is $\frac{3\pi}{2}$ and for $\mathcal{P}_2, \mathcal{P}_4$ the largest interior angle is $\frac{7\pi}{5}$. The polygons $\mathcal{P}_1, \mathcal{P}_2$ are fully supported while $\mathcal{P}_3, \mathcal{P}_4$ are not. Therefore, by Lemmas~\ref{lem: Polygonal domains} and~\ref{FullySupportedPolyLem}, we know that $\lambda_p(\mathcal{P}_1) = \lambda_p(\Cee^2_{3\pi/4})$, $\lambda_p(\mathcal{P}_2) = \lambda_p(\Cee^2_{7\pi/10})$, $\lambda_p(\mathcal{P}_3) \leq \lambda_p(\Cee^2_{3\pi/4})$, $\lambda_p(\mathcal{P}_4) \leq \lambda_p(\Cee^2_{7\pi/10})$.}
\label{fig:Polygons}
\end{figure}
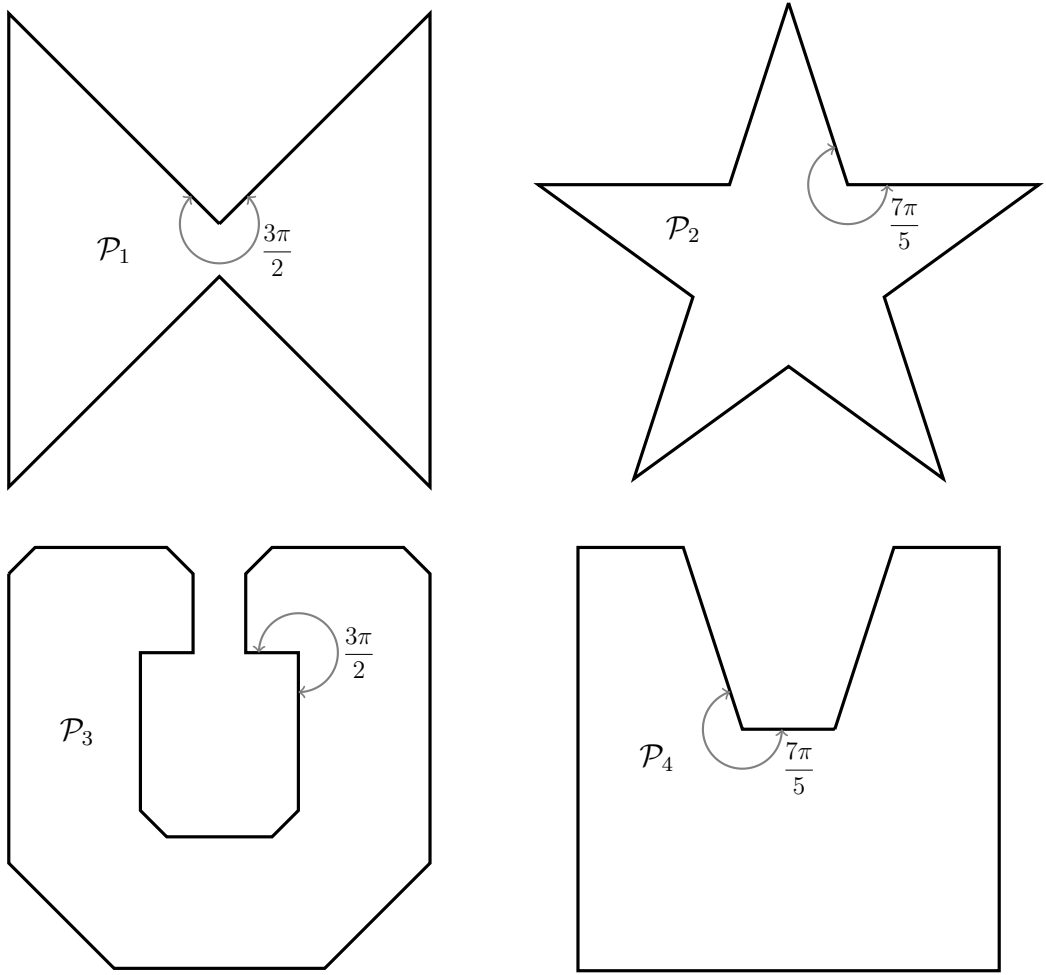

\begin{lem}\label{FullySupportedPolyLem}
Suppose $\Omega\subset \R^2$ is a fully supported polygonal domain with $\varphi^* = \max_{y\in\mathcal{S}_\Omega}\varphi_y$. Then 
$$
\lambda_p(\Omega)=\lambda_p(\Cee^2_{\varphi^*/2}).
$$
Moreover, $\Omega$ does not admit an extremal unless $\Omega=x+Q \Cee^2_{\varphi^*/2}$ for some $x\in \R^2$ and $Q\in O(2)$. 
\end{lem}

\begin{proof}
  That $\lambda_p(\Omega) = \lambda_p(\Cee^2_{\varphi^*/2})$ is a direct consequence of Proposition~\ref{prop: Supporting sets} and Lemma~\ref{lem: Polygonal domains}. To see that $\Omega$ does not admit an extremal, we can argue that the existence of an extremal would contradict Corollary~\ref{ExtSignCor} (as in the proof of Theorem~\ref{thm: convex, punctured, exterior domains}).
\end{proof}

Lemmas~\ref{lem: Polygonal domains} and~\ref{FullySupportedPolyLem} can be directly extended to planar domains whose boundary is $C^1$ outside of a finite set of corners at which the boundary is given by two simple $C^1$ curves that meet at a common endpoint. In this setting it is natural to allow for corners with interior angles equal to $0$ or $2\pi$ to allow for boundaries with cusps. Lemmas~\ref{lem: Polygonal domains} and~\ref{FullySupportedPolyLem} can also be generalized to polytopes in higher dimension by following the argument given above almost verbatim. However, the statements obtained become more complicated as the class of relevant model sets is much larger and our understanding of $\lambda_p$ for these sets is limited. Nevertheless, in the next subsection, we will provide a family of bounded, non-smooth, and non-convex domains in arbitrary dimension where we can determine the value of $\lambda_p$.

\subsection{Examples of non-smooth domains in higher dimensions}
Using the strategy discussed in the previous subsection, we can construct examples of bounded simply connected domains $\Omega$ whose boundary is regular except at a single point and $\lambda_p(\Omega) = \lambda_p(\Cee^n_\varphi)$ for every $n\geq 2$ and $\varphi \in (0, \pi]$.
Specifically one can construct such domains by starting from $B_1\cap \Cee^n_\varphi$ and regularizing the boundary in a small neighborhood of $\partial B_1\cap \partial\Cee^n_\varphi$ in such a manner that the resulting set remains fully supported by $\Cee^n_\varphi$. Two such domains are depicted in Figure~\ref{fig:Smooth Pacman}. The resulting set $\Omega$ satisfies that $\lambda_p(\Omega) = \lambda_p(\Cee^n_\varphi)$ by combining Proposition~\ref{prop: Supporting sets} and Lemma~\ref{lem: blow-up upper bound} with a blow-up around the singular boundary point.

\begin{figure}[h!]
\centering
\begin{subfigure}{.49\textwidth}
  \centering
  \input{SmoothPacman1}
\end{subfigure}
\begin{subfigure}{.49\textwidth}
  \centering
  \input{SmoothPacman2}
\end{subfigure}
\caption{Two planar domains constructed as described. The value of $\lambda_p$ for domain on the left is $\lambda_p(\Cee^2_{5\pi/6})$ and for the domain to the right $\lambda_p(\Cee^2_{\pi})$. Examples for $n>2$ can be obtained by rotation around the axis of symmetry.}
\label{fig:Smooth Pacman}
\end{figure}
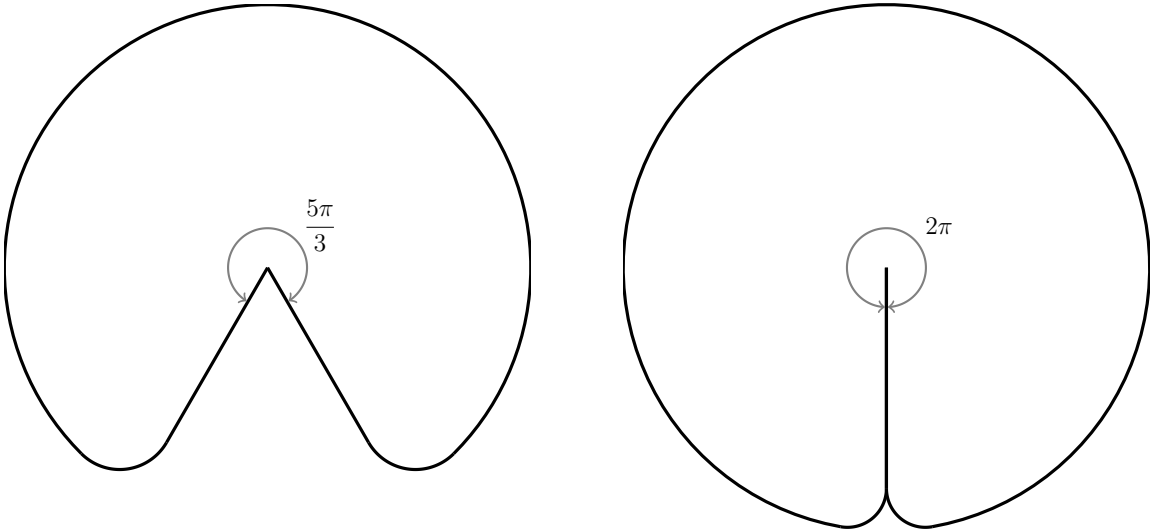

\subsection{Epigraphs}

In much of our analysis we have utilized Lemma~\ref{lem: blow-up upper bound} and Proposition~\ref{prop: blow-up lower bound on Lambda} by looking at the local geometry near some point on the boundary. In this subsection we demonstrate how these results can yield interesting information when instead zooming out. An argument of this form was applied in the proof of case (3) in Theorem~\ref{thm: convex, punctured, exterior domains} but the idea is interesting enough to deserve including a second example. Specifically, we shall consider sets given by the region that lies above the graph of a function.

Let $f \colon \R^{n-1}\to \R$ be a continuous function with $f(0)=0$ and such that along a sequence $\{t_k\}_{k\geq 1}\subset (0, \infty)$ with $\lim_{k\to \infty} t_k =0$ it holds that $F_k\colon \mathbb{R}^{n-1}\to \R$ defined by $z \mapsto t_k f(z/t_k)$ converges locally uniformly to a one-homogeneous function $F\colon \mathbb{R}^{n-1}\to \R$. If we define
\begin{equation*}
  \Omega_f = \{x =(x', x_n)\in \R^n: x_n >f(x')\}
\end{equation*}
and the dilation invariant set
\begin{equation*}
  \Omega_F = \{x \in (x', x_n) \in \R^n: x_n> F(x')\}
\end{equation*}
then it holds that $\Lambda_p(\Omega_f)\leq \lambda_p(\Omega_F)$. Indeed, by the uniform convergence when zooming further and further out the set $\Omega_f$ locally converges to $\Omega_F$ allowing us to apply Lemma~\ref{lem: blow-up upper bound}. Indeed, given $\delta>0$ there is $k$ sufficiently large so that
\begin{align*}
  (\Omega_F + \delta e_n) \cap B_1
  \subset t_k\Omega_f \cap B_1 
  \subset (\Omega_F - \delta e_n) \cap B_1\,.
\end{align*}
Thus by arguing as in the proof of Corollary~\ref{cor: boundaryblowup}, the assumptions of Lemma~\ref{lem: blow-up upper bound} are fulfilled.

A similar argument can be made to work even if the limit of $tf(\theta/t)$ is infinite on some set. However, in this situation one needs to replace the locally uniform convergence in a suitable manner which is in general difficult. But in special cases there are natural ways to do this. For instance if $f \colon \R^n \to \R$ is such that $f(0)=0$ and $f(x) \leq -c|x|^\alpha$ if $|x|\geq R$ for some constants $c>0, \alpha>1,$ and $R>0$, 
then arguing as above one proves
\begin{equation*}
   \Lambda_p(\Omega_f) \leq \lambda_p(\Cee^n_\pi)\,.
 \end{equation*} 
In fact, since $\Omega_f$ (for any $f$) is fully supported by $\Cee^n_\pi$ the prescribed asymptotic behavior implies that $\lambda_p(\Omega_f) = \lambda_p(\Cee^n_\pi)$ by Proposition~\ref{prop: Supporting sets} . Indeed, the picture to keep in mind is that if viewing $\Omega_f$ from farther and farther away the asymptotic behavior of $f$ will lead to the epigraph more and more resembling $\Cee^n_\pi$.

\subsection{Instability under small perturbations}

Here we provide a few examples of domains which are almost identical but in which the value of $\lambda_p$ and the existence/non-existence of extremals are different.

\begin{ex}\label{ex: Removing a conical piece} Let $\Omega \subsetneq \R^n$ be an open set satisfying the assumptions of Theorem~\ref{thm: neg mean curvature}. Then $\lambda_p(\Omega)< \lambda_p(\R^n_\limplus)$ and $\Omega$ admits an extremal. Assume that $\lambda_p(\Omega)> \lambda_p(\Cee^n_\varphi)$ for some $\varphi \in (\pi/2, \pi]$. Given any $r>0$ and $x_0 \in \partial\Omega$ we can construct a set $\Omega'$ such that $\Omega \Delta \Omega'\subset B_r(x_0)$ and $\lambda_p(\Omega')\leq \lambda_p(\Cee^n_\varphi)< \lambda_p(\Omega)$. The idea is to remove a small conical piece of $\Omega$ near $x_0$, see Figure~\ref{fig:InstabilityFig1}.
  
  By translation and rotation we can without loss of generality assume that $x_0 = 0$ and that the outward unit pointing normal to $\partial\Omega$ at $x_0$ is $(0, \ldots, -1)$. The regularity of $\partial \Omega$ ensures that there is an $0<r'<r$ so that $\partial\Omega \cap B_{r'}$ is contained in a $r'/4$ neighborhood of the hyperplane $x_n=0$. We can then take $\Omega'$ as 
  $$(\Omega \cap B_{r'}^c) \cup (\Omega \cap (\Cee^n_\varphi+\tfrac{r'}2e_n))\,.$$
  That is we locally remove a conical piece of $\Omega$ in such a manner that we create a singular boundary point matching that of $\Cee^n_\varphi$. That $\lambda_p(\Omega') \leq \lambda_p(\Cee^n_\varphi)$ now follows directly from Lemma~\ref{lem: blow-up upper bound}.
\end{ex}

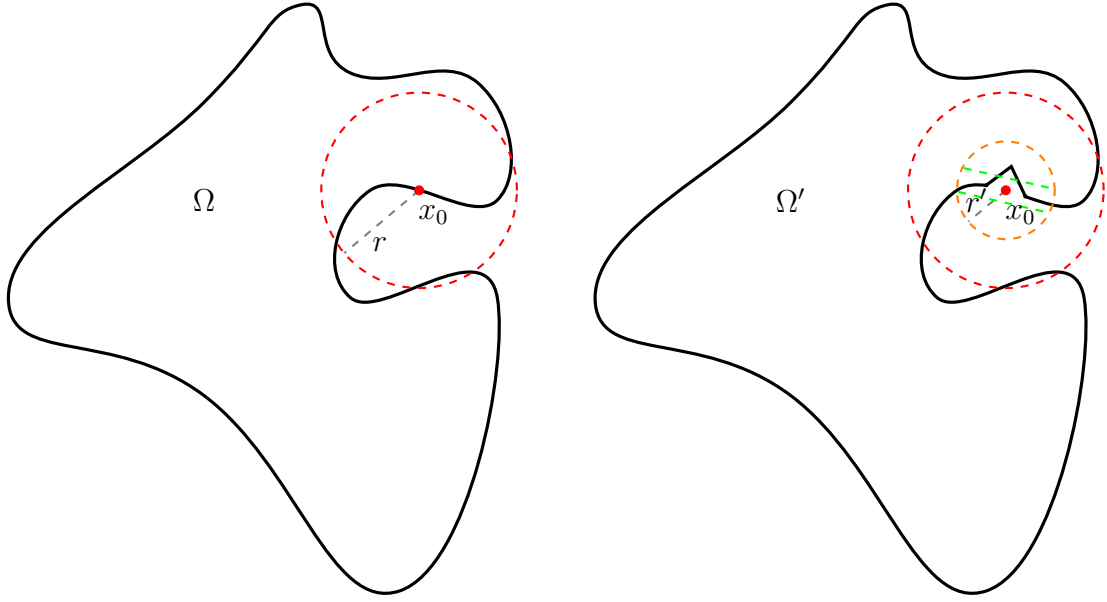
\begin{figure}[ht]
\centering
 \input{InstabilityFig1}
 \caption{A pictorial description of Example~\ref{ex: Removing a conical piece}. Given $\Omega, x_0, r$ we construct the modified set $\Omega'$ by locally removing a conical piece from $\Omega$.}
 \label{fig:InstabilityFig1}
\end{figure}

\begin{ex} 
Let $\Omega \subsetneq \R^n$ be an open, bounded, convex set with $C^2$-regular boundary.  Part (1) of Theorem~\ref{thm: convex, punctured, exterior domains} implies that $\lambda_p(\Omega)=\lambda_p(\R^n_\limplus)$ and that $\Omega$ does not admit an extremal. Given $r>0, x_0 \in \partial\Omega$ we can construct $\Omega'\subset \R^n$ so that $\Omega \Delta \Omega' \subset B_r(x_0)$ and, either 
\begin{enumerate}
   \item $\lambda_p(\Omega') = \lambda_p(\Cee^n_\varphi)$ for any $\varphi \in (\pi/2, \pi]$ and $\Omega'$ does not admit an extremal, or
   \item $\lambda_p(\Omega') <\lambda_p(\R^n_\limplus)$ and $\Omega'$ admits an extremal.
 \end{enumerate} 

The construction for (1) is identical to that in the previous example take $\Omega' = \Omega \cap (Q\Cee^n_\varphi + y_0)$ for a suitably chosen $y_0 \in \Omega \cap B_r(x_0)$ and $Q\in O(n)$. That $\lambda_p(\Omega')\geq \lambda_p(\Cee^n_\varphi)$ follows by noting that $\Omega$ being convex implies that $\Omega'$ is fully supported by $\Cee^n_\varphi$. By Lemma~\ref{lem: blow-up upper bound} and a blow-up at $y_0$ we also have $\lambda_p(\Omega')\leq \Lambda_p(\Omega') \leq \lambda_p(\mathcal{C}^n_{\varphi})$. Non-existence of an extremal follows by observing that if an extremal existed it would be extremal also for some rotated and translated copy of $\Cee^n_\varphi$ and so Corollary~\ref{ExtSignCor} would imply that $\Omega'$ coincides with $\Cee^n_\varphi$ up to translation and rotation which is impossible. 

The construction for (2) is similar but instead of introducing a singular point at the boundary we make a smooth indentation and apply Theorem~\ref{thm: neg mean curvature}. Since $\Omega$ is convex there exists an $0<r'<r$ so that $\partial\Omega \cap B_{r'}(x_0)$ can be represented as the graph of a $C^2$ convex function. By rotation and translating we may assume that $x_0 =0$ and that there is a convex function $f \colon \R^{n-1} \to [0, \infty)$ such that $f(0)=|Df(0)|=0$ and
\begin{equation*}
  \Omega \cap B_{r'} = \{x = (x', x_n)\in \R^n: x_n >f(x')\} \cap B_{r'}\,.
\end{equation*}
If we set $$\phi_\delta(x') = \begin{cases}
  e^{1-1/(1-|x'|^2/\delta^2)} & \mbox{for }|x'|<\delta\,,\\
  0 & \mbox{otherwise},
\end{cases}$$ 
then we can define $\Omega'$ by letting 
\begin{equation*}
  \Omega' \cap B_{r'} =  \Bigl\{x = (x', x_n)\in \R^n: x_n >f(x')+ \frac{r'}{2} \phi_\delta(x')\Bigl\} \cap B_{r'}
\end{equation*}
for some $0<\delta<r'$. Provided $\delta$ is chosen sufficiently small the mean curvature of the boundary at $(0, \ldots, 0, r'/2)$ will be negative and we can apply Theorem~\ref{thm: neg mean curvature} to draw the desired conclusion.
\end{ex}
\begin{ex}
  Fix a nontrivial $f\in C^2(\R^{n-1})$ which has compact support and define
  \begin{equation*}
    \Omega_f = \{x = (x', x_n) \in \R^n: x_n >f(x')\}\,.
  \end{equation*}
  Then $\Omega_f$ admits an extremal and $\lambda_p(\Omega_f) < \lambda_p(\R^n_\limplus)$. To see this we first argue that $\Lambda_p(\Omega_f)=\lambda_p(\R^n_\limplus)$. By Remark~\ref{rem: blow-up at a smooth point remark}, $\Lambda_p(\Omega_f)\leq \lambda_p(\R^n_\limplus)$. For the reverse inequality, let $\{x_k\}_{k\geq 1}\in \mathcal{Y}_{\Omega_f}$ realize the infimum defining $\Lambda_p(\Omega_f)$. Then up to passing to a subsequence we may assume that either $\liminf_{k\to \infty}|x_k| = \infty$ or $\limsup_{k\to \infty}d_{\Omega_f}(x_k)=0$. In either case, as $f$ is $C^1$ and has compact support, the set $\{v>0\}$ in Proposition~\ref{prop: blow-up lower bound on Lambda} will be a halfspace so $\Lambda_p(\Omega_f)\geq \lambda_p(\R^n_\limplus)$. Since 
  $$
    \int_{\R^{n-1}} \mathrm{div}\biggl(\frac{Df(x')}{\sqrt{1+|Df(x')|^2}}\biggr)\,dx' =0
  $$ it follows that the mean curvature of $\partial\Omega_f$ must be negative somewhere. Therefore, Theorem~\ref{thm: neg mean curvature implies energy estimate} implies that $\lambda_p(\Omega_f)<\lambda_p(\R^n_\limplus)$. The desired conclusion thus follows from Proposition~\ref{prop: Compactness threshold}.
\end{ex}

\section{Open problems} 
We bring this paper to an end by listing a few open problems that appear as natural possible extensions of the results we obtained. 

\bigskip
\begin{adjustwidth}{15pt}{15pt}
{\noindent\it \!Open problem 1.} Assume that $\Omega \subsetneq \R^n$ is connected, that $\lambda_p(\Omega)=\lambda_p(\R^n_\limplus)$, and that $\Omega$ admits an extremal. Is it true that $\Omega$ is a halfspace? 

\bigskip
{\noindent\it \!Open problem 2.} Assume that $\Omega \subsetneq \R^n$ is mean-convex. Is it true that $\lambda_p(\Omega) =\lambda_p(\R^n_\limplus)$? 

\bigskip
{\noindent\it \!Open problem 3.} Assume that $\Omega \subsetneq \R^n$ is bounded and has boundary homeomorphic to $\mathbb{S}^{n-1}$. Is it true that $\lambda_p(\Omega) \geq \lambda_p(\Cee^n_\pi)$?

\bigskip
{\noindent \it \!Open problem 4.} If $\Omega \subsetneq \R^n$ admits an extremal, is the extremal unique up to multiplication by constants and similarity transforms that leave $\Omega$ invariant?
\end{adjustwidth}

\bigskip 

{\noindent \it Remark on open problem 1}. By Theorem~\ref{thm: convex, punctured, exterior domains}, the only convex domains with an extremal are halfspaces. We also recall that if $\Omega\subset \R^2$ is a bounded $C^2$ domain, then either $\Omega$ is convex, $\Omega$ does not have an extremal, and $\lambda_p(\Omega)=\lambda_p(\R^2_\limplus)$ or $\Omega$ is not convex, $\Omega$ has an extremal, and $\lambda_p(\Omega)<\lambda_p(\R^2_\limplus)$. 

\bigskip 

{\noindent \it Remark on open problem 2}.
We saw that if $\Omega$ is bounded and fails to be mean-convex, then $\lambda_p(\Omega)<\lambda_p(\R^n_\limplus)$. Furthermore, the hypothesis of this problem implies that every connected component of $\Omega$ is convex when $n=2$. It follows that in the plane, the answer is yes by Theorem~\ref{thm: convex, punctured, exterior domains}. It is also worth noting that for Hardy's inequality~\eqref{eq: phardy}, the sharp constant in any bounded, $C^2$, and mean-convex domain coincides with the constant of the halfspace~\cite{MR2885951}. That is, the analog of open problem 2 is settled for inequality~\eqref{eq: phardy}. 

\bigskip 

{\noindent \it Remark on open problem 3}. The motivation behind this problem is to further understand to what degree $\lambda_p$ is governed by local and/or global geometric properties. The assumptions entail that $\Omega$ is topologically very simple and the suggested lower bound is motivated by the fact that $\Cee^n_\pi$ should be the blow-up that gives the lowest value of $\lambda_p$ possible under the assumptions. We note that Proposition~\ref{prop: Supporting sets} yields the desired conclusion if $\Omega$ is additionally assumed to be fully supported by $\Cee^n_\pi$. In particular, this includes the case when $\Omega$ is a star-shaped domain. Indeed, suppose $\Omega$ is star-shaped with respect to the $0\in \Omega$ and $x\in \partial\Omega$. Then $tx\in \Omega$ for $t\in [0,1)$ and $tx\not\in \Omega$ for $t\ge 1$. It follows that there is $Q\in O(n)$ with $Q(\Omega-x)\subset \Cee^n_\pi$. As $x\in \partial\Omega$ was arbitrary, $\Omega$ is fully supported by $\Cee^n_\pi$. 

\bigskip 

{\noindent \it Remark on open problem 4}.
The answer is yes when $\Omega$ is a halfspace or a punctured wholespace for $n\ge 2$. This follows from Proposition~\ref{lem: half- and punctured space} and the uniqueness of Morrey extremals up to similarity transformations (Section 3 of~\cite{MR4275749}).

\appendix

\section{Approximation}
\label{sec: App A Approximiation}
This appendix is dedicated to proving Lemma~\ref{lem: approx with compact support}. To this end, we suppose $u\in \DD^{1,p}_0(\Omega)$ and $\epsilon>0$. Recall that our goal is to find $v\in C^\infty_c(\Omega)$ with 
\begin{equation}\label{appendix approx inequality}
\|Du-Dv\|_p\le \epsilon \|Du\|_p\,. 
\end{equation}
By translating $\Omega$ if necessary, we may assume that $0\in \partial \Omega$. And to ease notation in the following proof, we will write $d(x)$ for $d_\Omega(x)$. 

\medskip 

{\it Step 1.}  First, we choose a non-increasing $\eta \in C^\infty([0, \infty))$ with $0\leq \eta \leq 1$, $\eta(1/2)=1$, and $\eta(1)=0$. Next, we set
\begin{equation}
  f(x) = (1-\eta(d(x)/\delta))\eta(|x|/r)
  \end{equation}
for $x\in \Omega$ and $r, \delta >0$. It is evident that $f$ is supported in $\{x \in \Omega: d(x)\geq  \delta/2, |x|\leq r\}$, 
  \begin{equation}\label{the appendix multiplier function} 
  f=1 \mbox{ in } \{x \in \Omega: d(x)\geq \delta, |x|\leq r/2\}\,,
  \end{equation}
and $f$ is Lipschitz continuous. It follows that 
$$
v_1 =fu\in \DD^{1,p}_0(\Omega)
$$
and 
\begin{equation}\label{v1 support} 
  \text{supp}(v_1)\subset \{x \in \Omega: d(x)\geq  \delta/2, |x|\leq r\}\,.
\end{equation}
  
\par  Direct computation gives 
$$
Dv_1(x) = f(x)Du(x) -\delta^{-1}\eta'(d(x)/\delta)\eta(|x|/r)u(x) Dd(x)+ r^{-1}(1-\eta(d(x)/\delta))\eta'(|x|/r)u(x) \frac{x}{|x|}\,.
$$
Employing~\eqref{the appendix multiplier function} and recalling $|Dd|\leq 1$ almost everywhere, we also find 
  \begin{align*}
 & \|Du-Dv_1\|_p \\
   &\quad \leq
    \|\chi_{\{d<\delta\}} Du\|_p+\|\chi_{\R^n\setminus  B_{r/2}} Du\|_p +\delta^{-1}\|\eta'\|_\infty \|\chi_{\{d<\delta\}}u\|_p+ r^{-1}\|\eta'\|_\infty \|\chi_{B_r \setminus B_{r/2}}u\|_p\,.
  \end{align*}
 By dominated convergence, 
  $$
   \|\chi_{\{d<\delta\}} Du\|_p+\|\chi_{\R^n\setminus  B_{r/2}} Du\|_p \le \frac{\epsilon}4 \|Du\|_p\,,
  $$
  provided we choose $\delta$ sufficiently small and $r$ sufficiently large. To see that the remaining terms $\delta^{-1}\|\eta'\|_\infty \|\chi_{\{d<\delta\}}u\|_p$ and $r^{-1}\|\eta'\|_\infty \|\chi_{B_r \setminus B_{r/2}}u\|_p $ can also be made small, we argue as follows.

\medskip

{\it Step 2.}  By Morrey's estimate~\eqref{MorreyEstimate},
there is a constant $c$ such that
  \begin{equation*}
    |u(x)|^p \leq c d(x)^{p-n} \int_{B_{d(x)}(x)}|Du(w)|^p\,dw\,
  \end{equation*}
  for all $x\in \Omega$. Therefore,
  \begin{align*}
    \delta^{-p}\int_\Omega \chi_{\{d<\delta\}}(x)|u(x)|^p\,dx
    &\leq 
    c\delta^{-p}\int_\Omega \chi_{\{d<\delta\}}(x)d(x)^{p-n} \int_{B_{d(x)}(x)}|Du(w)|^p\,dwdx\\
    &=
    c\delta^{-p}\int_{\Omega}\int_{\Omega} \chi_{\{d<\delta\}}(x)\chi_{B_{d(x)}(x)}(w)d(x)^{p-n}|Du(w)|^p\,dwdx\,.
  \end{align*}
  It is not hard to see that if $d(x)<\delta$ and $w\in B_{d(x)}(x)$, then $d(w)<2\delta$ and $x\in B_{\delta}(w)$. As a result, 
  \begin{equation*}
    d(x)^{p-n}\chi_{\{d<\delta\}}(x)\chi_{B_{d(x)}(x)}(w) \leq \delta^{p-n}\chi_{\{d<\delta\}}(x)\chi_{B_{d(x)}(x)}(w)\leq \delta^{p-n}\chi_{\{d<2\delta\}}(w)\chi_{B_\delta(w)}(x)\,.
  \end{equation*}
  Thus, there is a constant $C$ with 
  \begin{align*}
    \delta^{-p}\int_{\Omega}\chi_{\{d<\delta\}}(x)|u(x)|^p\,dx
    &\leq 
    c\delta^{-n}\int_{\Omega}\int_{\Omega} \chi_{\{d<2\delta\}}(w)\chi_{B_{\delta}(w)}(x)|Du(w)|^p\,dwdx\\
    &\leq
    C\int_{\Omega} \chi_{\{d<2\delta\}}(w)|Du(w)|^p\,dw\,.
  \end{align*}
  By dominated convergence, 
  $$
  \delta^{-1}\|\eta'\|_\infty\|\chi_{\{d<\delta\}}u\|_p \leq \frac{\epsilon}4 \|Du\|_p
  $$
  provided $\delta>0$ is small enough.

\medskip 

{\it Step 3.}  In order to estimate $r^{-1}\|u \chi_{B_{r}\setminus B_{r/2}}\|_p$, we first change variables in the integral
  \begin{equation*}
    r^{-p}\int_{B_{r}\setminus B_{r/2}} |u(x)|^p\,dx = \int_{B_1\setminus B_{1/2}}|r^{n/p-1}u(ry)|^p\,dy\,.
  \end{equation*}
  Recall that we assumed that $0 \in \Omega^c$. Thus for $y \in B_1\setminus B_{1/2}$ such that $ry \in \mathrm{supp}(v)\subset\Omega$,
  \begin{equation*}
    |r^{n/p-1}u(ry)| = \frac{d(ry)^{1-n/p}}{r^{1-n/p}}\frac{|u(ry)|}{d(ry)^{1-n/p}} 
    \leq |y|^{1-n/p}\frac{|u(ry)|}{d(ry)^{1-n/p}} \leq \frac{|u(ry)|}{d(ry)^{1-n/p}}\,.
  \end{equation*}
 Consequently,
  \begin{align*}
    r^{-1}\|\chi_{B_r\setminus B_{r/2}}u\|_p &\leq \biggl(\int_{B_1\setminus B_{1/2}} \frac{|u(ry)|^p}{d(ry)^{p-n}}\,dy\biggr)^{1/p}\,.
  \end{align*}
  By~\eqref{eq: inequality intro}, the functions $B_1\ni y \mapsto \frac{|u(ry)|}{d(ry)^{1-n/p}}$ are bounded uniformly in $r>0$; and by Lemma~\ref{TechLimLem}~\ref{limitatinfinity}, they tend to zero pointwise in the limit as $r \to \infty$. Therefore, 
  \begin{equation*}
    r^{-1}\|\eta'\|_\infty\|\chi_{B_r\setminus B_{r/2}}u\|_p \leq \frac{\epsilon}4 \|Du\|_p
  \end{equation*}
  provided $r>0$ is large enough. In summary, for $r>0$ sufficiently large and $\delta>0$ sufficiently small  
  \begin{equation*}
    \|Du-Dv_1\|_p \leq \frac{3\epsilon}{4}\|Du\|_p\,.
  \end{equation*}

 \medskip

{\it Step 4.}  Select any $\psi \in C_c^\infty(\R^n)$ with $\mathrm{supp}(\psi)\in B_1$ and $\int_{\R^n}\psi(x)\,dx =1$.
Also define $\psi_\tau= \tau^{-n}\psi(\cdot/\tau)$ and note that $\mathrm{supp}(\psi_\tau)\in B_\tau$. Using~\eqref{v1 support},  it is straightforward to verify
$$
v = \psi_{\tau}*v_1 \in C_c^\infty(\Omega)
$$
for $\tau<\delta/2$. Note that
  \begin{equation*}
    \|Du-Dv\|_p \leq \|Du-Dv_1\|_p + \|Dv_1-Dv\|_p \leq \frac{3\epsilon}{4}\|Du\|_p + \|Dv_1-\psi_{\tau}*Dv_1\|_p\,.
  \end{equation*}
  Since $|Dv_1|\in L^p(\R^n)$, it follows from standard results on mollification that the last term is smaller than $\frac{\epsilon}{4}\|Du\|_p$ for $\tau>0$  chosen sufficiently small (see, e.g.~\cite{MR3409135}). This concludes the proof of~\eqref{appendix approx inequality}.

\section{\texorpdfstring{$\Lambda_p(\Omega)$}{Lambda} is attained}
\label{app: Lambda is attained}
In this appendix, we argue that for any open $\Omega\subsetneq\R^n$ there exists a sequence $\{x_k\}_{k\geq1} \subset \mathcal{Y}_\Omega$ which realizes the infimum defining $\Lambda_p(\Omega)$. Namely, we will show
\begin{equation*}
  \liminf_{k\to \infty}d_\Omega(x_k)^{p-n} \|Dw_{x_k}^\Omega\|_p^p = \Lambda_p(\Omega)\,.
\end{equation*}
With this goal in mind, we let $\{\epsilon_j\}_{j\geq 1}\subset (0, 1)$ satisfy $\lim_{j\to \infty}\epsilon_j =0$. For each $j$, there exists a sequence $\{x_k^j\}_{k\geq 1}\in \mathcal{Y}_\Omega$ with 
\begin{equation}\label{limsup appendix inequality}
\liminf_{k\to\infty} d_\Omega(x_k^j)^{p-n}\|Dw_{x_k^j}^\Omega\|_p^p \leq \Lambda_p(\Omega)+\epsilon_j\,.
\end{equation}
Let $A \subset \N$ be the subset of indexes $j$ such that $\liminf_{k\to \infty}|x_k^j| = \infty$. Recall that for $j \in A^c$, $\limsup_{k\to \infty}d_\Omega(x_k^j)=0$. Notice that at most one of $A$ and $A^c$ is a finite set. 

\par First suppose that $A$ is infinite, and consider the subsequence of the sequence of sequences $\{x_k^j\}_{k\geq 1}$ for which $j \in A$. Relabelling the subsequence if necessary, we obtain a collection of sequences $\{x_k^j\}_{k\geq 1} \in \mathcal{Y}_\Omega$ which satisfy~\eqref{limsup appendix inequality} for some sequence of positive numbers $\{\epsilon_j\}_{j\geq 1}$ with $\lim_{j\to \infty}\epsilon_j = 0$ and for which $\liminf_{k\to \infty}|x_k^j|=\infty$ for each $j$. 

\par We now iteratively construct a new sequence $\{x^*_k\}_{k\geq 1}\in \mathcal{Y}_\Omega$ as follows. Let $x^*_1 = x_1^1$. Given $\{x^*_k\}_{k=1}^{N-1}$ for $N\ge 2$, choose $x^*_N=x_l^N$, where $l$ is the first index so that 
\begin{equation*}
  |x_l^N|\geq |x^*_{N-1}|+1 \quad \mbox{and} \quad d_\Omega(x_l^N)^{p-n}\|Dw_{x_l^N}^\Omega\|_p^p \leq \Lambda_p(\Omega)+2\epsilon_N\,.
\end{equation*}
The construction implies that 
$$\liminf_{k\to \infty}|x^*_k|\geq \liminf_{k\to \infty}(k-1) = \infty$$  
and
\begin{equation*}
   \liminf_{k\to \infty}d_\Omega(x_k^*)^{p-n}\|Dw_{x_k^*}^\Omega\|_p^p \leq \liminf_{k\to \infty}\Lambda_p(\Omega)+2\epsilon_k = \Lambda_p(\Omega)\,.
 \end{equation*} 
Thus $\{x^*_k\}_{k\geq 1} \in \mathcal{Y}_\Omega$ and $$\liminf_{k\to \infty}d_\Omega(x_k^*)^{p-n}\|Dw_{x_k^*}^\Omega\|_p^p= \Lambda_p(\Omega)\,.
$$ 

The case when $A$ is finite can be treated similarly except the first criteria when choosing $x^*_N$ is replaced with $d_\Omega(x_l^N)\leq d_\Omega(x_{N-1}^*)/2$.


\bibliographystyle{plain}

\typeout{get arXiv to do 4 passes: Label(s) may have changed. Rerun}

\end{document}

%% file: SupportedFig1
\begin{tikzpicture}[scale=0.95]

\clip (-5,-3.7) rectangle (5,3.7);

\def \phi {300};
\def \shift {1};

\draw [very thick] ({-3-\shift},3)--({-3-\shift},-3)--({-1-\shift},-3)--({1-\shift},-3)--({3-\shift},-3)--({3-\shift},{-3*tan((360-\phi)/2)})--({0-\shift},0)--({3-\shift},{3*tan((360-\phi)/2)})--({3-\shift},3)--({-3-\shift},3);

\draw [red, dashed, very thick] ({6-\shift},{-6*tan((360-\phi)/2)})--({0-\shift},0)--({6-\shift},{6*tan((360-\phi)/2)});

\draw [thick, gray, domain={180-\phi/2}:{180+\phi/2}, samples=100, <->] plot ({3*cos(\x)/4-\shift},{3*sin(\x)/4});

\node at ({-0.7-\shift}, 0.7) {$\varphi$};

\node at ({-2-\shift}, 2) {{$\mathcal{P}$}};

\end{tikzpicture}

%% file: SupportedFig2
\begin{tikzpicture}[scale=1]

\clip (-8,-4) rectangle (8,4);

\def \theta {55};
\def \d {0.85};

\def \rotation {95};


\draw [very thick] plot [smooth cycle, tension =0.75, cm={1,0,0,1, (-3.7,0.5)}] coordinates {({-4},0) ({-2},2) ({-1},3) ({-0.8},2.3) ({0},2.5) ({-0.5},0) ({1},0) ({0},-3) ({-2}, -1)};

\draw [gray, fill, cm={1,0,0,1, (-3.7,0.5)}] (-2,-1) circle (1.3pt);
\node [cm={1,0,0,1, (-3.7,0.5)}] at (-2,-1.25)  {$x_0$};

\draw [gray, fill, cm={1,0,0,1, (-3.7,0.5)}] ({-2+\d*cos(\theta)},{-1+\d*sin(\theta)}) circle (1.3pt);
\node [cm={1,0,0,1, (-3.7,0.5)}] at ({-2+\d*cos(\theta)+0.2},{-1+\d*sin(\theta)-0.2})  {$y_0$};
\node [cm={1,0,0,1, (-3.7,0.5)}] at ({-2+\d*cos(\theta)/2-0.25},{-1+\d*sin(\theta)/2+0.2})  {$d_\Omega$};
\draw [dashed, thick, gray, cm={1,0,0,1, (-3.7,0.5)}] (-2,-1)--({-2+\d*cos(\theta)},{-1+\d*sin(\theta)}) ;
\draw [gray, thick, dashed, cm={1,0,0,1, (-3.7,0.5)}] ({-2+\d*cos(\theta)},{-1+\d*sin(\theta)}) circle (\d);

\node at (-6.5,1) {$\Omega$};


\draw [scale=1.5, very thick] plot [smooth cycle, tension =0.75, cm={cos(\rotation),-sin(\rotation),sin(\rotation),cos(\rotation), (2,-1.5)}] coordinates {({-4},0) ({-2},2) ({-1},3) ({-0.8},2.3) ({0},2.5) ({-0.5},0) ({1},0) ({0},-3) ({-2}, -1)};

\draw [scale=1.5, gray, fill, cm={cos(\rotation),-sin(\rotation),sin(\rotation),cos(\rotation), (2,-1.5)}] (-2,-1) circle (1.3pt);

\draw [scale=1.5, gray, fill, cm={cos(\rotation),-sin(\rotation),sin(\rotation),cos(\rotation), (2,-1.5)}] ({-2+\d*cos(\theta)},{-1+\d*sin(\theta)}) circle (1.3pt);

\draw [scale=1.5, dashed, thick, gray, cm={cos(\rotation),-sin(\rotation),sin(\rotation),cos(\rotation), (2,-1.5)}] (-2,-1)--({-2+\d*cos(\theta)},{-1+\d*sin(\theta)}) ;
\draw [scale=1.5, blue, very thick, cm={cos(\rotation),-sin(\rotation),sin(\rotation),cos(\rotation), (2,-1.5)}] ({-2+\d*cos(\theta)},{-1+\d*sin(\theta)}) circle (\d);
\draw [scale=1.5, gray, fill, cm={cos(\rotation),-sin(\rotation),sin(\rotation),cos(\rotation), (2,-1.5)}] (-2,-1) circle (1.3pt);
\draw [scale=1.5, gray, fill, cm={cos(\rotation),-sin(\rotation),sin(\rotation),cos(\rotation), (2,-1.5)}] ({-2+\d*cos(\theta)},{-1+\d*sin(\theta)}) circle (1.3pt);

\node  at (3,0)  {$0$};

\node at (1.5,0.9)  {$x$};

\node at (4, 2) {$T\Omega$};

\node at (2.2,-0.8) {$B_1$};

\node at (2.4,0.67) {$1$};

\draw [very thick, red, dashed, ->] (-3.5,1.2) to [out=30, in=150] (1.4,1.2);

\node at (-1, 2.2) {\scalebox{1.2}{$T$}};

\end{tikzpicture}

%% file: CeePhiFig
\begin{tikzpicture}[scale=0.85]

\clip (-8,-3.5) rectangle (8,4.1);

\def \L {7};
\def \l {6};
\def \phi {115};

\draw [blue, very thick, domain=0:\L, samples=100] plot ({\x},{\x*cos(\phi)/sin(\phi)});
\draw [blue, very thick, domain=0:\L, samples=100] plot ({-\x},{\x*cos(\phi)/sin(\phi)});

\draw [purple, very thick, dashed, domain=0:\l, samples=100] plot ({\x+2*cos(180-\phi)},{\x*cos(\phi)/sin(\phi)+2*sin(180-\phi)});
\draw [purple, very thick, dashed, domain=0:\l, samples=100] plot ({-\x+2*cos(\phi)},{\x*cos(\phi)/sin(\phi)+2*sin(\phi)});

\draw [gray, dotted, very thick] (0,0)--({2*cos(\phi)},{2*sin(\phi)});
\draw [gray, dotted, very thick] (0,0)--({-2*cos(\phi)},{2*sin(\phi)});

\draw [gray, thick, ->] (0, -3.5)--(0, 4);
\draw [gray, thick, ->] (-7, 0)--(7, 0);

\draw [very thick, purple, domain=-\phi+180:\phi, samples=100] plot ({2*cos(\x)},{2*sin(\x)});

\draw [thick, black, domain=90-\phi:90, samples=100, <->] plot ({3*cos(\x)/4},{3*sin(\x)/4});

\node at (0.9, 0.5) {{$\displaystyle\varphi$}};

\node at (-3, 3) {{$\mathcal{C}^2_\varphi$}};
\node at (0.9, 2.8) {{$K^2_\varphi$}};

\node at (0.375, 3.9) {{$x_2$}};
\node at (6.9, -0.3) {{$x_1$}};

\draw [thick, ->] (0.7,2.6)--({2.02*cos(80)},{2.02*sin(80)});

\draw[fill, purple] ({2*cos(\phi)},{2*sin(\phi)}) circle (0.05);
\draw[fill, purple] ({-2*cos(\phi)},{2*sin(\phi)}) circle (0.05);

\end{tikzpicture}

%% file: BlowupFig
\begin{tikzpicture}[scale=0.59]

\def \Scale {4};

\begin{scope}[scale=1]

\def \Xshift {-5};
\def \Yshift {0};

\clip (\Xshift,\Yshift) circle (4.5);

\draw [very thick, domain=-5:0, samples=100] plot ({\x+sin(180*\x/pi)*3/5+\Xshift},{exp(\x)*(-2*\x*\x*\x-7*\x*\x+1.5*\x)+\Yshift});
\draw [very thick, domain=0:4, samples=100] plot ({\x+\Xshift},{-(\x+1)*sin(90*\x/pi)+\x*\x/4+\Yshift});
\draw [very thick, domain=4:6, samples=100] plot ({\x-1.5*(\x-4)^2-(\x-4)/6+\Xshift},{-(\x+1)*sin(90*\x/pi)+\x*\x/4+\Yshift});

\draw [red, dashed, very thick, domain=0:360, samples=50] plot ({cos(\x)+\Xshift},{sin(\x)+\Yshift});

\draw [blue, thick] (-2*8/5+\Xshift,-2*1.5+\Yshift)--(0+\Xshift,0+\Yshift);
\draw [blue, thick] (4+\Xshift,-2+\Yshift)--(0+\Xshift,0+\Yshift);
\end{scope}


\begin{scope}[scale=\Scale]
\def \Xshift {5/\Scale};
\def \Yshift {0};

\clip (\Xshift,\Yshift) circle ({4.5/\Scale});

\draw [very thick, domain=-2:0, samples=100] plot ({\x+sin(180*\x/pi)*3/5+\Xshift},{exp(\x)*(-2*\x*\x*\x-7*\x*\x+1.5*\x)+\Yshift});
\draw [very thick, domain=0:2, samples=100] plot ({\x+\Xshift},{-(\x+1)*sin(90*\x/pi)+\x*\x/4+\Yshift});

\draw [red, dashed, very thick, domain=0:360, samples=50] plot ({cos(\x)+\Xshift},{sin(\x)+\Yshift});
\draw [orange, dashed, very thick, domain=0:360, samples=50] plot ({cos(\x)/\Scale+\Xshift},{sin(\x)/\Scale+\Yshift});

\draw [blue, thick] (-2*8/5+\Xshift,-2*1.5+\Yshift)--(0+\Xshift,0+\Yshift);
\draw [blue, thick] (4+\Xshift,-2+\Yshift)--(0+\Xshift,0+\Yshift);
\end{scope}


\begin{scope}[scale={\Scale*\Scale}]

\def \Xshift {5/\Scale/\Scale};
\def \Yshift {-9/\Scale/\Scale};

\clip (\Xshift,\Yshift) circle ({4.5/\Scale/\Scale});

\draw [very thick, domain=-1:0, samples=100] plot ({\x+sin(180*\x/pi)*3/5+\Xshift},{exp(\x)*(-2*\x*\x*\x-7*\x*\x+1.5*\x)+\Yshift});
\draw [very thick, domain=0:1, samples=100] plot ({\x+\Xshift},{-(\x+1)*sin(90*\x/pi)+\x*\x/4+\Yshift});

\draw [red, dashed, very thick, domain=0:360, samples=50] plot ({2*cos(\x)+\Xshift},{2*sin(\x)+\Yshift});

\draw [orange, dashed, very thick, domain=0:360, samples=50] plot ({cos(\x)/\Scale+\Xshift},{sin(\x)/\Scale+\Yshift});
\draw [red, dashed, very thick, domain=0:360, samples=50] plot ({cos(\x)/\Scale/\Scale+\Xshift},{sin(\x)/\Scale/\Scale+\Yshift});

\draw [blue, thick] (-2*8/5+\Xshift,-2*1.5+\Yshift)--(0+\Xshift,0+\Yshift);
\draw [blue, thick] (4+\Xshift,-2+\Yshift)--(0+\Xshift,0+\Yshift);
\end{scope}


\begin{scope}[scale={\Scale*\Scale*\Scale}]

\def \Xshift {-5/\Scale/\Scale/\Scale};
\def \Yshift {-9/\Scale/\Scale/\Scale};

\clip (\Xshift,\Yshift) circle ({4.5/\Scale/\Scale/\Scale});

\draw [very thick, domain=-1:0, samples=100] plot ({\x+sin(180*\x/pi)*3/5+\Xshift},{exp(\x)*(-2*\x*\x*\x-7*\x*\x+1.5*\x)+\Yshift});
\draw [very thick, domain=0:1, samples=100] plot ({\x+\Xshift},{-(\x+1)*sin(90*\x/pi)+\x*\x/4+\Yshift});

\draw [red, dashed, very thick, domain=0:360, samples=50] plot ({cos(\x)/\Scale/\Scale+\Xshift},{sin(\x)/\Scale/\Scale+\Yshift});

\draw [blue, thick] (-2*8/5+\Xshift,-2*1.5+\Yshift)--(0+\Xshift,0+\Yshift);
\draw [blue, thick] (4+\Xshift,-2+\Yshift)--(0+\Xshift,0+\Yshift);
\end{scope}

\draw [very thick, dotted, red] ({-5+cos(105)},{sin(105)})--({5+4*cos(105)},{4*sin(105)});
\draw [very thick, dotted, red] ({-5+cos(105)},{-sin(105)})--({5+4*cos(105)},{-4*sin(105)});

\draw [very thick, dotted, orange] ({5+sin(105)},{-cos(105)})--({5+4*sin(105)},{-9-4*cos(105)});
\draw [thick, dotted, orange] ({5-sin(105)},{-cos(105)})--({5-4*sin(105)},{-9-4*cos(105)});

\draw [very thick, dotted, red] ({5-cos(105)},{-9+sin(105)})--({-5-4*cos(105)},{-9+4*sin(105)});
\draw [very thick, dotted, red] ({5-cos(105)},{-9-sin(105)})--({-5-4*cos(105)},{-9-4*sin(105)});

\node at (-7, 1) {$\Omega$};
\node at (-8.5, -2.6) {$\mathcal{C}$};

\node at (3, 1) {$\Omega$};
\node at (1.5, -2.6) {$\mathcal{C}$};

\node at (3, {1-9}) {$\Omega$};
\node at (1.5, {-2.6-9}) {$\mathcal{C}$};

\node at (-7, {1-9}) {$\Omega$};
\node at (-8.5, {-2.6-9}) {$\mathcal{C}$};

\end{tikzpicture}

%% file: NegMeanCurvature
\begin{tikzpicture}[scale=0.65]

\clip (-10,-5.5) rectangle (10,5.5);

\def \K {-0.02};
\def \L {5};
\def \angle {15};

\draw [gray!70, thick, ->] (0, {-\L-.5})--(0, {\L+.5});
\draw [gray!70, thick, ->] (-9.5, 0)--(9.5, 0);

\node at (0.5, {\L+0.3}) {{$x_n$}};
\node at (9.3, -0.4) {{$x'$}};
\node at ({3*\L/4*cos(65)},{3*\L/4*sin(65)-1}) {{$r$}};

\node at (-7,-1.5) {{$x'\cdot Kx'$}};
\node at (-8,0.5) {{$\partial\Omega$}};

\draw [very thick, gray, domain=0:360, samples=100] plot ({\L*cos(\x)},{\L*sin(\x)});

\draw [ultra thick,black,samples=10,domain=-10:-2] plot ({\x}, {sin(\x*600)/\x+ \x*\x/50-1.5});
\draw [ultra thick] (2,{-(2+2)/5+(-sin(-2*600)/2+4/50-1.5)})--(-0.8+3,-0.15-1)--(-0.4+3,0-1)--(0.2+3,0-1)--(0.5+3, -0.1-1)--(1+3, 0-1)--(2.1+3,-1/2-1)--(3.5+3,1/3-1)--(5+3, 1-1)--(5.8+3, 1.9-1)--(10,0.7);

\draw [ultra thick, samples=20, domain=-2:2] plot ({\x},{(-(\x+2)/5+(-sin(-2*600)/2+4/50-1.5))*\x*\x/4)});

\draw [very thick, blue, domain=-10:10, samples=100] plot ({\x},{\K*\x*\x});

\draw [thick, gray, dashed, <->] (0,0)--({\L*cos(65)},{\L*sin(65)});

\end{tikzpicture}

%% file: Polygon1
\begin{tikzpicture}[scale=0.7]

\clip (-4.1,-4.6) rectangle (4.1,4.6);

\draw [very thick] (0,1/2)--(4,4+1/2)--(4,-4-1/2)--(0,-1/2)--(-4,-4-1/2)--(-4,4+1/2)--(0,1/2);

\draw [thick, gray, domain=135:405, samples=100, <->] plot ({3*cos(\x)/4},{1/2+3*sin(\x)/4});

\node at (1.1, 0) {\scalebox{0.8}{$\displaystyle\frac{3\pi}{2}$}};

\node at (-2, 0) {{$\mathcal{P}_1$}};

\end{tikzpicture}

%% file: Polygon2
\begin{tikzpicture}[scale=0.7]

\clip (-4.9,-4.3) rectangle (4.9,5);

\def \L {5};
\def \l {1.90983};

\draw [very thick] ({\L*cos(90)},{\L*sin(90)})--({\l*cos(90+36)},{\l*sin(90+36)})--({\L*cos(90+72)},{\L*sin(90+72)})--({\l*cos(90+36+72)},{\l*sin(90+36+72)})--({\L*cos(90+2*72)},{\L*sin(90+2*72)})--({\l*cos(90+36+2*72)},{\l*sin(90+36+2*72)})--({\L*cos(90+3*72)},{\L*sin(90+3*72)})--({\l*cos(90+36+3*72)},{\l*sin(90+36+3*72)})--({\L*cos(90+4*72)},{\L*sin(90+4*72)})--({\l*cos(90+36+4*72)},{\l*sin(90+36+4*72)})--({\L*cos(90)},{\L*sin(90)});

\draw [thick, gray, domain=108:360, samples=100, <->] plot ({\l*cos(90+36+4*72)+3*cos(\x)/4},{\l*sin(90+36+4*72)+3*sin(\x)/4});

\node at (2.2, 0.8) {{\scalebox{0.8}{$\displaystyle\frac{7\pi}{5}$}}};

\node at (-2, 0.7) {{$\mathcal{P}_2$}};

\end{tikzpicture}

%% file: Polygon3
\begin{tikzpicture}[scale=0.7]

\clip (-4.1,-4.1) rectangle (4.1,5);

\def \L {5};
\def \l {1.90983};

\draw [very thick] (-4,3.5)--(-4,-2)--(-2,-4)--(2,-4)--(4,-2)--(4,3.5)--(3.5,4)--(1,4)--(0.5,3.5)--(0.5,2)--(1.5,2)--(1.5,-1)--(1,-1.5)--(-1,-1.5)--(-1.5,-1)--(-1.5,2)--(-0.5,2)--(-0.5,3.5)--(-1,4)--(-3.5,4)--(-4,3.5);

\draw [thick, gray, domain=-90:180, samples=100, <->] plot ({1.5+3*cos(\x)/4},{2+3*sin(\x)/4});

\node at (2.65, 2) {{\scalebox{0.8}{$\displaystyle\frac{3\pi}{2}$}}};

\node at (-2.7, 0.5) {{$\mathcal{P}_3$}};

\end{tikzpicture}

%% file: Polygon4
\begin{tikzpicture}[scale=0.7]

\clip (-4.1,-4.1) rectangle (4.1,5);

\def \L {5};
\def \l {1.90983};
\def \t {1};

\draw [very thick] ({-\l*cos(90+36+4*72)+2},{\l*sin(90+36+4*72)-\t})--({\l*cos(90+36+4*72)-2},{\l*sin(90+36+4*72)-\t})--({\L*cos(90)-2},{\L*sin(90)-\t})--({\L*cos(90)-4},{\L*sin(90)-\t})--({\L*cos(90)-4},{\L*sin(90+3*72)})--({-\L*cos(90)+4},{\L*sin(90+3*72)})--({-\L*cos(90)+4},{\L*sin(90)-\t})--({-\L*cos(90)+2},{\L*sin(90)-\t})--({-\l*cos(90+36+4*72)+2},{\l*sin(90+36+4*72)-\t});

\draw [thick, gray, domain=108:360, samples=100, <->] plot ({\l*cos(90+36+4*72)-2+3*cos(\x)/4},{\l*sin(90+36+4*72)+3*sin(\x)/4-\t});

\node at (0.2, -0.2) {{\scalebox{0.8}{$\displaystyle\frac{7\pi}{5}$}}};

\node at (-2.5, 0) {{$\mathcal{P}_4$}};

\end{tikzpicture}

%% file: SmoothPacman1
\begin{tikzpicture}[scale=0.7]

\clip (-5,-5) rectangle (5,5);

\def \angleInner {60};
\def \angleOuter {45};
\def \L {5};

\draw [very thick, domain=-\angleOuter:180+\angleOuter, samples=100] plot ({\L*cos(\x)},{\L*sin(\x)});

\draw [very thick] (0,0)--({\L*cos(-\angleOuter)-(\L-\L/(1+sin(\angleInner-\angleOuter)))*cos(-\angleOuter)+(\L-\L/(1+sin(\angleInner-\angleOuter)))*cos(-90-\angleInner)},{\L*sin(-\angleOuter)-(\L-\L/(1+sin(\angleInner-\angleOuter)))*sin(-\angleOuter)+(\L-\L/(1+sin(\angleInner-\angleOuter)))*sin(-90-\angleInner)});

\draw [very thick, domain=-\angleOuter:-90-\angleInner, samples=100] plot ({\L*cos(-\angleOuter)-(\L-\L/(1+sin(\angleInner-\angleOuter)))*cos(-\angleOuter)+(\L-\L/(1+sin(\angleInner-\angleOuter)))*cos(\x)},{\L*sin(-\angleOuter)-(\L-\L/(1+sin(\angleInner-\angleOuter)))*sin(-\angleOuter)+(\L-\L/(1+sin(\angleInner-\angleOuter)))*sin(\x)});

\draw [very thick] (0,0)--({\L*cos(180+\angleOuter)-(\L-\L/(1+sin(\angleInner-\angleOuter)))*cos(180+\angleOuter)+(\L-\L/(1+sin(\angleInner-\angleOuter)))*cos(270+\angleInner)},{\L*sin(180+\angleOuter)-(\L-\L/(1+sin(\angleInner-\angleOuter)))*sin(180+\angleOuter)+(\L-\L/(1+sin(\angleInner-\angleOuter)))*sin(270+\angleInner)});

\draw [very thick, domain=180+\angleOuter:270+\angleInner, samples=100] plot ({\L*cos(180+\angleOuter)-(\L-\L/(1+sin(\angleInner-\angleOuter)))*cos(180+\angleOuter)+(\L-\L/(1+sin(\angleInner-\angleOuter)))*cos(\x)},{\L*sin(180+\angleOuter)-(\L-\L/(1+sin(\angleInner-\angleOuter)))*sin(180+\angleOuter)+(\L-\L/(1+sin(\angleInner-\angleOuter)))*sin(\x)});

\draw [thick, gray, domain=-\angleInner+2:178+\angleInner, samples=100, <->] plot ({3*cos(\x)/4},{3*sin(\x)/4});

\node at (1, 0.8) {\scalebox{0.8}{$\displaystyle\frac{5\pi}3$}};

\end{tikzpicture}

%% file: SmoothPacman2
\begin{tikzpicture}[scale=0.7]

\clip (-5,-5) rectangle (5,5.5);

\def \angleInner {90};
\def \angleOuter {80};
\def \L {5};

\draw [very thick, domain=-\angleOuter:180+\angleOuter, samples=100] plot ({\L*cos(\x)},{\L*sin(\x)});

\draw [very thick] (0,0)--({\L*cos(-\angleOuter)-(\L-\L/(1+sin(\angleInner-\angleOuter)))*cos(-\angleOuter)+(\L-\L/(1+sin(\angleInner-\angleOuter)))*cos(-90-\angleInner)},{\L*sin(-\angleOuter)-(\L-\L/(1+sin(\angleInner-\angleOuter)))*sin(-\angleOuter)+(\L-\L/(1+sin(\angleInner-\angleOuter)))*sin(-90-\angleInner)});

\draw [very thick, domain=-\angleOuter:-90-\angleInner, samples=100] plot ({\L*cos(-\angleOuter)-(\L-\L/(1+sin(\angleInner-\angleOuter)))*cos(-\angleOuter)+(\L-\L/(1+sin(\angleInner-\angleOuter)))*cos(\x)},{\L*sin(-\angleOuter)-(\L-\L/(1+sin(\angleInner-\angleOuter)))*sin(-\angleOuter)+(\L-\L/(1+sin(\angleInner-\angleOuter)))*sin(\x)});

\draw [very thick] (0,0)--({\L*cos(180+\angleOuter)-(\L-\L/(1+sin(\angleInner-\angleOuter)))*cos(180+\angleOuter)+(\L-\L/(1+sin(\angleInner-\angleOuter)))*cos(270+\angleInner)},{\L*sin(180+\angleOuter)-(\L-\L/(1+sin(\angleInner-\angleOuter)))*sin(180+\angleOuter)+(\L-\L/(1+sin(\angleInner-\angleOuter)))*sin(270+\angleInner)});

\draw [very thick, domain=180+\angleOuter:270+\angleInner, samples=100] plot ({\L*cos(180+\angleOuter)-(\L-\L/(1+sin(\angleInner-\angleOuter)))*cos(180+\angleOuter)+(\L-\L/(1+sin(\angleInner-\angleOuter)))*cos(\x)},{\L*sin(180+\angleOuter)-(\L-\L/(1+sin(\angleInner-\angleOuter)))*sin(180+\angleOuter)+(\L-\L/(1+sin(\angleInner-\angleOuter)))*sin(\x)});

\draw [thick, gray, domain=-\angleInner+2:178+\angleInner, samples=100, <->] plot ({3*cos(\x)/4},{3*sin(\x)/4});

\node at (1, 0.8) {\scalebox{0.8}{$2\pi$}};

\end{tikzpicture}

%% file: InstabilityFig1
\begin{tikzpicture}[scale=1.3]

\clip (-7.1,-3.1) rectangle (4.3,3.1);

\def \r  {1};
\def \rr  {0.2};
\def \R  {0.5};
\def \angle {77};
\def \Lshift {-3};
\def \Rshift {3};


\draw [very thick] plot [smooth cycle, tension =0.85, cm={1,0,0,1, (\Lshift,0)}] coordinates {({-4},0) ({-2},2) ({-1},3) ({-0.5},2.3) (0.8, 2.2) ({1}, 1) (-0.3, 1.1) ({-0.5},0) ({1},0) ({0},-3) ({-2}, -1)};

\draw [gray, dashed, thick, cm={1,0,0,1, (\Lshift,0)}] (0.2,1.1)--({0.2+\r*cos(-140)}, {1.1+\r*sin(-140)});
\draw [red, fill, cm={1,0,0,1, (\Lshift,0)}] (0.2,1.1) circle (1.3pt);
\draw [red, thick, dashed, cm={1,0,0,1, (\Lshift,0)}] (0.2,1.1) circle (\r);

\node at (0.35+\Lshift,0.85)  {$x_0$};
\node at (-0.2+\Lshift,0.55)  {$r$};
\node at (-2+\Lshift,1) {$\Omega$};


\draw [very thick] plot [smooth cycle, tension =0.85, cm={1,0,0,1, (\Rshift, 0)}] coordinates {({-4},0) ({-2},2) ({-1},3) ({-0.5},2.3) (0.8, 2.2) ({1}, 1) (-0.3, 1.1) ({-0.5},0) ({1},0) ({0},-3) ({-2}, -1)};

\draw [white, fill, cm={1,0,0,1, (\Rshift, 0)}] (0.2,1.1) circle (\rr);
\draw [very thick, cm={1,0,0,1, (\Rshift, 0)}] (0,1.152)--({0.2+\R*cos(\angle)/2},{1.1+\R*sin(\angle)/2})--(0.4,1.038);
\draw [fill, cm={1,0,0,1, (\Rshift, 0)}] (0,1.152) circle (0.45pt);
\draw [fill, cm={1,0,0,1, (\Rshift, 0)}] (0.4,1.038) circle (0.45pt);

\draw [gray, thick, dashed, cm={1,0,0,1, (\Rshift, 0)}] (0.2,1.1)--({0.2+\R*cos(-140)}, {1.1+\R*sin(-140)});
\draw [green, thick, dashed, cm={1,0,0,1, (\Rshift, 0)}]  ({0.2+\R*cos(\angle)/4+(\R-0.03)*cos(\angle +90)},{1.1+\R*sin(\angle)/4+(\R-0.03)*sin(\angle +90)})--({0.2+\R*cos(\angle)/4-(\R-0.03)*cos(\angle +90)},{1.1+\R*sin(\angle)/4-(\R-0.03)*sin(\angle +90)});
\draw [green, thick, dashed, cm={1,0,0,1, (\Rshift, 0)}]  ({0.2-\R*cos(\angle)/4+(\R-0.03)*cos(\angle +90)},{1.1-\R*sin(\angle)/4+(\R-0.03)*sin(\angle +90)})--({0.2-\R*cos(\angle)/4-(\R-0.03)*cos(\angle +90)},{1.1-\R*sin(\angle)/4-(\R-0.03)*sin(\angle +90)});
\draw [orange, dashed, thick, cm={1,0,0,1, (\Rshift, 0)}] (0.2,1.1) circle (\R);

\draw [red, fill, cm={1,0,0,1, (\Rshift, 0)}] (0.2,1.1) circle (1.3pt);

\draw [red, thick, dashed, cm={1,0,0,1, (\Rshift, 0)}] (0.2,1.1) circle (\r);

\node at (-2+\Rshift,1) {$\Omega'$};
\node at (0.35+\Rshift,0.85)  {$x_0$};
\node at (-0.1+\Rshift,1)  {$r'$};

\end{tikzpicture}

%% file: HardyMorrey.bbl
\begin{thebibliography}{42}




\bibitem{MR2401600}
L.~Ambrosio, N.~Gigli, and G.~Savar\'{e}.
\newblock {\em Gradient flows in metric spaces and in the space of probability
  measures}.
\newblock Lectures in Mathematics ETH Z\"{u}rich. Birkh\"{a}user Verlag, Basel,
  second edition, 2008.

\bibitem{Anc81}
A.~Ancona.
\newblock Une propri\'{e}t\'{e} des espaces de {S}obolev.
\newblock {\em C. R. Acad. Sci. Paris S\'{e}r. I Math.}, 292(9):477--480, 1981.

\bibitem{MR3408787}
A.~A. Balinsky, W.~D. Evans, and R.~T. Lewis.
\newblock {\em {The Analysis and Geometry of Hardy's Inequality}}.
\newblock Universitext. Springer, Cham, 2015.

\bibitem{Boggio}
T.~Boggio.
\newblock Sull'equazione del moto vibratorio delle membrane elastiche.
\newblock {\em Rend. Mat. Acc. Lincei, Ser.~5}, 16(2):386--393, 1907.

\bibitem{BrascoPrinariZagati_23}
L.~Brasco, F.~Prinari, and A.~C. Zagati.
\newblock Sobolev embeddings and distance functions.
\newblock {\em Adv. Calc. Var.}, 17(4):1365--1398, 2024.

\bibitem{MR2759829}
H.~Brezis.
\newblock {\em Functional analysis, {S}obolev spaces and partial differential
  equations}.
\newblock Universitext. Springer, New York, 2011.

\bibitem{MR1655516}
H.~Brezis and M.~Marcus.
\newblock Hardy's inequalities revisited.
\newblock {\em Ann. Scuola Norm. Sup. Pisa Cl. Sci. (4)}, 25(1-2):217--237,
  1997.

\bibitem{MR0709644}
H.~Brezis and L.~Nirenberg.
\newblock Positive solutions of nonlinear elliptic equations involving critical
  {S}obolev exponents.
\newblock {\em Comm. Pure Appl. Math.}, 36(4):437--477, 1983.

\bibitem{MR2196033}
J.~Chabrowski and M.~Willem.
\newblock Hardy's inequality on exterior domains.
\newblock {\em Proc. Amer. Math. Soc.}, 134(4):1019--1022, 2006.

\bibitem{MR3902383}
H.~Cheikh~Ali.
\newblock Hardy--{S}obolev inequalities with singularities on non smooth
  boundary: {H}ardy constant and extremals. {P}art {I}: {I}nfluence of local
  geometry.
\newblock {\em Nonlinear Anal.}, 182:316--349, 2019.

\bibitem{MR4150374}
H.~Cheikh~Ali.
\newblock Hardy--{S}obolev inequalities with singularities on non smooth
  boundary. {P}art 2: {I}nfluence of the global geometry in small dimensions.
\newblock {\em J. Differential Equations}, 270:185--216, 2021.

\bibitem{MR1868901}
F.~Colin.
\newblock Hardy's inequality in unbounded domains.
\newblock {\em Topol. Methods Nonlinear Anal.}, 17(2):277--284, 2001.

\bibitem{MR1034662}
H.~Egnell.
\newblock Positive solutions of semilinear equations in cones.
\newblock {\em Trans. Amer. Math. Soc.}, 330(1):191--201, 1992.

\bibitem{MR3409135}
L.~C. Evans and R.~F. Gariepy.
\newblock {\em Measure theory and fine properties of functions}.
\newblock Textbooks in Mathematics. CRC Press, Boca Raton, FL, revised edition,
  2015.

\bibitem{MR3527624}
S.~Filippas and G.~Psaradakis.
\newblock The {H}ardy--{M}orrey \& {H}ardy--{J}ohn--{N}irenberg inequalities
  involving distance to the boundary.
\newblock {\em J. Differential Equations}, 261(6):3107--3136, 2016.

\bibitem{MR2097030}
N.~Ghoussoub and X.~S. Kang.
\newblock Hardy--{S}obolev critical elliptic equations with boundary
  singularities.
\newblock {\em Ann. Inst. H. Poincar\'{e} C Anal. Non Lin\'{e}aire},
  21(6):767--793, 2004.

\bibitem{MR4599209}
N.~Ghoussoub, S.~Mazumdar, and F.~Robert.
\newblock Multiplicity and stability of the {P}ohozaev obstruction for
  {H}ardy--{S}chr\"{o}dinger equations with boundary singularity.
\newblock {\em Mem. Amer. Math. Soc.}, 285(1415):v+126, 2023.

\bibitem{MR3052352}
N.~Ghoussoub and A.~Moradifam.
\newblock {\em Functional inequalities: new perspectives and new applications},
  volume 187 of {\em Mathematical Surveys and Monographs}.
\newblock American Mathematical Society, Providence, RI, 2013.

\bibitem{MR2276538}
N.~Ghoussoub and F.~Robert.
\newblock The effect of curvature on the best constant in the
  {H}ardy--{S}obolev inequalities.
\newblock {\em Geom. Funct. Anal.}, 16(6):1201--1245, 2006.

\bibitem{MR3472850}
N.~Ghoussoub and F.~Robert.
\newblock Sobolev inequalities for the {H}ardy--{S}chr\"{o}dinger operator:
  extremals and critical dimensions.
\newblock {\em Bull. Math. Sci.}, 6(1):89--144, 2016.

\bibitem{hardy}
G.~H. Hardy.
\newblock Note on a theorem of {H}ilbert.
\newblock {\em Math. Z.}, 6(3-4):314--317, 1920.

\bibitem{hardy2}
G.~H. Hardy.
\newblock {Notes on some points in the integral calculus, LX: An inequality
  between integrals}.
\newblock {\em Messenger of Mathematics}, 54:150--156, 1922.

\bibitem{HyndLarsonLindgren-Decay}
R.~Hynd, S.~Larson, and E.~Lindgren.
\newblock Decay of extremals of {M}orrey's inequality.
\newblock {\em Ark. Mat.}, 62(1):73--81, 2024.

\bibitem{zbMATH07126544}
R.~Hynd and E.~Lindgren.
\newblock Extremal functions for {Morrey}'s inequality in convex domains.
\newblock {\em Math. Ann.}, 375(3-4):1721--1743, 2019.

\bibitem{MR4160015}
R.~Hynd and F.~Seuffert.
\newblock On the symmetry and monotonicity of {M}orrey extremals.
\newblock {\em Commun. Pure Appl. Anal.}, 19(11):5285--5303, 2020.

\bibitem{MR4275749}
R.~Hynd and F.~Seuffert.
\newblock Extremal functions for {M}orrey's inequality.
\newblock {\em Arch. Ration. Mech. Anal.}, 241(2):903--945, 2021.

\bibitem{KMP}
A.~Kufner, L.~Maligranda, and L.-E. Persson.
\newblock {\em {T}he {H}ardy {I}nequality: About its {H}istory and {S}ome
  {R}elated {R}esults}.
\newblock Vydavatelsk\'{y} Servis, Plze\v{n}, 2007.

\bibitem{MR4012806}
P.~D. Lamberti and Y.~Pinchover.
\newblock {$L^p$} {H}ardy inequality on {$C^{1,\gamma}$} domains.
\newblock {\em Ann. Sc. Norm. Super. Pisa Cl. Sci. (5)}, 19(3):1135--1159,
  2019.

\bibitem{Lewis}
J.~L. Lewis.
\newblock Uniformly fat sets.
\newblock {\em Trans. Amer. Math. Soc.}, 308(1):177--196, 1988.

\bibitem{MR2885951}
R.~T. Lewis, J.~Li, and Y.~Li.
\newblock A geometric characterization of a sharp {H}ardy inequality.
\newblock {\em J. Funct. Anal.}, 262(7):3159--3185, 2012.

\bibitem{MR1458330}
M.~Marcus, V.~J. Mizel, and Y.~Pinchover.
\newblock On the best constant for {H}ardy's inequality in {$\mathbf{R}^n$}.
\newblock {\em Trans. Amer. Math. Soc.}, 350(8):3237--3255, 1998.

\bibitem{MR1817710}
M.~Marcus and I.~Shafrir.
\newblock An eigenvalue problem related to {H}ardy's {$L^p$} inequality.
\newblock {\em Ann. Scuola Norm. Sup. Pisa Cl. Sci. (4)}, 29(3):581--604, 2000.

\bibitem{Mazya}
V.~G. Maz'ja.
\newblock {\em Sobolev spaces}.
\newblock Springer Series in Soviet Mathematics. Springer-Verlag, Berlin, 1985.
\newblock Translated from the Russian by T. O. Shaposhnikova.

\bibitem{Kaj}
K.~Nystr\"om.
\newblock {Smoothness properties of solutions to Dirichlet problems in domains
  with a fractal boundary}.
\newblock {\em Doctoral thesis, Ume{\aa} University, Department of
  Mathematics}, 1994.

\bibitem{OK90}
B.~Opic and A.~Kufner.
\newblock {\em Hardy-type inequalities}, volume 219 of {\em Pitman Research
  Notes in Mathematics Series}.
\newblock Longman Scientific \& Technical, Harlow, 1990.

\bibitem{MR2984139}
G.~Psaradakis.
\newblock {\em An optimal {H}ardy--{M}orrey inequality}.
\newblock {\em Calc. Var. Partial Differential Equations}, 45(3-4):421--441, 2012.

\bibitem{psaradakis2013hardy}
G.~Psaradakis.
\newblock {\em Hardy Inequalities in General Domains}.
\newblock {\em Doctoral thesis, University of Crete, Department of Applied Mathematics}, 2011.
\newblock {(arXiv: 1302.3961)}

\bibitem{Str84}
E.~W. Stredulinsky.
\newblock {\em Weighted inequalities and degenerate elliptic partial
  differential equations}, volume 1074 of {\em Lecture Notes in Mathematics}.
\newblock Springer-Verlag, Berlin, 1984.

\bibitem{SunWang24}
L.~Sung and L.~Wang.
\newblock {Attainability of the best constant of Hardy--Sobolev inequality with full boundary singularities}.
\newblock {Preprint 2024, arXiv: 2405.09795.}

\bibitem{MR4169674}
N.~Ustinov.
\newblock The effect of curvature in fractional {H}ardy--{S}obolev inequality
  involving the spectral {D}irichlet {L}aplacian.
\newblock {\em Trans. Amer. Math. Soc.}, 373(11):7785--7815, 2020.

\bibitem{WangZhu24}
L.~Wang and M.~Zhu.
\newblock Hardy--{S}obolev inequalities with distance to the boundary weight functions.
\newblock {\em J. Math. Anal. Appl.}, 530(2): 127716, 2024.

\bibitem{MR1010807}
A.~Wannebo.
\newblock Hardy inequalities.
\newblock {\em Proc. Amer. Math. Soc.}, 109(1):85--95, 1990.

\end{thebibliography}
